\documentclass[11pt,leqno,a4paper]{amsart}

\usepackage[margin=25mm]{geometry}

\usepackage[title]{appendix}
\usepackage{amssymb}
\usepackage{amsmath}
\usepackage{amsthm}
\usepackage{tikz}
\usepackage{tikz-cd}
\usepackage{mathtools}
\usepackage{framed}
\usepackage{comment}
\definecolor{darkblue}{rgb}{0,0,0.6}
\usepackage[breaklinks, ocgcolorlinks,colorlinks=true, citecolor=darkblue, filecolor=darkblue, linkcolor=darkblue, urlcolor=darkblue]{hyperref}
\usepackage[capitalize]{cleveref}
\usepackage{stmaryrd}

\newcommand{\C}{\mathbb{C}}

\newcommand{\Q}{\mathbb{Q}}
\newcommand{\R}{\mathbb{R}}

\newcommand{\Z}{\mathbb{Z}}
\renewcommand{\H}{\mathbb{H}}
\newcommand{\F}{\mathbb{F}}

\renewcommand{\mod}{\,\,\text{\normalfont mod}\,}
\renewcommand{\l}{\Lambda}
\newcommand{\wh}{\widehat}
\newcommand{\wt}{\widetilde}
\newcommand{\ol}{\overline}
\newcommand{\PP}{\mathcal{P}}
\newcommand{\hh}{\string^}

\DeclareMathOperator{\Aut}{Aut}

\DeclareMathOperator{\Hom}{Hom}
\DeclareMathOperator{\id}{id}

\DeclareMathOperator{\Tor}{Tor}
\DeclareMathOperator{\Ker}{Ker}
\DeclareMathOperator{\coker}{coker}
\DeclareMathOperator{\Wh}{Wh}

\DeclareMathOperator{\GL}{GL}

\DeclareMathOperator{\Out}{Out}

\DeclareMathOperator{\IM}{Im}

\DeclareMathOperator{\SF}{SF}
\DeclareMathOperator{\HT}{HT}
\DeclareMathOperator{\Def}{def}

\DeclareMathOperator{\ab}{ab}
\DeclareMathOperator{\gp}{gp}
\DeclareMathOperator{\AC}{AC}
\DeclareMathOperator{\std}{std}

\DeclareMathOperator{\Proj}{Proj}

\newtheorem{thm}{Theorem}[section]
\newtheorem{theorem}[thm]{Theorem}
\newtheorem*{thm*}{Theorem}
\newtheorem{prop}[thm]{Proposition}
\newtheorem*{prop*}{Proposition}
\newtheorem{lemma}[thm]{Lemma}
\newtheorem{corollary}[thm]{Corollary}
\newtheorem*{corollary*}{Corollary}
\newtheorem{conj}[thm]{Conjecture}

\newtheorem{question}[thm]{Question}
\newtheorem*{question*}{Question}

\newtheorem*{problem*}{Problem}

\newtheorem*{Dproblem}{D2 problem}

\newtheorem{thmx}{Theorem}

\theoremstyle{definition}

\newtheorem{example}[thm]{Example}
\newtheorem*{theorem*}{Theorem}

\theoremstyle{remark}
\newtheorem{remark}[thm]{Remark}
\newtheorem*{remark*}{Remark}

\usepackage{rotating}

\usepackage{color}
\usepackage{enumerate}

\usepackage{cite}
\usepackage[nodisplayskipstretch]{setspace}
\usepackage{placeins}

\newenvironment{clist}[1]
{\begin{enumerate}[\normalfont #1]}
{\end{enumerate}}

\usepackage{mathrsfs}
\usepackage{adjustbox}

\makeatletter
\@namedef{subjclassname@2020}{%
  \textup{2020} Mathematics Subject Classification}
\makeatother

\usetikzlibrary{decorations.markings}

\usepackage{letltxmacro}

\LetLtxMacro\Oldfootnote\footnote

\newcommand\dhxrightarrow[2][]{%
  \mathrel{\ooalign{$\xrightarrow[#1\mkern4mu]{#2\mkern4mu}$\cr%
  \hidewidth$\rightarrow\mkern4mu$}}
}

\begin{document}

\title[Exotic presentations of quaternion groups and Wall's D2 problem]{Exotic presentations of quaternion groups and \\ Wall's D2 problem}

\author{Tommy Hofmann} 
\address{Naturwissenschaftlich-Technische Fakult\"{a}t, Universit\"{a}t Siegen,  Walter-Flex-Strasse \newline \indent 3, 57068 Siegen, Germany}
\email{tommy.hofmann@uni-siegen.de}

\author{John Nicholson}
\address{School of Mathematics and Statistics, University of Glasgow, United Kingdom}
\email{john.nicholson@glasgow.ac.uk}

%\keywords{ }

%57K20 2-dimensional topology (including mapping class groups of surfaces, Teichm¨uller theory, curve complexes, etc.)
%20C05 Group rings of finite groups and their modules (group-theoretic aspects)
%57M05 Fundamental group, presentations, free differential calculus
%57Q12 Wall finiteness obstruction for CW-complexes

\begin{abstract}
The D2 problem of C. T. C. Wall asks whether every finite cohomologically 2-dimensional CW-complex is homotopy equivalent to a finite 2-complex.
Several potential counterexamples have been proposed, the longest standing of which is a CW-complex constructed by Cohen and Dyer whose fundamental group is a quaternion group of order $32$.
We show that this CW-complex is homotopy equivalent to the presentation 2-complex of a presentation constructed by Mannan--Popiel, thus showing it is not a counterexample to the D2 problem.
We next introduce an infinite family of presentations for a quaternion group of order $4n$ and prove that they achieve homotopy types which are not achieved by the presentations of Mannan--Popiel.
\end{abstract}

\maketitle

\vspace{-7mm}

\section{Introduction}

Given a finitely presented group, can we classify its collection of finite presentations? To make this precise, we must first specify an equivalence relation on the class of finite presentations.
A finite presentation $\PP$ for a group $G$ has an associated finite 2-complex $X_{\PP}$ with fundamental group $G$, and we say that two finite presentations are \textit{homotopy equivalent} if their associated 2-complexes are.
The homotopy classification of finite $2$-complexes, and hence finite presentations, is inextricably linked to the D2 problem.
We say a CW-complex $X$ is a \textit{D2 complex} if it satisfies Wall's D2 finiteness condition \cite[p62]{Wa65}: $H_i(\wt X) = 0$ for $i > 2$ and $H^3(X;M)=0$ for all finitely generated $\Z[\pi_1(X)]$-modules $M$.
The following can be found in \cite[Problem D3]{Wa79}.

\begin{Dproblem}
Is every finite {\normalfont D2} complex homotopy equivalent to a finite $2$-complex?
\end{Dproblem}

We say a finitely presented group $G$ has the \textit{D2 property} if every finite D2 complex $X$ with $\pi_1(X) \cong G$ is homotopy equivalent to a finite 2-complex.
Note that the D2 problem is closely related to the Eilenberg--Ganea conjecture and the Relation Gap problem (see \cite{BT07,Ha18}).

There is a one-to-one correspondence between the homotopy types of finite D2 complexes $X$ with $\pi_1(X) \cong G$ and the chain homotopy types of chain complexes of finitely generated free $\Z G$-modules $\mathcal{F} = (F_2 \to F_1 \to F_0)$ such that $H_1(\mathcal{F}) = 0$ and $H_0(\mathcal{F}) \cong \Z$ (see \cite[Theorem 2.1]{Ni21b}). 
Thus, if $G$ has the D2 property, then the homotopy classification of finite presentations for $G$ reduces to the chain homotopy classification of chain complexes $\mathcal{F}$ of the above form. This is often approachable using module-theoretic methods \cite{Jo03b, Jo12, Ni20b}.
The D2 problem is currently open, and the list of groups whose finite presentations have been classified up to homotopy currently coincides with the groups known to have the D2 property.

We say a pair of finite presentations for a group $G$ are \textit{exotic} if they have the same deficiency but are not homotopy equivalent.
Exotic presentations were first constructed for certain finite abelian groups \cite{Me76} and the trefoil group \cite{Du76}. This led to much activity in the intervening years whereby finite abelian groups were shown to have the D2 property \cite{Br79b}, and the quaternion group $Q_{32}$ was independently proposed as a counterexample by Cohen \cite{Co77} and Dyer \cite{Dy78}.
Quaternion groups $Q_{4n}$ of order $4n$ were later studied and, for $2 \le n \le 5$, it was shown that $Q_{4n}$ admits no exotic presentations and satisfies the D2 property \cite{Jo03b,Jo04}.

Our main new contribution is the introduction of the following family of presentations
\[ \mathcal{P}_n(n_1,\dotsc,n_{k};m_1,\dotsc,m_{k}) := \langle x, y \mid x^ny^{-2}, x^{n_1}yx^{m_1}y^{-1} \cdots x^{n_{k+1}}yx^{m_{k+1}}y^{-1} \rangle \]
for $n \ge 2$, $k \ge 0$, $n_i, m_i \in \Z$ with $n_{k+1} = 1 - \sum_{i=1}^{k} n_i$ and $m_{k+1} = 1 - \sum_{i=1}^{k} m_i$. We will show:
\begin{clist}{(i)}
\item 
Every presentation for $Q_{4n}$ on the standard generating set of the form $\langle x, y \mid x^ny^{-2}, R \rangle$ is homotopy equivalent to a presentation of the above form (\cref{thm:general-presentations-Q4n}).
\item 
If $n_1, \dotsc, n_k$ are all even, then $\mathcal{P}_n(n_1,\dotsc,n_k;-n_k,\dotsc,-n_1)$ presents $Q_{4n}$ (\cref{thm:even-coeffs}). 
\end{clist}

The case $k=0$ gives the standard presentation $\mathcal{P}_n^{\std} = \langle x,y \mid x^ny^{-2}, xyxy^{-1} \rangle$ for $Q_{4n}$.
Determining the $n_i,m_i$ for which $\mathcal{P}_n(n_1,\dotsc,n_{k};m_1,\dotsc,m_{k})$ presents $Q_{4n}$ appears to be difficult.
For example, $\mathcal{P}_3(2;2)$ does not present $Q_{12}$ (\cref{prop:MP-not-presentation}) whilst $\mathcal{P}_{10}(3;5)$ does present $Q_{40}$.

\subsection{Main results}

Our main result on the D2 property for quaternion groups is as follows.
Recall that, if $G$ is a finite group, then all exotic presentations have maximal deficiency \cite{Br79a}. Thus we restrict to presentations for $Q_{4n}$ of deficiency zero, i.e. balanced presentations.

\begin{thmx} \label{thmx:main-D2}
If $6 \le n \le 8$, then $Q_{4n}$ has the {\normalfont D2} property and every balanced presentation is homotopy equivalent to precisely one of $\mathcal{P}_n^{\std}$ and $\mathcal{P}_n(2;-2)$. Furthermore:
\begin{clist}{(i)}
\item $Q_{36}$ has the {\normalfont D2} property if and only if there exists a balanced presentation $\mathcal{P}$ for $Q_{36}$ such that $\mathcal{P} \not \simeq \mathcal{P}_9^{\std}, \mathcal{P}_9(2;-2)$.
\item $Q_{40}$ has the {\normalfont D2} property if and only if there exist homotopically distinct balanced presentations $\mathcal{P}_1$, $\mathcal{P}_2$ for $Q_{40}$ such that $\mathcal{P}_i \not \simeq \mathcal{P}_{10}^{\std}, \mathcal{P}_{10}(2;-2), \mathcal{P}_{10}(3;5)$ for $i=1,2$.
\end{clist}
\end{thmx}

The case $n=7$ was established previously in \cite{MP21, Ni21b} and the family of presentations introduced above generalises the family $\mathcal{E}_{n,r} = \langle x,y \mid x^ny^{-2}, x^ryx^2yx^{1-r}y^{-1}x^{-1}y^{-1}\rangle$ introduced by Mannan--Popiel \cite{MP21} which is $\mathcal{P}_n(2;1-r)$ in our notation. They showed that $\mathcal{P}_n(2;-2)$ presents $Q_{4n}$ for all $n \ge 2$, and that $\mathcal{P}_7^{\std}$, $\mathcal{P}_7(2;-2)$ are exotic presentations for $Q_{28}$.
The case $n=8$ implies that the CW-complex of Cohen and Dyer is homotopy equivalent to $X_{\mathcal{P}_8(2;-2)}$.

We now give a more detailed exploration of these presentations for quaternion groups of larger order, and their significance for the D2 problem.

Firstly, in \cite[Section 1.2]{Ni20b} we pointed out that the Mannan--Popiel presentations are such that $r \equiv s \mod n$ implies $\mathcal{E}_{n,r} \simeq \mathcal{E}_{n,s}$ and so can represent at most $n$ homotopy types of balanced presentations for $Q_{4n}$.
Since the number of corresponding chain homotopy types grows super-exponentially in $n$ (see \cite[Theorem B]{Ni20b}), in order for $Q_{4n}$ to have the D2 property, a presentation must exist which is not homotopy equivalent to one of this form.
We will show:

\begin{thmx} \label{thmx:first-example}
$\mathcal{P}_{10}(3;5) = \langle x, y \mid x^{10}y^{-2}, x^3yx^5y^{-1}x^{-2}yx^{-4}y^{-1} \rangle$ is not homotopy equivalent to any presentation for $Q_{40}$ of the form $\mathcal{P}_{10}(2;r)$ for $r \in \Z$.
\end{thmx}

We show that, if $n_i \equiv n_i' \mod n$ and $m_i \equiv m_i' \mod n$, then $\mathcal{P}_n(n_1,\dotsc,n_{k};m_1,\dotsc,m_{k})$ and $\mathcal{P}_n(n_1',\dotsc,n_{k}';m_1',\dotsc,m_{k}')$ are homotopy equivalent (\cref{lemma:operations}). 
Thus, for fixed $k$, there are at most $n^{2k}$ homotopy types of presentations. This is sub-exponential in $n$ and so, in order to use these presentations to show that $Q_{4n}$ has the D2 property, there must exist presentations for $k$ arbitrarily large which do not reduce down to the case $k-1$ up to homotopy.
We will show:

\begin{thmx} \label{thmx:a-non-MP-presentation}
$\mathcal{P}_{12}(2,2;-2,-2) = \langle x, y \mid x^{12}y^{-2}, (x^{2}yx^{-2}y^{-1})^2x^{-3}yx^{5}y^{-1} \rangle$ is not homotopy equivalent to any presentation for $Q_{48}$ of the form $\mathcal{P}_{12}(n_1;m_1)$ for $n_1,m_1 \in \Z$.
\end{thmx}

This gives some evidence to support the possibility that the missing presentations for $Q_{36}$ and $Q_{40}$ in \cref{thmx:main-D2} are of the form $\mathcal{P}_n(n_1,\dotsc,n_k;m_1,\dotsc,m_k)$ for some $k \ge 2$, and thus that these groups have the D2 property.
The following seems plausible (cf. \cite[Conjecture 6.1]{MP21}):

\begin{conj}
If $n \ge 2$, then $Q_{4n}$ has the {\normalfont D2} property and every balanced presentation is homotopy equivalent to a presentation of the form $\mathcal{P}_n(n_1,\dotsc,n_k;m_1,\dotsc,m_k)$ for some $n_i,m_i,k$.
\end{conj}

By \cite[Theorem A]{Ni20b}, if $n \ge 6$ and $Q_{4n}$ has the D2 property, then $Q_{4n}$ has an exotic presentation. We conjecture that this holds unconditionally, and we will show:

\begin{thmx} \label{thmx:main-Q4n}
If $n = m k$ where $6 \le m \le 12$ and $k \ge 1$ is odd, then $Q_{4n}$ has an exotic presentation. More specifically, $\mathcal{P}_{n}(2;-2) \not \simeq \mathcal{P}_n^{\std}$.
\end{thmx}

See \cref{s:pres-small-order} for more detailed calculations. For example, we show that $\mathcal{P}_{12}^{\std}$, $\mathcal{P}_{12}(2;-2)$, $\mathcal{P}_{12}(4;-4)$ and $\mathcal{P}_{12}(2,2;-2,-2)$ are homotopically distinct (\cref{prop:Q48-main}).

\subsection{Methods and computational results}

We will now give an overview of the ideas behind the proofs of Theorems \ref{thmx:main-D2}, \ref{thmx:first-example}, \ref{thmx:a-non-MP-presentation} and \ref{thmx:main-Q4n}.

\subsubsection*{General strategy and the key new idea.}
To determine whether $Q_{4n}$ has the D2 property, we need a method for classifying chain complexes $\mathcal{F}$ of the required form up to chain homotopy equivalence. 
In previous work, we showed that there is a one-to-one correspondence $\Psi$ between such chain complexes and finitely generated stably free $\Z Q_{4n}$-modules, considered up to isomorphism modulo a certain action of $\Aut(Q_{4n})$ \cite[Theorem B]{Ni20a} (see also \cite{Jo03b,Ni21b}).

For each presentation $\mathcal{P} = \mathcal{P}_n(n_1,\dotsc,n_k;m_1,\dotsc,m_k)$, we construct an ideal $I_{\mathcal{P}} \le \Z Q_{4n}$ which is stably free and represents $\Psi(C_*(X_{\mathcal{P}}))$ (\cref{thm:Psi(P_R)}).
This reduces the question of whether two presentations $\mathcal{P}_1$, $\mathcal{P}_2$ are homotopy equivalent to the question of whether the corresponding ideals $I_{\mathcal{P}_1}$, $I_{\mathcal{P}_2}$ are isomorphic modulo the action of $\Aut(Q_{4n})$. This is in contrast to the approach of Mannan--Popiel \cite{MP21} whereby the $\Z Q_{4n}$-modules $\pi_2(X_{\mathcal{P}_1})$ and $\pi_2(X_{\mathcal{P}_2})$ are compared, and which is more difficult to determine since these modules are not projective.

Next, we need methods for determining whether two stably free $\Z Q_{4n}$-modules are isomorphic. Since these modules are projective, we can use the theory of Milnor squares to tackle this question using decompositions for the ring $\Z Q_{4n}$ in terms of certain pullback squares. In the case $Q_{4p}$ for $p$ prime (i.e. $Q_{28}$, $Q_{44}$), useful pullback square decompositions for the ring $\Z Q_{4p}$ exist which allow us to classify projective $\Z Q_{4p}$-modules by computing double coset quotients of units in the various rings (\cref{ss:Q28}). This applies to the cases $Q_{8p}$ for $p$ prime (i.e. $Q_{24}$, $Q_{40}$) and $Q_{4p^2}$ for $p$ prime (i.e. $Q_{36}$), albeit with more involved computations.
In contrast, the ring $\Z Q_{32}$ does not admit such a nice pullback square decomposition and so we instead consider the extension of scalars of the ideals $I_{\mathcal{P}}$ to an explicit maximal order $\Gamma \le \Q Q_{32}$ containing $\Z Q_{32}$. We then apply a freeness criterion established by Bley--Hofmann--Johnston (\cref{prop:bhjcrit}).

\subsubsection*{Strategy for constructing presentations.}
Lengthy computations are required in order to prove Theorems \ref{thmx:first-example} and \ref{thmx:a-non-MP-presentation}. For example, in \cref{thmx:a-non-MP-presentation}, we must show that $\mathcal{P}_{12}(2,2;-2,-2)$ is not homotopy equivalent to one of $13$ presentations of the form $\mathcal{P}_{12}(n_1;m_1)$ (\cref{lemma:height1}).
Whilst these computations have been completed (\cref{s:pres-small-order}), identifying presentations which actually have these properties proved a significant challenge. For example, before considering $\mathcal{P}_{12}(2,2;-2,-2)$, it makes sense to consider $\mathcal{P}_n(n_1,n_2;m_1,m_2)$ for $n \in \{9,10,11\}$ and various $n_i,m_i \in \Z$.

In order to help with finding presentations with these properties, we used OSCAR~\cite{OSCAR-book} to decide whether two given stably free $\Z Q_{4n}$-modules are isomorphic with a high probability. 
Algorithms for determining whether stably free $\Z Q_{4n}$-modules are isomorphic exist due to work of Bley--Johnston (see \cite{BJ08, BJ11, HJ20}), but for large group orders ($4n > 24$) these algorithms are too slow.
Instead we use a heuristic variant of these algorithms, which is much faster and predicts whether the modules are isomorphic, though has the drawback that the output of the algorithm is not a complete proof and so the results need to be recovered by hand.
Note that for $Q_{24}$, the Bley--Johnston algorithm can be used to verify our results.

\subsubsection*{Further results of heuristic computations.}
The heuristic computations above can also be used to analyse \cref{thmx:main-D2} (i) and (ii). These computations suggest the following:

\begin{conj}
If $\mathcal{P} = \mathcal{P}_9(n_1,n_2;m_1,m_2)$ presents $Q_{36}$, then $\mathcal{P} \simeq \mathcal{P}_9^{\std}$ or $\mathcal{P}_9(2;-2)$.
\end{conj}

In the case of $Q_{40}$, the full computation for all presentations of the form $\mathcal{P}_{10}(n_1,n_2;m_1,m_2)$ was too lengthy to complete but randomly selected presentations were all heuristically homotopy equivalent to the three presentations listed in \cref{thmx:main-D2} (ii).

This suggests that either the missing presentation for $Q_{36}$ does not exist (implying that $Q_{36}$ does not have the D2 property) or a presentation exists which is either of the form above for $k \ge 3$ or which is not homotopy equivalent to any presentation of the form $\langle x, y \mid x^{9}y^{-2}, R\rangle$.

\subsection{Further equivalence relations}

Besides homotopy equivalence, there are several other natural equivalence relations which can be placed on the class of finite presentations.
We say that two finite presentations are \textit{simple homotopy equivalent} ($\simeq_s$) if their associated $2$-complexes are simple homotopy equivalent in the sense of Whitehead (see \cite{Co73}). Two finite presentations are \textit{generalised Andrews--Curtis equivalent} ($\simeq_{\AC}$), also known as $Q^{**}$-equivalent \cite{HM93}, if they are related by a sequence of operations of types (i)-(vi) which alter the generators and relators (see \cite{HM93,Ba25}).
These equivalence relations are related by $\simeq_{\AC}$ $\Rightarrow$ $\simeq_s$ $\Rightarrow$ $\simeq$. 

For finite presentations, it is currently open whether $\simeq_s$ $\Rightarrow$ $\simeq_{\AC}$ in general, and this is known as the generalised Andrews--Curtis conjecture. 
In the case of the trivial group, this coincides with the usual Andrews--Curtis conjecture which asks whether all balanced presentations for the trivial group are Andrews--Curtis equivalent to the trivial presentation $\langle\,\cdot\, | \,\cdot \, \rangle$ \cite{HM93}. 
Metzler \cite{Me90} and Lustig \cite{Lu91} independently constructed finite presentations which are homotopy but not simple homotopy equivalent. Their examples both presented infinite groups.

The family of presentations $\mathcal{P}_n(n_1,\dotsc,n_k;m_1,\dotsc,m_k)$ present an interesting test case for simple homotopy equivalence and generalised Andrews--Curtis equivalence.
For each $n \ge 2$, there are infinitely many presentations for $Q_{4n}$ of this form for $|n_i|,|m_i| < n$ and $k \ge 0$ (\cref{thm:even-coeffs}). On the other hand, since $Q_{4n}$ is finite, there are only finitely many balanced presentations up to homotopy equivalence (this follows from the Jordan--Zassenhaus theorem \cite{Sw70}). In fact, over $Q_{4n}$ for $2 \le n \le 5$, all these presentations are homotopy equivalent.

The difference between homotopy and simple homotopy is measured by an invariant lying in the Whitehead group $\Wh(Q_{4n})$. By \cite{Ol88}, we have that $\Wh(Q_8) = \Wh(Q_{12}) = 0$. Thus, over $Q_8$ and $Q_{12}$, all these presentations are simple homotopy equivalent.
We ask:

\begin{question} \label{q:sh-h}
Do there exist presentations for $Q_{4n}$ of the form $\mathcal{P}_n(n_1,\dotsc,n_k;m_1,\dotsc,m_k)$ which are homotopy equivalent but not simple homotopy equivalent? 
\end{question}

By \cite{Ol88}, $\Wh(Q_{4n})$ is infinite for $n \ge 4$ so the possibility remains that there are infinitely many simple homotopy types of presentations in these cases. If so then, since there are only finitely many homotopy types, this would give the first examples of an infinite family of finite presentations which are all homotopy equivalent but pairwise not simple homotopy equivalent.
This would also give the first examples of finite groups with presentations which are homotopy equivalent but not simple homotopy equivalent.

If \cref{q:sh-h} had a negative answer and these presentations were also not counterexamples to the generalised Andrews--Curtis conjecture then, in light of the discussion above, there must exist $Q^{**}$-equivalences which reduce the value of $k$ under certain conditions. However, we could not find any operations which alter the value of $k$ except in the trivial case where one of $n_1, \dotsc, n_{k+1}$ is equal to zero. We therefore ask:

\begin{question}
Do there exist presentations for $Q_{4n}$ of the form $\mathcal{P}_n(n_1,\dotsc,n_k;m_1,\dotsc,m_k)$ which are simple homotopy equivalent but not generalised Andrews--Curtis equivalent?
\end{question}

\subsection*{Organisation of the paper}

The article will be structured as follows. In \cref{section:prelim}, we give the necessary background on modules, group presentations and the D2 problem. In \cref{s:regular-pres}, we introduce the presentations of the form $\mathcal{P}_n(n_1,\dotsc,n_k;m_1,\dotsc,m_k)$ and establish their properties. In \cref{s:pres-small-order}, we use these results to establish Theorems \ref{thmx:main-D2}, \ref{thmx:first-example}, \ref{thmx:a-non-MP-presentation} and \ref{thmx:main-Q4n}.

\subsection*{Acknowledgements} 
TH gratefully acknowledges support by the Deutsche Forschungsgemeinschaft -- Project-ID 286237555 -- TRR 195; and Project-ID 539387714. JN was supported by a Rankin-Sneddon Research Fellowship from the University of Glasgow. We would like to thank Thomas Breuer, Derek Holt, Max Horn and John Voight for helpful conversations.

\section{Preliminaries} \label{section:prelim}

We now develop the necessary preliminaries for the proof of the main results. In \cref{ss:modules-prelim}, we discuss $\Z G$-modules and pullback squares. In \cref{ss:D2-prelim}, we give algebraic methods for tackling the D2 problem in the case of quaternion groups. Finally, in \cref{ss:Fox}, we show how to determine the chain complex $C_*(\wt X_{\mathcal{P}})$ using Fox calculus.
Throughout this section, all $R$-modules will be taken to be left $R$-modules unless otherwise specified.

\subsection{Modules over integral group rings} \label{ss:modules-prelim}

Let $A$ be a finite-dimensional semisimple $\Q$-algebra. An \textit{order in $A$} is a subring $\l \le A$ which is finitely generated as an abelian group and which has $\Q \cdot \l = A$. For example, if $G$ is a finite group, then $\Q G$ is a finite-dimensional semisimple $\Q$-algebra and $\Z G$ is an order in $\Q G$.
Note that, without specifying $A$, an \textit{order} is simply a ring whose underlying abelian group is free of finite rank.

For a prime $p$ and an order $\l$, define a ring $\l_p = \Z_p \otimes \l$ where $\Z_p$ denotes the $p$-adic integers. We say a $\l$-module $M$ is \textit{locally free} if it is finitely generated and $M_p = \l_p \otimes M$ is a free $\l_p$-module for all primes $p$.
Every locally free $\l$-module is projective \cite[Lemma 2.1]{Sw80}. The converse does not hold in general but does hold in the case of $\Z G$ for $G$ a finite group, hence a finitely generated $\Z G$-module is projective if and only if it is locally free \cite[p156]{Sw80}.

Throughout this article we will make frequent use of the following general lemma. This is presumably well known but we could not locate a suitable reference in the literature. 

\begin{lemma} \label{lemma:ext-ideal}
Let $\l$, $R$ be orders and let $f : \l \to R$ be ring homomorphism.
If $I \le \l$ is a locally free ideal, then $f_\#(I)$ is isomorphic to the ideal $R \cdot f(I) \le R$ generated by $f(I)$. Furthermore, if $f$ is surjective, then $f_\#(I)$ is isomorphic to the ideal $f(I) \le R$. 
\end{lemma}

\begin{proof}
Let $\l_p = \Z_p \otimes \l$.
Since $I$ is locally free, we have that $\Z_p \otimes I \cong \l_p^m$ for some $m$. Since $\Z \to \Z_p$ is a flat extension, the inclusion $I \hookrightarrow \l$ induces an inclusion $\l_p^m \hookrightarrow \l_p$. By considering $\Z_p$-ranks this implies that $m \le 1$ since $\Z_p$ is commutative and so satisfies the strong rank condition. If $m=0$, then $I=\{0\}$ and the statement holds trivially, so we will assume that $m=1$. Thus the $\Z$-submodule $I$ of $\Lambda$ is of full rank, which is equivalent to $\Lambda/I$ being finite.

Consider the exact sequence of $\l$-modules:
\[ 0 \to I \to \l \to \l/I \to 0. \]
Now apply $f_\#(\,\cdot\,) = R \otimes_{\l} (\,\cdot\,)$ to obtain an
exact sequence of $R$-modules
\begin{equation}  \Tor^{\l}_1(\l/I,R) \to f_\#(I) \to \underbrace{f_\#(\l)}_{\cong R} \to f_\#(\l/I) \to 0 \label{eq:Tor} \end{equation}
where, in $\Tor^{\l}_1(\l/I,R)$, we view $R$ as a right $\l$-module via $f$. Since $I$ is locally free, and hence projective, $f_\#(I)$ is a projective $R$-module. Since the underlying abelian group of $R$ is torsion free, this implies that $f_\#(I)$ is torsion free as an abelian group. On the other hand, $\Tor^{\l}_1(\l/I,R)$ is annihilated by $|\l/I| < \infty$ and so is torsion~\cite[Corollary~3.2.12]{Weibel1994}. Hence the map $f_\#(I) \to R$ in \cref{eq:Tor} is injective. This completes the proof since the image is the ideal of $R$ generated by $f(I)$.

The final part follows from the fact that, if $f$ is surjective, then $f(I) = f(\l \cdot I) = R \cdot f(I)$.
\end{proof}

For $\theta \in \Aut(G)$ and a $\Z G$-module $M$, let $M_\theta$ denote the $\Z G$-module with the same underlying abelian group as $M$ but with action $g \cdot m = \theta(g)m$ for $g \in G$ and $m \in M$. We can view $\theta \in \Aut(G)$ as a ring isomorphism $\theta : \Z G \to \Z G$ sending $\sum_{i=1}^n a_ig_i \mapsto \sum_{i=1}^n a_i\theta(g_i)$ for $a_i \in \Z$ and $g_i \in G$.

\begin{corollary} \label{cor:theta-action-ideal}
Let $G$ be a finite group and let $\theta \in \Aut(G)$. If $I = (x_1, \dotsc, x_n) \le \Z G$ is an ideal, then $I_\theta$ is isomorphic to the ideal $(\theta^{-1}(x_1),\dotsc,\theta^{-1}(x_n)) \le \Z G$.
\end{corollary}

\begin{proof}
This follows from \cref{lemma:ext-ideal} and the fact that $I_\theta \cong (\theta^{-1})_\#(I)$ \cite[Lemma 8.3]{Ni20a}.
\end{proof}

Using \cref{lemma:ext-ideal}, we now give a specialisation of Milnor's theory of projective modules over pullback squares \cite[Chapter 2]{Mi71} to the case of ideals.
We will restrict to pullback squares of the following standard type (see \cite[Section 2.3]{Ni24}).
Let $\l$ be an order in a finite-dimensional semisimple $\Q$-algebra $A$ and let $A \cong A_1 \times A_2$ be an isomorphism of $\Q$-algebras. For $i=1,2$, let $\l_i = \l/I_i$ denote the projection of $\l$ onto $A_i$. Then there is a pullback square of rings
\[
\mathcal{R}(\l,A_1,A_2) = 
\begin{tikzcd}
   \l \ar[r,"i_2"] \ar[d,"i_1"] & \l_2 \ar[d,"j_2"] \\
   \l_1 \ar[r,"j_1"] & \bar{\l}
\end{tikzcd}
\]
where $\bar{\l} = \l/(I_1+I_2)$ is a finite ring and all maps are the natural surjections. For a ring $R$, let $P(R)$ denote the set of isomorphism classes of finitely generated projective $R$-modules.

\begin{lemma} \label{lemma:Milnor-square-ideals}
Let $\mathcal{R}(\l,A_1,A_2)$ be as above. Then there is a bijection
\[ \varphi : \{ [P] \in P(\l) : (i_1)_\#(P) \cong \l_1, (i_2)_\#(P) \cong \l_2\} \to \l_1^\times \setminus \bar{\l}^\times / \l_2^\times.  \]
If $I \le \l$ is a projective $\l$-ideal such that $i_1(I) = (\lambda_1) \le \l_1$, $i_2(I) = (\lambda_2) \le \l_2$ and $j_1(\lambda_1) \in \ol{\l}^\times$, then $\varphi(I) = [j_1(\lambda_1)j_2(\lambda_2)^{-1}]$.
\end{lemma}

\begin{proof}
The existence of such a bijection follows from \cite[Chapter 2]{Mi71}, so we will just prove the statement about the projective ideal.
Since $i_k$ is surjective for $k=1,2$, \cref{lemma:ext-ideal} implies that $(i_k)_\#(I) \cong i_k(I) = (\lambda_k)$. Thus $(i_k)_\#(I) \cong \l_k$ since it is singly generated and projective. Thus $I$ is in the domain of $\varphi$.
Note that $(j_2(\lambda_2)) = (j_2 \circ i_2)(I) = (j_1 \circ i_1)(I) = (j_1(\lambda_1)) = \ol{\l}$ since $j_1(\lambda_1) \in \ol{\l}^\times$. Hence $j_2(\lambda_2) \in \ol{\l}^\times$ and so we have that $j_1(\lambda_1)j_2(\lambda_2)^{-1} \in \ol{\l}^\times$.

Since $j_1$, $j_2$ are surjective, $\mathcal{R}(\l,A_1,A_2)$ is a Milnor square and so results in \cite[Chapter 2]{Mi71} imply there exists a bijection $\varphi$ such that, if $u \in \bar{\l}^\times$, then $\varphi^{-1}(u) = [M(\l_1,\l_2,u)]$ where
\[ M(\l_1,\l_2,u) := \{ (x_1,x_2) \in \l_1 \times \l_2 : j_1(x_1)u = j_2(x_2)\} \le \l_1 \times \l_2. \]
We view $\l_1 \times \l_2$ as a $\l$-module under the action $\lambda \cdot (x_1,x_2) = (i_1(\lambda)x_1,i_2(\lambda)x_2)$ for $\lambda \in \l$, and $M(\l_1,\l_2,u)$ is a $\l$-submodule of $\l_1 \times \l_2$.

It now suffices to prove that, if $u = j_1(\lambda_1)j_2(\lambda_2)^{-1}$, then there is an isomorphism of $\l$-modules $I \cong M(\l_1,\l_2,u)$.
Note that there is a $\l$-module homomorphism $g : I \to M(\l_1,\l_2,u)$ such that $g(x) = (x_1,x_2)$ where $i_k(x) = x_k \lambda_k$ for each $k=1,2$. This is well-defined since $j_1 \circ i_1 = j_2 \circ i_2$ implies that $j_1(x_1)j_1(\lambda_1) = j_2(x_2)j_2(\lambda_2)$ and so $j_1(x_1)u = j_2(x_2)$. This is bijective since $\mathcal{R}(\l,A_1,A_2)$ is a pullback square. 
\end{proof}

\begin{remark}\label{lemma:Milnor-square-modules}
We have the following variant of Lemma~\ref{lemma:Milnor-square-ideals}, which can be proven similarly (see~\cite[\S{}4]{Sw80}).
Let $Q$ be a projective $\Lambda$-module, $Q_k = i_k(Q)$, $k = 1,2$ and $\overline Q = (j_1)_\#(Q_1) \cong (j_2)_\#(Q_2)$. Then there is a bijection
\[ \varphi : \{ [P] \in P(\l) : (i_1)_\#(P) \cong Q_1, (i_2)_\#(P) \cong Q_2\} \to \Aut(Q_1) \backslash {\Aut(\bar{Q})}/{\Aut(Q_2)}.  \]
If $P$ is a projective $\l$-module, $\bar P = (j_1 \circ i_1)_\#(P) \cong (j_2 \circ i_2)_\#(P)$, $f_k \colon Q_k \to (i_k)_\#(P)$ isomorphisms with induced isomorphisms $\bar f_k \colon \bar Q \to \bar P$, then $\varphi([P]) = [\bar f_2^{-1} \bar f_1]$.
\end{remark}

We next introduce a special class of modules over $\Z G$ for $G$ a finite group.
A \textit{Swan module} is a $\Z G$-ideal of the form $(N,r) \le \Z G$ where $N = \sum_{g \in G} g$ and $r \in \Z$ is an integer coprime to $|G|$. This is projective and only depends up to isomorphism on $r \mod |G|$, so we can view $r \in (\Z/|G|)^\times$.
Swan modules are also defined as $(I_G,r) \le \Z G$ where $I_G$ is the augmentation ideal. These definitions coincide since $(I_G,r) \cong (N,r^{-1})$ for all $r \in (\Z/|G|)^\times$ \cite[Proposition 3.4]{Ni20a}.

If $I \le \Z G$ is an ideal then, since $(N,r)$ is a two-sided ideal and hence a $(\Z G, \Z G)$-bimodule, we can form a (left) $\Z G$-module $(N,r) \otimes I$. This is isomorphic to the ideal $N \cdot I +  r \cdot I \le \Z G$.

\begin{lemma} \label{lemma:Swan-theta-action}
Let $r \in (\Z/|G|)^\times$ and let $\theta \in \Aut(G)$. Then $(N,r)_\theta \cong (N,r)$. More generally, if $I \le \Z G$ is an ideal, then $((N,r) \otimes I)_\theta \cong (N,r) \otimes I_\theta$.
\end{lemma}

\begin{proof}
By \cref{cor:theta-action-ideal}, we have $(N,r)_\theta \cong (\theta^{-1}(N),\theta^{-1}(r)) = (N,r)$ since $N,r$ are invariant under the action of $\Aut(G)$. Let $I \le \Z G$ be an ideal. Since $\Z G$ is Noetherian, we have $I = (x_1,\dotsc,x_n)$ for some $x_i \in \Z G$, and so $N \cdot I +  r \cdot I = (Nx_1,\dotsc,Nx_n,rx_1,\dotsc,rx_n)$. Hence:
\begin{align*} ((N,r) \otimes I)_\theta & \cong (\theta^{-1}(Nx_1),\dotsc,\theta^{-1}(Nx_n),\theta^{-1}(rx_1),\dotsc,\theta^{-1}(rx_n)) \\
& = (N\theta^{-1}(x_1),\dotsc,N\theta^{-1}(x_n),r\theta^{-1}(x_1),\dotsc,r\theta^{-1}(x_n)) \cong (N,r) \otimes I_\theta. \qedhere
\end{align*}
\end{proof}

We now define the dual of a left $\Z G$-module $M$. Firstly, let $M\hh := \Hom_{\Z G}(M,\Z G)$ denote the right $\Z G$-module with the standard action $\varphi \cdot \lambda : m \mapsto \varphi(m)\lambda$ for $\lambda \in \Z G$, $\varphi \in M\hh$ and $m \in M$.
Secondly, given a right $\Z G$-module $N$, define $c(N)$ to be the left $\Z G$-module with underlying abelian group $N$ and action $\lambda \cdot m := m \cdot_{N} \bar{\lambda}$ for $\lambda \in \Z G$, $m \in M$ where $\cdot_{N}$ denotes the $\Z G$-action on $N$ and $\ol{\cdot}: \Z G \to \Z G$ is the involution $\sum_{i=1}^n a_i g_i \mapsto \sum_{i=1}^n a_i g_i^{-1}$. 
For a left $\Z G$-module $M$, we define the \textit{dual} of $M$ to be the left $\Z G$-module $M^* := c(M\hh)$ (see \cite[p318]{CR62}).

\begin{remark}
Many authors use the notation $M\hh$ to refer to $M^*$ and vice versa. We chose the notation above since the left $\Z G$-module dual $M^*$ is the primary object of interest in this article.
\end{remark}

For a left $\Z G$-module $N$, we can define a right $\Z G$-module $c(N)$ similarly. If $N$ is a $(\Z G,\Z G)$-bimodule, we define $c(N)$ to be the $(\Z G, \Z G)$-bimodule with action $\lambda \cdot x \cdot \mu := \bar{\mu} \cdot_N x \cdot_N \bar{\lambda}$ for $\lambda, \mu \in \Z G$ and $x \in N$.
For a right $\Z G$-module (resp. $(\Z G, \Z G)$-bimodule) $M$, we can also define left $\Z G$-modules (resp. $(\Z G, \Z G)$-bimodules) $M\hh$ and $M^* = c(M\hh)$.

\begin{lemma} \label{lemma:dual-of-tensor-product}
Let $M$ be a $(\Z G,\Z G)$-bimodule and $N$ a left $\Z G$-module. Then there is an isomorphism of left $\Z G$-modules $(M \otimes N)^* \cong M^* \otimes N^*$.
\end{lemma}

\begin{proof}
We begin by noting that there is an isomorphism $f : c(N\hh \otimes M\hh) \to M^* \otimes N^*$ given by sending $\varphi \otimes \psi \mapsto \psi \otimes \varphi$ for $\varphi \in N\hh$ and $\psi \in M\hh$. This is bijective and is a left $\Z G$-module homomorphism since, for $\lambda \in \Z G$, $\varphi \in N\hh$ and $\psi \in M\hh$ we have 
\[ f(\lambda \cdot_{c(N\hh \otimes M\hh)} (\varphi \otimes \psi)) = f(\varphi \otimes (\psi \bar{\lambda})) = (\psi \bar{\lambda}) \otimes \varphi = (\lambda \cdot_{M^*} \psi) \otimes \varphi = \lambda \cdot_{M^* \otimes N^*} f(\varphi \otimes \psi). \]

Since $(M \otimes N)^* \cong c((M \otimes N)\hh)$, it now suffices to prove that $N\hh \otimes M\hh \cong (M \otimes N)\hh$ as right $\Z G$-modules.
By~\cite[(2.29)]{CR81} (particularly the formulation given in 2.32), the map
\[ N\hh \otimes M\hh \to \Hom_{\Z G}(N, M\hh), \varphi \otimes \psi \mapsto (x \mapsto \varphi(x) \cdot \psi) \]
is an isomorphism of abelian groups, for which it is straightforward to verify that it is an isomorphism of right $\Z G$-modules.
By~\cite[(2.19)]{CR81} the map
\[ \Hom_{\Z G}(N, M\hh) \to (M \otimes N)\hh, \varphi \mapsto ((x \otimes y) \mapsto (\varphi(y))(x) \]
is an isomorphism of abelian groups, and we can similarly verify this is an isomorphism of right $\Z G$-modules.
Thus $N\hh \otimes M\hh \cong \Hom_{\Z G}(N, M\hh) \cong (M \otimes N)\hh$, as required.
\end{proof}

\subsection{Finite 2-complexes over groups with free period 4} \label{ss:D2-prelim}

For a finitely presented group $G$, let $\HT(G)$ denote the set of homotopy types of finite $2$-complexes $X$ with $\pi_1(X) \cong G$. This can be viewed as a graph with edges between each $X$ and $X \vee S^2$. It is well-known that $\HT(G)$ is a tree and has a grading coming from $\chi(X)$ which has a minimum value $\chi_{\min}(G)$ (see, for example, \cite[p2245]{Ni20a}).

In the case where $G$ is finite, it was shown by W. H. Browning \cite{Br79a} (see \cite{HK93} for a published proof) that $X \simeq Y$ provided $\chi(X) = \chi(Y) > \chi_{\min}(G)$ and that there are finitely many $X$ up to homotopy with $\chi(X)= \chi_{\min}(X)$. In particular, $\HT(G)$ has the following form:

\begin{figure}[h] \vspace{-2mm}
\begin{tikzpicture}
%% vertices
\draw[fill=black] (0,0) circle (2pt);
\draw[fill=black] (1,0) circle (2pt);
\draw[fill=black] (2,0) circle (2pt);
\draw[fill=black] (3,0) circle (2pt);
\draw[fill=black] (4,0) circle (2pt);
\draw[fill=black] (2,1) circle (2pt);
\draw[fill=black] (2,2) circle (2pt);
\draw[fill=black] (2,3) circle (2pt);
%% labels
\node at (2,3.6) {$\vdots$};
%% edges
\draw[thick] (0,0) -- (2,1) (1,0) -- (2,1) (2,0) -- (2,1) (3,0) -- (2,1) (4,0) -- (2,1) -- (2,2) -- (2,3);
\end{tikzpicture}
\caption{Tree structure for $\HT(G)$ when $G$ is a finite group}
\label{figure:fork}
\end{figure}
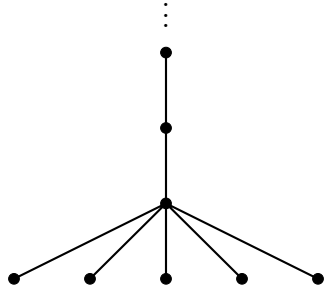

Suppose $G$ is a finite group with 4-periodic cohomology, i.e. the Tate cohomology groups satisfy $\widehat{H}^{i}(G;\Z) \cong \widehat{H}^{i+4}(G;\Z)$ for all $i \in \Z$. This is equivalent to the existence of a 4-periodic projective resolution of $\Z$ \cite[XII]{CE56}. We say $G$ has \textit{free period $4$} if there exists a 4-periodic free resolution of $\Z$, or equivalently an exact sequence of finitely generated $\Z G$-modules
\begin{equation} F = (0 \to \Z \to F_3 \to F_2 \to F_1 \to F_0 \to \Z \to 0) \label{eq:free-res} \end{equation}
where each $F_i$ is free. For example, $Q_{4n}$ has free period 4 for all $n \ge 2$ \cite[p253]{CE56}.

Let $\SF(\Z G)$ denote the set of isomorphism classes of non-zero finitely generated stably free $\Z G$-modules. 
It follows from \cite[Theorem B]{Ni20a} that, if $G$ has free period $4$, then there is an injective map $\Psi$ from $\HT(G)$ to the set of orbits of $\SF(\Z G)$ by a certain non-standard action of $\Aut(G)$ on $\SF(\Z G)$, and $\Psi$ is bijective if and only if $G$ has the D2 property.

Let $\HT_{\min}(G) = \{ X \in \HT(G) : \chi(X) = \chi_{\min}(X)\}$ and let $\SF_1(\Z G)$ denote the set of isomorphism classes of rank one stably free $\Z G$-modules. Since $\HT(G)$ is as in \cref{figure:fork}, $\Psi$ is bijective if and only if the restricted map from $\HT_{\min}(G)$ to the orbits of $\SF_1(\Z G)$ is bijective. 

One drawback with this result is that both the action of $\Aut(G)$ on $\SF(\Z G)$ and the map $\Psi$ itself are difficult to determine in specific examples. 
We circumvent these difficulties by quotienting $\SF(\Z G)$ by an equivalence relation which is a priori weaker than the action of $\Aut(G)$. We use that every non-zero finitely generated projective $\Z G$-module is of the form $I \oplus \Z G^n$ where $I \le \Z G$ is an ideal and $n \ge 0$ \cite[Theorem A]{Sw60}.

\begin{lemma} \label{lemma:SF-eq-rel}
For $P$, $Q \in \SF(\Z G)$, say $P \sim Q$ if there exist ideals $I, J \le \Z G$ and $n \ge 0$ such that $P \cong I \oplus \Z G^n$, $Q \cong J \oplus \Z G^n$ and $I \cong (N,r) \otimes J_\theta$ for some $r \in (\Z/|G|)^\times$, $\theta \in \Aut(G)$.
Then:
\begin{clist}{(i)}
\item
$\sim$ is an equivalence relation on $\SF(\Z G)$. 
\item
If $I \cong (N,r) \otimes J_\theta$ as above, then $(N,r)$ is stably free. In particular, if stably free Swan modules are free, then $\sim$ coincides with isomorphism under the action of $\Aut(G)$.
\item 
If $P \sim Q$, then $P^* \sim Q^*$. More specifically, if $P \cong I \oplus \Z G^n$, $Q \cong J \oplus \Z G^n$ and $I \cong (N,r) \otimes J_\theta$, then $I^* \cong (N,r^{-1}) \otimes (J^*)_\theta$.
\end{clist}
\end{lemma}

\begin{proof}
(i) Reflexivity is clear. Note that, if $r,s \in (\Z/|G|)^\times$, then $(N,r) \otimes (N,s) \cong (N,rs)$. For symmetry, if $P \sim Q$, then $P \cong I \oplus \Z G^n$, $Q \cong J \oplus \Z G^n$  and $I \cong (N,r) \otimes J_\theta$ which implies that
\[ (N,r^{-1}) \otimes I_{\theta^{-1}} \cong (N,r^{-1}) \otimes ((N,r) \otimes J_\theta)_{\theta^{-1}} \cong (N,r^{-1}) \otimes (N,r) \otimes (J_\theta)_{\theta^{-1}} \cong (N,r^{-1}r) \otimes J \cong J \]
where we have used \cref{lemma:Swan-theta-action} for the second isomorphism, and the fact that $r^{-1}r \equiv 1 \mod |G|$ implies that $(N,r^{-1}r) \cong (N,1) = \Z G$ for the fourth.
Transitivity follows similarly.

(ii) Since $P$ is stably free, $P_\theta$ is stably free. By \cite[Lemma 4.15]{Ni20a}, we have $[Q] = [(N,r)] + [P_\theta] \in \wt K_0(\Z G)$. Since $[Q]=[P_\theta]=0$, we have that $[(N,r)]=0$, i.e. $(N,r)$ is stably free.

(iii) First note that, if $I \le \Z G$ is an ideal, then $((N,r) \otimes I)^* \cong (N,r)^* \otimes I^*$ by \cref{lemma:dual-of-tensor-product}.
Next note that $(N,r)^* \cong (I_G,r) \cong (N,r^{-1})$ by combining \cite[Lemma 17.1]{Sw83} with \cite[Proposition 3.4]{Ni20a}, and we have $(I_\theta)^* \cong (I^*)_\theta$ (see, for example, the proof of \cite[Lemma 6.3]{Ni20a}). Hence, if $I \cong (N,r) \otimes J_\theta$, then $I^* \cong (N,r)^* \otimes (J_\theta)^* \cong (N,r^{-1}) \otimes (J^*)_\theta$, as required.
\end{proof}

The following gives weaker yet easily computable analogue of $\Psi$ (see (i)) and also computes $\Psi$ in a special case (see (ii)). Here $\SF(\Z G)/{\Aut(G)}$ denotes the set of orbits of the standard action of $\Aut(G)$ on $\SF(\Z G)$, where $\theta \in \Aut(G)$ sends $P \mapsto P_\theta$.

\begin{prop} \label{prop:Psi-explicit}
Let $G$ have free period $4$. Then:
\begin{clist}{(i)}
\item
There is a map $\wh \Psi : \HT(G) \to \SF(\Z G) /{\sim}$ where $\wh \Psi(X) = [P]$ for any projective $\Z G$-module $P$ such that there is an exact sequence of the form $0 \to \Z  \to P \to \pi_2(X) \to 0$.
\item 
If stably free Swan modules over $\Z G$ are free, then there is an injective map $\Psi : \HT(G) \to \SF(\Z G)/{\Aut(G)}$ where $\Psi(X) = [P]$ for any projective $\Z G$-module $P$ such that there is an exact sequence of the form $0 \to \Z  \to P \to \pi_2(X) \to 0$. Furthermore, $\Psi$ is bijective if and only if $G$ has the {\normalfont D2} property.
\end{clist}
\end{prop}

\begin{proof}
We begin by recalling the bijection given by \cite[Theorem B]{Ni20a} in more detail. First fix a 4-periodic free resolution $F$ as in \cref{eq:free-res}. 
For $P$, $Q \in \SF(\Z G)$, say $P \equiv Q$ if there exist ideals $I, J \le \Z G$ and $n \ge 0$ such that $P \cong I \oplus \Z G^n$, $Q \cong J \oplus \Z G^n$ and $I \cong (N,\psi_4(\theta)) \otimes J_\theta$ for some $\theta \in \Aut(G)$, where $\psi_4 : \Aut(G) \to (\Z/|G|)^\times$ is a group homomorphism which depends only on $G$ (see \cite[Section 7]{Ni20a} for a definition).
This is an equivalence relation and, by \cite[Theorem 5.3 (i)]{Ni20a}, there is an injective map 
\[ \Psi_F : \HT(G) \to \SF(\Z G) /\equiv \, , \quad X \mapsto [P_{X,F}] \]
where $P_{X,F}$ is the unique projective $\Z G$-module such that there is a chain homotopy equivalence
\[ (0 \to \Z \xrightarrow[]{\alpha} P_{X,F} \xrightarrow[]{\beta} \pi_2(X) \to 0) \circ C_*(\wt X) \simeq F \]
for some $\alpha$, $\beta$, and where $\circ$ denotes Yoneda product. Note that $P_{X,F}$ is necessarily stably free.

(i) For $P$, $Q \in \SF(\Z G)$, it is clear that $P \equiv Q$ implies $P \sim Q$. Define $\wh \Psi$ to be the composition of $\Psi_F$ and the quotient map $\SF(\Z G) /\equiv \, \twoheadrightarrow \SF(\Z G) /{\sim}$.
Let $P$ be any projective $\Z G$ module such that there is an exact sequence of the form $0 \to \Z \xrightarrow[]{\alpha} P \xrightarrow[]{\beta} \pi_2(X) \to 0$ for some maps $\alpha$, $\beta$. We claim that $\wh \Psi(X) = [P]$ which, in particular, implies that $\wh \Psi$ does not depend on the choice of $F$.
It suffices to prove that $P \sim P_{X,F}$.

Let $F' = (0 \to \Z \xrightarrow[]{\alpha} P \xrightarrow[]{\beta} \pi_2(X) \to 0) \circ C_*(\wt X)$. 
Let $\Proj_{\Z G}^4(\Z, \Z)$ denote the set of $4$-periodic projective resolutions of $\Z$ up to congruence (see \cite[p2249]{Ni20a}), which is a weaker notion than chain homotopy equivalence.
For any $E \in \Proj_{\Z G}^4(\Z, \Z)$, there is a bijection
\[ (m_{\cdot})^* : (\Z/|G|)^\times \to \Proj_{\Z G}^4(\Z,\Z), \quad r \mapsto (m_r)^*(E) \]
where $m_r : \Z \to \Z$ denotes multiplication by $r$, and $(\,\cdot\,)^*$ denotes the pullback of exact sequences \cite[Proposition 3.5]{Ni20a}.
Let $E$ be the dual extension
\[E = (F')^* = C_*(\wt X)^* \circ (0 \to \pi_2(X)^*\xrightarrow[]{\beta^*} P^* \xrightarrow[]{\alpha^*} \Z \to 0).\] 
Then there exists $r \in (\Z/|G|)^\times$ such that $(m_r)^*(E) \simeq F^*$. By \cite[Lemma 4.12]{Ni20a}, we have
\begin{align*} F^* \simeq (m_r)^*(E) &\simeq C_*(\wt X)^* \circ (m_r)^*(0 \to \pi_2(X)^*\xrightarrow[]{\beta^*} P^* \xrightarrow[]{\alpha^*} \Z \to 0) \\
&\simeq C_*(\wt X)^* \circ (0 \to \pi_2(X)^*\xrightarrow[]{\gamma} Q \xrightarrow[]{\delta} \Z \to 0) \end{align*}
for some $\gamma$, $\delta$ and where, if $P^* \cong I \oplus \Z G^n$ for an ideal $I \le \Z G$, then $Q \cong ((I_G,r)\otimes I) \oplus \Z G^n$. By \cite[Proposition 3.4]{Ni20a}, $(I_G,r) \cong (N,r^{-1})$ and so $Q^* \cong ((N,r^{-1})\otimes I) \oplus \Z G^n$.
Dualising the sequence above now gives that
\[ F \simeq (0 \to \Z \xrightarrow[]{\delta^*} Q^* \xrightarrow[]{\gamma^*} \pi_2(X) \to 0) \circ C_*(\wt X) \]
and so $P_{X,F} = Q^*$.  By \cref{lemma:SF-eq-rel} (iii), we have $Q^* \cong ((N,r) \otimes I^*) \oplus \Z G^n$. Since $P \cong I^* \oplus \Z G^n$ and $I^*$ is isomorphic to a $\Z G$-ideal, this implies that $P_{X,F} = Q^* \sim P$, as required.

(ii) If stably free Swan modules are free, \cref{lemma:SF-eq-rel} (ii) implies that the three equivalence relations we have considered on $\SF(\Z G)$ all coincide: $\sim$, $\equiv$ and equivalence under the standard action of $\Aut(G)$. Hence $\Psi := \Psi_F = \wh \Psi : \HT(G) \to \SF(\Z G)/{\Aut(G)}$ is injective and of the required form. 
By \cite[Theorem B]{Ni20a}, $\Psi$ is bijective if and only if $G$ has the D2 property.
\end{proof}

If $\mathcal{P}$ is a presentation for a finite group $G$, then we can write $\Psi(\mathcal{P})$ to denote $\Psi(X_{\mathcal{P}})$, where $X_{\mathcal{P}}$ is the presentation 2-complex associated to $\mathcal{P}$.

\begin{example}
Let $n \ge 2$ and let $\mathcal{P}_n^{\std}$ denote the standard presentation for $G =Q_{4n}$. Then $\pi_2(X_{\mathcal{P}_n^{\std}}) \cong \Z G/N$ by comparing $C_*(\wt X_{\mathcal{P}_n^{\std}})$ with the resolution in \cite[p253]{CE56}. There is an exact sequence $0 \to \Z \xrightarrow[]{\cdot N} \Z G \to \Z G/N \to 0 $ and so $\wh \Psi(\mathcal{P}_n^{\std}) = [\Z G]$.
\end{example}

\subsection{Group presentations and Fox calculus}
\label{ss:Fox}

Suppose $\mathcal{P} = \langle x_1, \dotsc, x_n \mid r_1, \dotsc, r_m \rangle$ is a finite presentation for a group $G$. We can transform $\mathcal{P}$ into another finite presentation for $G$ using the following moves:
\begin{clist}{(i)}
\item Replacing a relator $r_i$ by $r_ir_j$ for some $i \ne j$.
\item Replacing a relator $r_i$ by $r_i^{-1}$.
\item Replacing a relator $r_i$ by $w r_i w^{-1}$ for some $w \in F(x_1, \dotsc, x_n)$.	
\item Replacing each relator $r_i$ by $\phi(r_i)$ for some $\phi \in \Aut(F(x_1, \dotsc, x_n))$.
\item Add a generator $x_{n+1}$ and a relator $r_{m+1}$ which coincides with $x_{n+1}$.
\item The inverse of (v), when possible.
\end{clist}
The moves (i)-(iii) are $Q$-transformations, the moves (i)-(iv) are $Q^*$-transformations and the moves (i)-(vii) are $Q^{**}$-transformations.
Let $R = Q$, $Q^*$ or $Q^{**}$. We say two finite presentations $\mathcal{P}$, $\mathcal{Q}$ are \textit{$R$-equivalent}, written $\mathcal{P} \simeq_{R} \mathcal{Q}$, if they are related via a sequence of $R$-transformations.

If $\mathcal{P}$ has $n$ generators and $m$ relations, then define the \textit{deficiency} $\Def(\mathcal{P}) = n - m$. This is a homotopy invariant since $\chi(X_{\mathcal{P}}) = 1 -\Def(\mathcal{P})$. Recall that we defined what it means for two finite presentations to be homotopy (resp. simple homotopy) equivalent in the introduction.
It is straightforward to see that the $Q^{**}$-transformations induce simple homotopy equivalences and so we have the following chain of implications: 
\[ \mathcal{P} \simeq_{Q} \mathcal{Q} 
\Rightarrow \mathcal{P} \simeq_{Q^*} \mathcal{Q} 
\Rightarrow \mathcal{P} \simeq_{Q^{**}} \mathcal{Q} 
\Rightarrow \mathcal{P} \simeq_{s} \mathcal{Q} 
\Rightarrow \mathcal{P} \simeq \mathcal{Q}
\Rightarrow \Def(\mathcal{P}) = \Def(\mathcal{Q}). \] 

We will now define Fox differentiation of elements in a free group, and use this to describe the boundary maps in the chain complex $C_*(\wt X_{\mathcal{P}})$, where $\mathcal{P}$ is a finite presentation for a group $G$.
We will follow \cite{Fo53} and \cite[II.3]{LS77} (see also \cite[Section 1.2]{Ha18}).

Let $X = \{x_1,\dotsc,x_n\}$ be a finite set and let $F(X)$ denote the free group on $X$. Then the \textit{Fox derivative} with respect to $x_i$ is the homomorphism of abelian groups
\[ \partial_{x_i} : \Z F(X) \to \Z F(X) \]
which is defined by the requirements that $\partial_{x_i}(x_j) = \delta_{ij}$, where $\delta_{ij}$ is the Kronecker delta function, and the product rule $\partial_{x_i}(uv) = \partial_{x_i}(u)+u\partial_{x_i}(v)$ for any $u,v \in F(X)$.

More generally, if $G$ is a group with generating set $X = \{x_1,\dotsc,x_n\}$, we have an implicit surjection $\phi : F(X) \twoheadrightarrow G$. From this, we obtain a map $\partial_{x_i}(G) : \Z F(X) \to \Z G$ by post-composing with $\phi$. The special case where $G = F(X)$ coincides with the original definition above. We will write $\partial_{x_i} = \partial_{x_i}(G)$ when the choice of $G$ is clear from the context.

\begin{prop} \label{prop:chain-complex-Fox}
Let $\mathcal{P} = \langle x_1,\dotsc,x_m \mid r_1, \dotsc, r_m \rangle$ be a presentation for $G$, and let $C_*(\wt X_{\mathcal{P}})$ denote the corresponding chain complex of $\Z G$-modules. Then $C_0(\wt X_{\mathcal{P}}) \cong \Z G$, $C_1(\wt X_{\mathcal{P}}) \cong \Z G^n$ with basis denoted by $e(x_i)$, and $C_2(\wt X_{\mathcal{P}}) \cong \Z G^m$ with basis denoted by $e(r_i)$. Then, with respect to these identifications, we have
\[ C_*(\wt X_{\mathcal{P}}) = (C_2(\wt X_{\mathcal{P}}) \xrightarrow[]{\partial_2} C_1(\wt X_{\mathcal{P}}) \xrightarrow[]{\partial_1} C_0(\wt X_{\mathcal{P}})) \]
where $\partial_2(e(r_i)) = \sum_{j=1}^m \partial_{x_j}(r_i)e(x_j)$ and $\partial_1(e(x_i)) = x_i-1$.
\end{prop}

We have that $\pi_2(X_{\mathcal{P}}) \cong \pi_2(\wt X_{\mathcal{P}}) \cong H_2(\wt X_{\mathcal{P}}) \cong \Ker(\partial_2)$. For the sake of this article, we will therefore just define $\pi_2(X_{\mathcal{P}}) = \Ker(\partial_2)$. Since $H_0(\wt X_{\mathcal{P}}) \cong \Z$, the chain complex $C_*(\wt X_{\mathcal{P}})$ extends to an exact sequence
\[ 0 \to \pi_2(X_{\mathcal{P}}) \hookrightarrow C_2(\wt X_{\mathcal{P}}) \xrightarrow[]{\partial_2} C_1(\wt X_{\mathcal{P}}) \xrightarrow[]{\partial_1} C_0(\wt X_{\mathcal{P}}) \xrightarrow[]{\varepsilon} \Z \to 0 \]
where $\varepsilon : \Z G \to \Z$ denotes the augmentation map.

We conclude this section by recalling two useful properties of Fox derivatives. The first is a straightforward consequence of the product rule \cite[(2.5)]{Fo53} and the second is \cite[(2.6)]{Fo53}.

\begin{prop} \label{prop:Fox-formula}
Let $X = \{x_1,\dotsc,x_n\}$ and let $w = w_0 x_i^{n_1} w_1 x_i^{n_2} \dotsm w_{k-1}x_i^{n_k} w_k \in F(X)$ where $n_i \in \Z$ and each $w_i \in F(X)$ does not involve the generator $x_i$. Then
\[ \partial_{x_i} w = \sum_{j=1}^k w_0 x_i^{n_1} w_1 x_i^{n_2} \dotsm w_{j-1} \partial_{x_i}(x_i^{n_j}) \in \Z F(X). \]
\end{prop}

\begin{prop}[Chain rule for Fox derivatives] \label{prop:chain-rule}
Let $X = \{x_1,\dotsc,x_n\}$, let $Y = \{y_1, \dotsc, y_m\}$ and let $f : F(Y) \to F(X)$ be a homomorphism. If $w \in F(Y)$, then
\[ \partial_{x_i} f(w) = \sum_{j=1}^m f(\partial_{y_j} w) \partial_{x_i}(f(y_j)) \in \Z F(X). \]
\end{prop}

\section{Regular presentations for quaternion groups}
\label{s:regular-pres}

In this section, we will study presentations for quaternion groups of the form $\langle x, y \mid x^ny^{-2}, R \rangle$ on the standard generating set. In \cref{ss:std-form}, we show how $Q$-transformations can be used to put such presentations in the form $\mathcal{P}_n(n_1,\dotsc,n_{k};m_1,\dotsc,m_{k})$ (\cref{thm:general-presentations-Q4n}), and we establish basic properties of these presentations in \cref{ss:regular-pres-properties}.
In \cref{ss:pi_2-regular-presentation}, we give a formula for the stably free $\Z G$-module such a presentation corresponds to via the map $\Psi$ defined in \cref{ss:D2-prelim} (\cref{thm:Psi(P_R)}) and, 
in \cref{ss:specific-examples}, we make this formula more precise for particular examples.

\subsection{Standard form for regular presentations}
\label{ss:std-form}

Let $F(x_1, \dotsc, x_n)$ denote the free group of rank $n$ on generators $x_1, \dotsc, x_n$.
Suppose $R \in F(x_1, \dotsc, x_n)$ and $w_1, \dotsc, w_n \in F(x_1, \dotsc, x_n)$. Then we define $R(w_1, \dotsc, w_n) \in F(x_1, \dotsc, x_n)$ by replacing each appearance of $x_i$ in $R$ by $w_n$. 

Fix $n \ge 2$ and fix an identification $Q_{4n} = \langle x, y \mid x^ny^{-2}, xyxy^{-1} \rangle$.
Define a \textit{regular} presentation to be a presentation of the form $\langle x, y \mid x^ny^{-2}, R \rangle$ for some $R \in F(x,y)$ which presents $Q_{4n}$ on the standard generating set, i.e. $\id : \langle x, y \mid x^ny^{-2}, R \rangle \to Q_{4n}$, $x \mapsto x, y \mapsto y$ is a group isomorphism. 

For integers $k \ge 0$ and $n_1,m_1, \dots, n_k,m_k \ne 0$, define
\[ \mathcal{P}_n(n_1,\dotsc,n_{k};m_1,\dotsc,m_{k}) := \langle x, y \mid x^ny^{-2}, R(x,yxy^{-1}) \rangle \]
 where $R = a^{n_1}b^{m_1} \cdots a^{n_{k+1}}b^{m_{k+1}}$ for $n_{k+1} = 1 - \sum_{i=1}^{k} n_i$ and $m_{k+1} = 1 - \sum_{i=1}^{k} m_i$. 
We say that $k$ is the \textit{height} of the presentation. If $k=0$ then $R=ab$ and so we obtain the standard presentation $\mathcal{P}_n^{\std} = \langle x, y \mid x^ny^{-2}, yxy^{-1}x \rangle$ for $Q_{4n}$.
The presentations $\mathcal{E}_{n,r}$ of Mannan--Popiel have height one and coincide with $\mathcal{P}_n(2;1-r)$. 
We now aim to prove the following.

\begin{theorem} \label{thm:general-presentations-Q4n}
If $\mathcal{P}$ is a regular presentation for $Q_{4n}$, then there exist integers $k \ge 0$ and $n_1,m_1, \dots, n_k,m_k \ne 0$ such that
$\mathcal{P} \simeq_{Q} \mathcal{P}_n(n_1,\dotsc,n_{k};m_1,\dotsc,m_{k})$.
\end{theorem}

We will begin by establishing basic operations on the second relator which preserve the $Q$-type.

\begin{lemma} \label{lemma:regular-moves}
Let $\mathcal{P} = \langle x, y \mid x^ny^{-2}, R \rangle$ be a presentation for $Q_{4n}$ for some $R \in F(x,y)$. Suppose $R' \in F(x,y)$ has one of the following forms:
\begin{clist}{(i)}
\item
$R'$ is a cyclic permutation of $R$.
\item
$R'$ is formed from $R$ by replacing $y^{\pm 2} 	\subseteq R$ with $x^{\pm n} \subseteq R$.
\item
$R'$ is formed from $R$ by moving a subword $y^{\pm 2} 	\subseteq R$ (or $x^{\pm n} \subseteq R$) to a different place in $R$.
\end{clist}
Then $\mathcal{P} \simeq_{Q} \langle x, y \mid x^ny^{-2}, R' \rangle$.
\end{lemma}

\begin{proof}
(i) This is a $Q$-move of type (iii) since cyclically permuting $R$ corresponds to conjugation.

(ii) First cyclically permute $R$ so that $y^{\pm 2}$ is the leading term. Then alter the first relator to $x^{\pm n} y^{\mp 2}$ (using moves of type (ii) and (iii) if necessary). Then use a move of type (i) to replace $R$ with $x^{\pm n} y^{\mp 2} R$, which simplifies to $R$ with $y^{\pm 2}$ replaced with $x^{\pm n}$. Now cyclically permute $R$ again to move it back in its place. Clearly the same operation works in reverse, i.e. we can replace a $x^{\pm n}$ with a $y^{\pm 2}$.

(iii)
By part (ii), we can move $y^{\pm 2} \subseteq R$ anywhere in $R$ by switching it between $y^{\pm 2}$ and $x^{\pm n}$ whenever it needs to move past a power of $y$ or $x$ respectively.
\end{proof}

This allows us to prove the following, which resolves part of \cref{thm:general-presentations-Q4n}.

\begin{prop} \label{prop:general-presentations-Q4n}
If $\mathcal{P}$ is a regular presentation for $Q_{4n}$, then there exists $R \in F(a,b)$ such that $\mathcal{P} \simeq_{Q} \langle x, y \mid x^ny^{-2}, R(x,yxy^{-1}) \rangle$ where $R = a^{n_1}b^{m_1} \cdots a^{n_k}b^{m_k}$ for some $n_i,m_i \in \Z$ and $k \ge 1$.
\end{prop}

\begin{proof}
We will write $R \simeq R'$ for $R, R' \in F(x,y)$ if $\langle x, y \mid x^ny^{-2}, R \rangle \simeq_{Q} \langle x, y \mid x^ny^{-2}, R' \rangle$.
By abuse of notation, we will continue referring to the second relator as $R$ after the presentation has been altered by a $Q$-transformation.

Let $\mathcal{P} =  \langle x, y \mid x^ny^{-2}, R \rangle$ be a regular presentation for $Q_{4n}$ for some $R \in F(x,y)$. Let $R = x^{n_1}y^{m_1} \cdots x^{n_k}y^{m_k}$ for some $n_i,m_i \in \Z$ and some $k \ge 1$. By \cref{lemma:regular-moves} (ii), we can replace any $y^2 \subseteq R$ with $x^n \subseteq R$. If $m_i=2m$ is even, then replace $y^{m_i}$ with $x^{nm_i}$. If $m_i=2m+1$ is odd, then replace $y^{m_i}$ with $x^{nm_i}y$. After relabelling the $n_i$ and $k$, this gives that $R \simeq x^{n_1}y \cdots x^{n_k} y$.

Since $\mathcal{P}$ is regular, we have that $\id : \langle x, y \mid x^ny^{-2}, R \rangle \to Q_{4n}$ is a group isomorphism. In particular, $x^{n_1}y \cdots x^{n_k} y=1$ holds in $Q_{4n}$. It follows that $x^{\sum_{i=1}^k (-1)^{i+1} n_i} y^k=1$ and so $k$ must be even. We can therefore write $R$ in the form
\begin{align*} 
	R &= (x^{n_1}yx^{m_1}y) \cdots (x^{n_k}yx^{m_k}y) \simeq x^{nk}(x^{n_0}yx^{m_1}y^{-1}) \cdots (x^{n_k}yx^{m_k}y^{-1}) \\
	&= x^{n_1+nk}(yxy^{-1})^{m_1} x^{n_2}(yxy^{-1})^{m_2} \cdots x^{n_k}(yxy^{-1})^{m_k}.
\end{align*}
By relabelling the $n_i$ and $m_i$, we have $R=S(x,yxy^{-1})$ for $S = a^{n_1}b^{m_1} \cdots a^{n_k}b^{m_k} \in F(a,b)$. 	
\end{proof}

We now give operations on the words $R \in F(a,b)$ which preserve the $Q$-type of presentations of the form given in \cref{prop:general-presentations-Q4n}. They each follow from \cref{lemma:regular-moves}.
Moves (i), (ii) and (iv) will be used in the proof but (iii) and (v) will be stated here for later use (see \cref{lemma:operations}).

\begin{lemma} \label{lemma:proxy-moves}
Let $\mathcal{P} = \langle x, y \mid x^ny^{-2}, R(x,yxy^{-1}) \rangle$ be a presentation for $Q_{4n}$ for $R \in F(a,b)$ and $R = a^{n_1}b^{m_1} \cdots a^{n_k}b^{m_k}$ for some $n_i,m_i \in \Z$ and $k \ge 1$.
Suppose $R' \in F(a,b)$ is the result of one of the following operations of $R$:
\begin{clist}{(i)}
\item
For some $1 \le i, j \le k$, $i \ne j$, replace $n_i$ with $n_i \pm n$ and $n_j$ with $n_j \mp n$. 
\item
For some $1 \le i, j \le k$, replace $n_i$ with $n_i \pm n$ and $m_j$ with $m_j \mp n$.
\item
Cyclically permuting the pairs $(n_i,m_i)$, i.e. $R' = a^{n_2}b^{m_2} \cdots a^{n_k}b^{m_k}a^{n_1}b^{m_1}$.
\item
Inversion followed by a cyclic permutation, i.e. $R' = a^{-n_1}b^{-m_k}a^{-n_k}b^{-m_{k-1}} \cdots a^{-n_2}b^{-m_1}$.
\item
Regrouping of terms, i.e. $R' = a^{m_1}b^{n_2}a^{m_2}b^{n_3}\cdots a^{m_{k-1}} b^{n_k}a^{m_k}b^{n_1}$.
\end{clist}
Then $\mathcal{P} \simeq_{Q} \langle x, y \mid x^ny^{-2}, R'(x,yxy^{-1}) \rangle$.
\end{lemma}

\begin{proof}[Proof of \cref{thm:general-presentations-Q4n}]
We will write $R \simeq R'$ for $R,R' \in F(a,b)$ if 
\[ \langle x, y \mid x^ny^{-2}, R(x,yxy^{-1}) \rangle \simeq_{Q} \langle x, y \mid x^ny^{-2}, R'(x,yxy^{-1}) \rangle.\]
For convenience, we will continue referring to the relator as $R$ after the presentation has been altered by a $Q$-transformation.

By \cref{prop:general-presentations-Q4n}, we have that $\mathcal{P} \simeq_{Q} \langle x, y \mid x^ny^{-2}, R(x,yxy^{-1}) \rangle$ where $R \in F(a,b)$ and $R = a^{n_1}b^{m_1} \cdots a^{n_{k+1}}b^{m_{k+1}}$ for some $n_i,m_i \in \Z$ and $k \ge 0$. Since $\mathcal{P}$ is a regular presentation, $\id : \mathcal{P} \to Q_{4n}$ is a group isomorphism and so $R=1$ holds in $Q_{4n}$. It follows that $x^{\sum_{i=1}^{k+1} n_i-m_i}=1$ and so $\sum_{i=1}^{k+1} n_i \equiv \sum_{i=1}^{k+1} m_i \mod 2n$.
We can now apply moves as in \cref{lemma:proxy-moves} (ii) to get that $\sum_{i=1}^{k+1} n_i = \sum_{i=1}^{k+1} m_i$. To see this, let $\sum_{i=1}^{k+1} n_i = \sum_{i=1}^{k+1} m_i + 2nr$ for some $r \in \Z$. It then takes $|r|$ moves to replace $n_1$ with $n_1-nr$ and $m_1$ with $m_1+nr$.

We will now show that, given that $\mathcal{P}$ presents $Q_{4n}$, we must have $\sum_{i=1}^{k+1} n_i = \sum_{i=1}^{k+1} m_i = \pm 1$. Let $t = \sum_{i=1}^{k+1} n_i = \sum_{i=1}^{k+1} m_i$. We compute the abelianisation
\[ \mathcal{P}^{\ab} \cong_{\gp} \langle x,y \mid x^ny^{-2}, x^{2t} , [x,y] \rangle \cong_{\gp} \coker(\Z^2 \xrightarrow[]{\left(\begin{smallmatrix} 2t & n \\0 & -2 \end{smallmatrix}\right)} \Z^2) \cong_{\gp} \begin{cases} \Z/2|t| \times \Z/2 & \text{if $n$ is even,} \\ \Z/4|t|  & \text{if $n$ is odd,} \end{cases}  \]
where $\cong_{\gp}$ denotes group isomorphism and the last isomorphism comes from putting the matrix into Smith normal form. It is well known that $Q_{4n}^{\ab} \cong (\Z/2)^2$ for $n$ even, and $Q_{4n}^{\ab} \cong \Z/4$ for $n$ odd. Hence we have that $|t|=1$ and so $t= \pm 1$. In fact, if $t=-1$, then apply inversion as in \cref{lemma:proxy-moves} (iv) to get $t=1$. Hence we can assume that $\sum_{i=1}^{k+1} n_i = \sum_{i=1}^{k+1} m_i = 1$.

Next note that if there exists $i$ such that $n_i = 0$, then we can just write
\[ R = a^{n_1} \cdots a^{n_{i-1}}b^{m_{i-1}+m_i}a^{n_{i+1}} \cdots b^{m_{k+1}}\]
and similarly if there exists $i$ such that $m_i=0$. By induction, this process will end in either (a) $R = a^{n_1}b^{m_1} \cdots a^{n_{k+1}}b^{m_{k+1}}$ where $k \ge 0$, $n_i, m_i \ne 0$, or (b) one of the cases $R=1$, $R=a^{n_1}$ for $n_1 \ne 0$, $R=b^{m_1}$ for $m_1 \ne 0$. Suppose for contradiction that case (b) arises. After reduction we have that $R=x^m$ for some $m \in \Z$. This implies that $Q_{4n} \cong \langle x,y \mid x^ny^{-2}, x^{m} \rangle$. Let $d = \gcd(n,m)$. If $d = 1$, then there exists $a, b \in \Z$ such that $an+bm=1$ and so $x = (x^n)^a (x^{m})^b = y^{2a}$. In particular, $Q_{4n} = \langle y \rangle$ is cyclic, which is a contradiction. If $d \ne 1$, then $Q_{4n} / \langle y^2 \rangle \cong \Z/d \ast \Z/2$ is infinite, which is a contradiction.
Hence the process will terminate in case (a), which is the required presentation.
\end{proof}

\subsection{Properties of regular presentations} 
\label{ss:regular-pres-properties}

We are interested in the following two basic tasks:
\begin{clist}{(1)}
\item Find conditions on $n, k ,k'$ and $n_i, m_i, n_i', m_i'$, which determine when
\[ \mathcal{P}_n(n_1,\dotsc,n_{k};m_1,\dotsc,m_{k}) \simeq_{Q} \mathcal{P}_n(n_1',\dotsc,n_{k'}';m_1',\dotsc,m_{k'}'). \]
\item Find conditions on $n_i, m_i$ such that $\mathcal{P}_n(n_1,\dotsc,n_{k};m_1,\dotsc,m_{k})$ is a presentation for $Q_{4n}$.
\end{clist}

First note that we have the following operations on this class of presentations. This follows directly from the operations listed in \cref{lemma:proxy-moves}.

\begin{lemma} \label{lemma:operations}
Let $n \ge 2$, let $k \ge 1$ and let $n_i, m_i \in \Z \setminus \{0\}$ for $1 \le i \le k$.
\begin{clist}{(I)}
\item If $n_i \equiv n_i' \mod n$ and $m_i \equiv m_i' \mod n$, then:
\[ \mathcal{P}_n(n_1,\dotsc,n_{k};m_1,\dotsc,m_{k}) \simeq_{Q} \mathcal{P}_n(n_1',\dotsc,n_{k}';m_1',\dotsc,m_{k}').\]
\item If $n_{k+1} = 1 - \sum_{i=1}^{k} n_i$ and $m_{k+1} = 1 - \sum_{i=1}^{k} m_i$, then:
\[ \mathcal{P}_n(n_1,\dotsc,n_{k};m_1,\dotsc,m_{k}) \simeq_{Q} \mathcal{P}_n(n_2,\dotsc,n_{k+1};m_2,\dotsc,m_{k+1}).\]
\item If $n_{k+1} = 1 - \sum_{i=1}^{k} n_i$, then:
\[ \mathcal{P}_n(n_1,\dotsc,n_{k};m_1,\dotsc,m_{k}) \simeq_{Q} \mathcal{P}_n(m_1,\dotsc,m_{k};n_2,\dotsc ,n_{k+1}).\]
\end{clist}
\end{lemma}

This has the following consequence for presentations of height one, which will be used in the proof of \cref{thmx:a-non-MP-presentation}.

\begin{corollary}\label{corollary:heightone}
  Let $\mathcal P = \mathcal P_n(n_1;m_1)$ be a presentation of height one. Then:
 \begin{clist}{(i)}
\item There exist $n_1', m_1'$ with
  $1 \leq n_1', m_1' \leq \lceil n/2 \rceil$ and $\mathcal P \simeq_Q \mathcal P_n(n_1';m_1')$.
\item 
If $n_1=1$ or $m_1=1$, then $\mathcal P_n(n_1;m_1) \simeq_Q \mathcal P_n^{\text{\normalfont std}}$.
\end{clist}  
\end{corollary}

\begin{proof}
(i)  By applying~(I) we may freely alter the values of $n_1$ and $m_1$ mod $n$. We will assume that $1 \leq n_1, m_1 \leq n - 1$. By successively applying operation (III), we get that
  \[ \mathcal P_n(n_1;m_1) \stackrel{\text{(III)}}{\simeq_Q} \mathcal P_n(m_1;1 - n_1+n)  \stackrel{\text{(III)}}{\simeq_Q} \mathcal P_n(1 - n_1+n;1 - m_1+n)  \stackrel{\text{(III)}}{\simeq_Q} \mathcal P_n(1 - m_1+n;n_1). \]
Note that (III)$^2 = $ (II) and so no further presentations can be obtained by applying (II).

If $a \in \Z$ with $1 \le a \le n-1$, then $2 \le 1-a+n \le n$ and $a+(1-a+n) = n+1$ implies that $\min\{a,1-a+n\} \le \lfloor \frac{n+1}{2} \rfloor = \lceil \frac{n}{2} \rceil$.
In particular, since $1 \le n_1,m_1 \le n-1$, there exists $n_1' \in \{n_1,1-n_1+n\}$, $m_1' \in \{m_1,1-m_1+n\}$ such that $1 \le n_1',m_1' \le \lceil n/2 \rceil$. Hence, by choosing the appropriate number of operations (III), we get that $\mathcal P_n(n_1;m_1) \simeq_Q \mathcal P_n(n_1';m_1')$ for some $n_1', m_1'$ such that $1 \le n_1',m_1' \le \lceil n/2 \rceil$.

(ii) This is clear from the fact that, if $n_1=1$ or $m_1=1$, then $n_2=0$ or $m_2=0$. In either case, the relator $R = a^{n_1}b^{m_1}a^{n_2}b^{m_2}$ can be cyclically permuted to give $ab$.
\end{proof}

We now deal with the question of when $\mathcal{P}_n(n_1,\dotsc,n_{k};m_1,\dotsc,m_{k})$ presents $Q_{4n}$.

\begin{prop}\label{prop:pnproperties}
    Let $n \ge 2$, let $k \ge 1$ and let $n_i, m_i \in \Z \setminus \{0\}$ for $1 \le i \le k$. Let $G$ denote the group described by $\mathcal{P}_n(n_1,\dotsc,n_{k};m_1,\dotsc,m_{k})$. Then:
\begin{clist}{(i)}
\item
There is a group isomorphism $G/\langle xyxy^{-1} \rangle \cong Q_{4n}$ where $x \mapsto x$, $y \mapsto y$.
\item
$G^{\ab} \cong Q_{4n}^{\ab}$.
\end{clist}
\end{prop}

\begin{proof}
(i) Let $a=x$ and $b=yxy^{-1}$. Then we have:
\begin{align*}
G/ \langle ab \rangle & \cong \langle x,y \mid x^ny^{-2}, ab, a^{n_1}b^{m_1} \cdots a^{n_{k+1}} b^{m_{k+1}} \rangle \\
& \cong \langle x,y \mid x^ny^{-2}, ab, a^{\sum n_i - \sum m_i} \rangle \cong \langle x, y \mid x^ny^{-2}, xyxy^{-1} \rangle \cong Q_{4n}.
\end{align*}

(ii) This follows from the argument given in the proof of \cref{thm:general-presentations-Q4n}.
\end{proof}

\begin{corollary} \label{cor:pn-standard-gen}
Let $n \ge 2$, let $k \ge 1$ and let $n_i, m_i \in \Z \setminus \{0\}$ for $1 \le i \le k$. Then if $\mathcal{P}_n(n_1,\dotsc,n_{k};m_1,\dotsc,m_{k})$ presents $Q_{4n}$ then it does so on the standard generating set, i.e. $\id : \mathcal{P}_n(n_1,\dotsc,n_{k};m_1,\dotsc,m_{k}) \to Q_{4n}$, $x \mapsto x$, $y \mapsto y$ is a group isomorphism.
\end{corollary}

\begin{proof}
This follows from \cref{prop:pnproperties} (i) since it implies that there exists a surjective homomorphism $f : \mathcal{P}_n(n_1,\dotsc,n_{k};m_1,\dotsc,m_{k}) \to Q_{4n}$, $x \mapsto x$, $y \mapsto y$ and this map must be a bijection since both groups a finite and have the same order.    
\end{proof}

The following is a broad generalisation of the example $\mathcal{E}_{n,3} = \mathcal{P}_n(2;-2)$ of Mannan--Popiel.

\begin{theorem} \label{thm:even-coeffs}
Let $n \ge 2$, let $k \ge 0$ and let $1 \le n_1, \dotsc, n_k \le 2n-2$ be even integers. Then $\mathcal{P}_n(n_1,\dotsc,n_k;-n_k,\dotsc,-n_1)$ is a regular presentation for $Q_{4n}$.
\end{theorem}

By \cref{lemma:operations}, we lose no generality in assuming that $1 \le n_i \le 2n-2$. If $n$ is even, then we can assume further that $1 \le n_i \le n-2$.

\begin{proof}
Let $n_0 = 1 - \sum_{i=1}^{k} n_i$, let $a=x$ and $b=yxy^{-1}$. Let $G$ denote the group described by $\mathcal{P}_n(n_1,\dotsc,n_k;-n_k,\dotsc,-n_1)$.
By \cref{prop:pnproperties}~(i), it suffices to show that $ab=1$ holds in $G$. Our strategy will be to use regrouping as described in \cref{lemma:proxy-moves} (v). That is, for any $s_i, t_i \in \Z$ and $m \ge 1$, we have that $a^{s_1}b^{t_1} \cdots a^{s_m}b^{t_m}=1$ implies $b^{s_1}a^{t_1} \cdots b^{s_m}a^{t_m}=1$ in $G$.

Recall that our presentation is defined by taking $m_i = -n_{k-i+1}$ for each $1 \le i \le k$ in the definition of $\mathcal{P}_n(n_1,\dotsc,n_{k};m_1,\dotsc,m_{k})$. Note that $m_{k+1} = 1 - \sum_{i=1}^{k-1} m_i = 2-n_0$. In $G$, we therefore have that $R = a^{n_1}b^{-n_k} \cdots a^{n_k} b^{-n_1}a^{n_0}b^{2-n_0} = 1$.
By regrouping, we get that $b^{n_1}a^{-n_k} \cdots b^{n_k} a^{-n_1}b^{n_0}a^{2-n_0} = 1$. Cyclically permuting then gives that
\[ b^{n_0}a^{2-n_0}b^{n_1}a^{-n_k} \cdots b^{n_k} a^{-n_1} = 1.\] 
We can now multiply this on the right by $R$ to get:
\[ 1 = (b^{n_0}a^{2-n_0}b^{n_1}a^{-n_k} \cdots b^{n_k} a^{-n_1}) \cdot (a^{n_1}b^{-n_k} \cdots a^{n_k} b^{-n_1}a^{n_0}b^{2-n_0}) = b^{n_0} a^2b^{2-n_0} \] 
which implies that $a^2b^2=1$, i.e. $a^2 = b^{-2}$.

For $1 \le i \le k$, $n_i$ is even and so this implies that $a^{n_i} = b^{-n_i}$. Since $n_0 = 1 - \sum_{i=1}^k n_i$, $n_0$ is odd and so $a^{n_0-1} = b^{1-n_0}$. Substituting this into $R$ gives that $a^{2\sum_{i=1}^k n_i}a^{2n_0-1}b=1$ which reduces further to $ab=1$. This completes the proof.
\end{proof}

\begin{example}
Let $k \ge 0$. For each $1 \le i \le 2k$, let $n_i = 2$ for $i$ odd and $n_i = -2$ for $i$ even. This gives the presentations:	
\[ \mathcal{P}_{n,k} = \mathcal{P}_n(\underbrace{2,-2, \dotsc, 2,-2}_{2k};\underbrace{2,-2, \dotsc, 2,-2}_{2k}) = \langle x, y \mid x^ny^{-2}, (x^2yx^2y^{-1}x^{-2}yx^{-2}y^{-1})^kxyxy^{-1} \rangle. \]
Note that $\mathcal{P}_{n,0} = \mathcal{P}_n^{\text{std}}$ is the standard presentation.
\end{example}

\begin{example}
Let $k \ge 0$. For each $1 \le i \le k$, take $n_i = 2$. This gives the presentations:
\[ \mathcal{P}_{n,k}' = \mathcal{P}_n(\underbrace{2, \dotsc, 2}_{k};\underbrace{-2, \dotsc, -2}_{k}) = \langle x, y \mid x^ny^{-2}, (x^2yx^{-2}y^{-1})^kx^{1-2k}yx^{1+2k}y^{-1} \rangle. \]
Note that $\mathcal{P}_{n,0}' = \mathcal{P}_n^{\text{std}}$, and $\mathcal{P}_{n,1}' = \mathcal{P}_n(2;-2) = \mathcal{E}_{n,3}$ is the Mannan--Popiel presentation.
\end{example}

We next give conditions under which $\mathcal{P}_n(n_1,\dotsc,n_{k};m_1,\dotsc,m_{k})$ does not present $Q_{4n}$.
The following will be proven in \cref{ss:specific-examples}.
The case $\mathcal{E}_{n,r}$ was conjectured in \cite[p2071]{MP21}. 

\begin{prop} \label{prop:MP-not-presentation}
If $r \equiv 2 \mod 3$ and $3 \mid n$, then $\mathcal{E}_{n,r} := \mathcal{P}_n(2;1-r)$ and $\mathcal{P}_n(r;r)$ are not regular presentations for $Q_{4n}$. 
\end{prop}

\begin{remark}
It is possible to use Magma \cite{magma} to determine whether a presentation presents $Q_{4n}$ for a specified integer $n \ge 2$. Such computations suggest that the presentations listed above are the only Mannan--Popiel presentations which do not present $Q_{4n}$, i.e. $\mathcal{E}_{n,r}$ presents $Q_{4n}$ if and only if $r \not \equiv 2 \mod 3$ or $3 \nmid n$. We verified this computationally for $n \le 25$.
\end{remark}

We conclude by remarking that we only considered presentations $\langle x, y \mid x^ny^{-2}, R \rangle$ which present $Q_{4n}$ on the standard generating set. It is not clear whether or not presentations of this form exist which do not have this property. However, we can show the following.

\begin{lemma} \label{lemma:pres-identification}
Let $\mathcal{P} = \langle x, y \mid x^ny^{-2}, R \rangle$ be a presentation for $Q_{4n}$ for some $R \in F(x,y)$. Then at least one of the following holds:
\begin{clist}{(i)}
\item
$\id : \langle x, y \mid x^ny^{-2}, R \rangle \to Q_{4n}$, $x \mapsto x, y \mapsto y$ is a group isomorphism
\item
$\varphi: \langle x, y \mid x^ny^{-2}, R \rangle \to Q_{4n}$, $x \mapsto xy, y \mapsto y$ is a group isomorphism.
\end{clist}
Furthermore, if $n=2$ or $n \not \equiv 2 \mod 4$, then {\normalfont (i)} holds, i.e. $\mathcal{P}$ is a regular presentation for $Q_{4n}$.
\end{lemma}

\begin{proof}
We will begin by stating some well-known facts about $Q_{4n}$ (see \cite[Section 5.8]{MR4559716}, \cite[Section 1]{MR2039764}). The elements of $Q_{4n}$ are $\{x^iy^j : i \in \Z/2n, j \in \Z/2 \}$. For each $i \in \Z/2n$, the elements $x^i$ have order $\gcd(i,2n)$ and the elements $x^i y$ have order $4$. The centre is $Z(Q_{4n}) = \{1, x^n\}$.
The generating pairs for $Q_{4n}$ are $\{x^i, x^jy\}$ and $\{x^iy, x^{i+j}y\}$ for $i \in (\Z/2n)^\times$, $j \in \Z/2n$.

For $n =2$, the automorphism group is given by:
\[	\Aut(Q_8) = \{\theta_{i,j} : x \mapsto x^i, y \mapsto x^jy\} \cup \{\theta_{i,j}' : x \mapsto x^jy, y \mapsto x^i\}\cup \{\theta_{i,j}'' : x \mapsto x^jy, y \mapsto x^{i+j}y\}\cong S_4 \\
\]
where $i \in \{\pm 1\}, j \in \Z/4$ in each set. For $n \ge 3$, we have:
\[\Aut(Q_{4n}) = \{\theta_{i,j} : x \mapsto x^i, y \mapsto x^jy : i \in (\Z/2n)^\times, j \in \Z/2n \} \cong \Z/2n \rtimes (\Z/2n)^\times.
\]

Since $\mathcal{P}$ is a presentation for $Q_{4n}$, there is a group isomorphism  $f:\langle x, y \mid x^ny^{-2}, R \rangle \to Q_{4n}$.
Let $a = f(x)$ and $b=f(y)$. These elements have the property that $a^n=y^2$ and $\langle a,b \rangle = Q_{4n}$.

If $n=2$, then there exists $\theta \in \Aut(Q_8)$ such that $\theta(x) = a$ and $\theta(y) = b$. In particular, $\id = \theta^{-1} \circ f : \langle x, y \mid x^ny^{-2}, R \rangle \to Q_{4n}$ is a group isomorphism and so (i) holds.

Now suppose $n \ge 3$. Note that $a^n=b^2$ is non-zero and central since it commutes with $a,b$. This implies that $a^n=b^2 = x^n=y^2 \in Q_{4n}$. We have $b^4=1$, $b^2 \ne 1$ and so $b$ has order $4$. This implies that either $b = x^{i}$ where $\gcd(i,2n)=4$ or $b=x^jy$ for some $j \in \Z/2n$. In the first case, $b=x^i$ cannot be part of a generating pair since $\gcd(i,2n) \ne 1$. Hence we have $b=x^jy$ for some $j \in \Z/2n$. By replacing $f$ with $\theta_{1,j}^{-1} \circ f$ and relabelling $a$ and $b$, we can now assume that $b=y$. It remains to determine $a$ which satisfies $a^n=x^n$ and $\langle a, y \rangle = Q_{4n}$. The second condition implies that $a = x^i$ or $x^iy$ for some $i \in (\Z/2n)^\times$. By replacing $f$ with $\theta_{i,0}^{-1} \circ f$ and relabelling $a$ and $b$, we get that $a = x$ or $xy$ and $b=y$.
If $a=xy$ then we additionally require that $(xy)^n=x^n$. We have $(xy)^2 = y^2=x^n$ and so this is equivalent to $(xy)^{n-2}=1$. Since $xy \in Q_{4n}$ has order 4, this holds if and only if $n \equiv 2 \mod 4$. Hence, if $n \not \equiv 2 \mod 4$, then (i) holds.
\end{proof}

\subsection{The second homotopy group of a regular presentation} \label{ss:pi_2-regular-presentation}

Recall from \cref{ss:D2-prelim} that, if $G$ is a group with free period 4 (of which $Q_{4n}$ for $n \ge 2$ are examples), then there is a map $\wh \Psi : \HT(G) \to \SF(\Z G)/\sim$, where $P \sim Q$ if there exist ideals $I, J \le \Z G$ and $n \ge 0$ such that $P \cong I \oplus \Z G^n$, $Q \cong J \oplus \Z G^n$ and $I \cong (N,r) \otimes J_\theta$ for some $r \in (\Z/|G|)^\times$, $\theta \in \Aut(G)$.

From now on, let $n \ge 2$, let $k \ge 0$, let $n_1,m_1,\dots,n_k,m_k \ne 0$, and let $R = a^{n_1}b^{m_1} \cdots a^{n_{k+1}}b^{m_{k+1}}$ for $n_{k+1} = 1 - \sum_{i=1}^{k} n_i$ and $m_{k+1} = 1 - \sum_{i=1}^{k} m_i$. 
Let $\mathcal{P}_R = \langle x, y \mid x^ny^{-2}, R(x,yxy^{-1})\rangle$, which coincides with $\mathcal{P}_n(n_1,\dotsc,n_{k};m_1,\dotsc,m_{k})$ defined previously. We will write $G = Q_{4n}$.
Define $\lambda_R = \sum_{i=1}^{k+1} x^{\sum_{j=1}^{i-1} (n_j-m_j)}\cdot \partial_x(x^{n_i}) \in \Z[\langle x \rangle] \le \Z G$ where $\partial_x$ denotes Fox differentiation and
\[ \partial_x(x^{r}) =\frac{x^r-1}{x-1} =
\begin{cases}
1+x+\dotsb+x^{r-1} & \text{if $r >0$,} \\
0 & \text{if $r=0$,} \\
-x^r(1+x+\dotsb+x^{|r|-1}) & \text{if $r < 0$.}
\end{cases}
\]

The aim of this section will be to establish the following theorem. 

\begin{theorem} \label{thm:Psi(P_R)}
Suppose $\mathcal{P}_R$ is a regular presentation for $G$. Then there exists $\alpha \in \Z G$ such that $\alpha \cdot (xy-1) \in 1+\Z G \cdot \lambda_R$ and, for any such $\alpha$, we have
\[ \wh \Psi(\mathcal{P}_R) = [I] \in \SF(\Z G)/\sim \]
where $I$ is the ideal $(1+(1-xy)\alpha, \lambda_R \alpha) \le \Z G$.
\end{theorem}

Since $xyxy^{-1}=1$ implies $R(x,yxy^{-1}) = 1$, there exists $\lambda \in \Z G$ for which
\[ (\partial_x R(x,yxy^{-1}), \partial_y R(x,yxy^{-1})) = \lambda \cdot (\partial_x(xyxy^{-1}), \partial_y(xyxy^{-1})) = \lambda \cdot (1+xy, x-1) \in \Z G^2. \]
The choice of $\lambda \in \Z G$ is not unique since we can add on multiples of $\lambda_0 = (\sum_{i=0}^{2n-1} x^i) \cdot (1-xy)$. 

\begin{lemma} \label{lemma:compute-lambda}
We have $(\partial_x R(x,yxy^{-1}), \partial_y R(x,yxy^{-1})) = \lambda_R \cdot (1+xy, x-1) \in \Z G^2$. Furthermore, $\lambda_R$ is the unique element of $\Z[\langle x \rangle] \le \Z G$ with this property.
\end{lemma}

\begin{proof}
Let $f : F(a,b) \to F(x,y)$, $a \mapsto x$, $b \mapsto yxy^{-1}$. By \cref{prop:chain-rule}, we have 
\[ \partial R(x,yxy^{-1}) = \partial f(R(a,b)) = f(\partial_a R(a,b)) \partial f(a) + f(\partial_b R(a,b)) \partial f(b) \]
where $\partial = (\partial_x,\partial_y)$. We have $\partial f(a) = (1,0)$ and $\partial f(b) = (y,1-x^{-1})$. 

Since $ab=1$ implies $R(a,b)=1$, we have that $\partial ab \mid \partial R(a,b)$ in $\Z[F(a,b)/ab]$, where $\partial = (\partial_a,\partial_b)$. Since $\partial ab = (1,a)$, it follows that $\partial_b R(a,b) = \partial_a R(a,b) a \in \Z[F(a,b)/ab]$. Since $f(ab)\mapsto 1 \in G$, this implies that $f(\partial_b R(a,b)) = f(\partial_a R(a,b)) x \in \Z G$, and so
\[ \partial R(x,yxy^{-1}) = f(\partial_a R(a,b)) \cdot [(1,0)+x(y,1-x^{-1})] = f(\partial_a R(a,b)) (1+xy,x-1). \]

By \cref{prop:Fox-formula}, we have
\[ \partial_a R(a,b) = \sum_{i=1}^{k+1} a^{n_1}b^{m_1} \cdots a^{n_{i-1}}b^{m_{i-1}} \partial_a(a^{n_i}) = \sum_{i=1}^{k+1} a^{\sum_{j=1}^{i-1} (n_j-m_j)} \cdot \partial_a(a^{n_i}) \in \Z[F(a,b)/ab] \]
which implies that $f(\partial_a R(a,b)) = \lambda_R$, as required.

If $\lambda_1, \lambda_2 \in \Z[\langle x \rangle]$ satisfy this equation, then $\lambda \cdot (1+xy, x-1) \in \Z G^2$ where $\lambda := \lambda_1 - \lambda_2 \in \Z[\langle x \rangle]$. In particular, $\lambda + (x\lambda)y = 0 \in  \Z G$ which implies that $\lambda=0$ and so $\lambda_1 = \lambda_2$.
\end{proof}

We will often write $\lambda_R = \lambda_n(n_1,\dotsc,n_{k};m_1,\dotsc,m_{k})$. 
We have
\[ \lambda_n(n_1,\dotsc,n_{k};-n_k,\dotsc,-n_1) = \sum_{i=1}^k x^{\sum_{j=1}^{i-1} (n_j+n_{k+1-j})} \sigma(n_i) - x^{1+\sum_{i=1}^k n_i}\sigma(\sum_{i=1}^k n_i - 1) \]
where, for $r > 0$, we let $\sigma(r) = 1+ x + \dotsb + x^{r-1}$ and $\sigma(0)=0$.
For $k=1$ and $n_1 > 0$, we have $\lambda_n(n_1;m_1) = \sigma(n_1) - x^{1-m_1}\sigma(n_1-1)$. For example, $\lambda_n(2;1-r) = 1+x-x^r$.

\begin{lemma} \label{lemma:compute-pi_2}
Suppose $\mathcal{P}_R$ is a regular presentation for $G$. Then:
\begin{clist}{(i)}
\item There exist $\alpha, \beta \in \Z G$ such that $\alpha \cdot (xy-1) + \beta \cdot \lambda_R = 1$.
\item 
$\lambda_R \in \Q G^\times$. In particular, $\lambda_R \in \Z G$ is not a zero divisor.
\item 
Let $\alpha \in \Z G$ be any element as in {\normalfont (i)}. Then there is an injective $\Z G$-module homomorphism
$i_R : \pi_2(X_{\mathcal{P}_R}) \to \Z G/N$
whose image is the ideal $(1+(1-xy)\alpha, \lambda_R\alpha) \le \Z G/N$.
\end{clist}
\end{lemma}

\begin{proof}
(i) Consider the standard presentation $\mathcal{P}_n^{\std} = \langle x,y \mid x^ny^{-2}, xyxy^{-1} \rangle$. Let $\partial_i$, $\partial_i'$ denote the boundary maps in $C(\wt X_{\mathcal{P}_n^{\std}})$ and $C(\wt X_{\mathcal{P}_R})$ respectively. By Fox differentiation, we have
\begin{equation} \partial_2 = \cdot \left( \begin{smallmatrix} 1+x+\cdots+x^{n-1} & -(1+y) \\ 1+xy & x-1 \end{smallmatrix} \right), \quad \partial_2' = \cdot \left( \begin{smallmatrix} 1+x+\cdots+x^{n-1} & -(1+y) \\ \lambda_R \cdot (1+xy) & \lambda_R \cdot (x-1) \end{smallmatrix} \right). \label{eq:d2-d2'-def} \end{equation}
By \cite[p253]{CE56}, we have that $\ker(\partial_2) = \Z G \cdot (x-1,1-xy)$.

Since $\mathcal{P}_R$ is regular, $\mathcal{P}_n^{\std}$ and $\mathcal{P}_R$ are both presentations for $G$ on the same generating set. This implies that $\ker(\partial_1)=\ker(\partial_1')$ and so, by exactness, $\IM(\partial_2) = \IM(\partial_2')$. This implies that there exists $\gamma,\beta \in \Z G$ such that $(1+xy,x-1) = \partial_2'(\gamma,\beta)$. Since $\partial_2'(\gamma,\beta) = \partial_2(\gamma,\beta\cdot \lambda_R)$, this rearranges to $\partial_2(\gamma,\beta \cdot \lambda_R -1)=0$ and so $(\gamma,\beta \cdot \lambda_R-1) \in \ker(\partial_2)$. Thus there exists $\alpha \in \Z G$ such that $ (\gamma,\beta \cdot \lambda_R -1) = \alpha \cdot (x-1,1-xy)$. In particular, $\alpha \cdot (xy-1) + \beta \cdot \lambda_R =1$ and $\gamma = \alpha \cdot (x-1)$.

(ii) By \cref{lemma:compute-lambda}, we have that $\lambda_R \in \Z C_{2n} \le \Q C_{2n}$ where $C_{2n} = \langle x \rangle \le Q_{4n}$.
There is an isomorphism of rings $\Q C_{2n} \cong \prod_{d \mid 2n} \Q(\zeta_d)$ where $\zeta_d$ denotes a primitive $d$th root of unity. Let $f_d : \Q C_{2n} \twoheadrightarrow \Q(\zeta_d)$ denote the map induced by $x \mapsto \zeta_d$. Since $\Q(\zeta_d)^\times = \Q(\zeta_d) \setminus \{0\}$, we obtain
\[ \Q C_{2n}^\times = \{ \alpha \in \Q C_{2n} : f_d(\alpha) \ne 0 \text{ for all $d \mid 2n$}\}. \]
It therefore remains to show that $f_d(\lambda_R) \ne 0$ for all $d \mid 2n$.

Let $\Z[\zeta_d,y] = (\Z[\zeta_d])\{1,y\}$ denote the cyclic algebra such that $y^2 = 1$ if $d \mid n$ and $y^2 = -1$ otherwise.
Then there is an isomorphism of rings $\Q G \cong \prod_{d \mid 2n} \Q[\zeta_d,y]$ where $\Q[\zeta_d,y] = \Q \otimes \Z[\zeta_d,y]$ \cite[p75]{Sw83}.
Let $F_d : \Q G \twoheadrightarrow \Q[\zeta_d,y]$ denote the map $x \mapsto \zeta_d$, $y \mapsto y$ and note that $F_d |_{\Q C_{2n}} = f_d$ and $F_d(\Z G) = \Z[\zeta_d,j]$.

Suppose for contradiction that $f_d(\lambda_R) = 0$ for some $d \mid 2n$. By the proof of (i), we have that $(1+xy,x-1) = \partial_2'(\gamma,\beta)$. 
Applying $F_d$ to this equation gives:
\begin{equation} (1+\zeta_dy,\zeta_d-1) = F_d(\gamma) \cdot (1+\zeta_d+\cdots+\zeta_d^{n-1},-(1+y)) \in \Z[\zeta_d,y]^2 \label{eq:d2=d2'} \end{equation}
where, since $\gamma \in \Z G$, we have $F_d(\gamma) \in \Z[\zeta_d,y]$. The second coordinate gives that $\zeta_d-1 = \delta (1+y)$ for some $\delta \in \Z[\zeta_d,y]$. If $\delta = a + by$ for some $a, b\in \Z[\zeta_d]$, this implies that $\zeta_d-1 = (a \pm b) + (a + b)y$ 
where the sign $\pm$ depends on whether or not $d \mid n$. Comparing coefficients of $y$ gives that $b=-a$ and so $\zeta_d-1 = a \pm (-a)$. If $d \mid n$, this gives that $\zeta_d-1=0$ which is a contradiction. If $d \nmid n$, this gives that $\zeta_d-1 = 2a$ and so $\frac{1}{2}(\zeta_d-1) \in \Z[\zeta_d]$ which is a contradiction. Hence $f_d(\lambda_R) \ne 0$, which completes the proof.
Note that, in the case $d \mid n$, we could have immediately obtained a contradiction from the first coordinate of \cref{eq:d2=d2'} since $1+\zeta_d+\cdots+\zeta_d^{n-1}=0$ and $1+\zeta_dy \ne 0$.

(iii) Recall that $\pi_2(X_{\mathcal{P}_R}) = \ker(\partial_2')$. Define $f : \ker(\partial_2') \to \ker(\partial_2)$, $(x_1,x_2) \mapsto (x_1,x_2\lambda_R)$.
By (ii), $\lambda_R$ is not a zero divisor and so $f$ is injective.
By \cite[p253]{CE56}, there is an isomorphism $g : \Z G/N \to \ker(\partial_2)$, $1 \mapsto (x-1,1-xy)$. We now define $i_R$ to be the composition
\[ i_R : \pi_2(X_{\mathcal{P}_R}) = \ker(\partial_2') \xrightarrow[]{f} \ker(\partial_2) \xrightarrow[]{g^{-1}} \Z G/N \]
which is injective since $f$ and $g^{-1}$ are injective.

It follows directly from the definition of $\gamma$, $\beta$ given in (i) that $\ker(\partial_2') = \Z G \cdot A + \Z G \cdot B$ where
\begin{equation} A =[(x-1)+(1-xy)\gamma,(1-xy)\beta], \quad B = [-\lambda_R \gamma,1-\lambda_R\beta]. \label{eq:AB} \end{equation}
The equivalent statement for right modules is established in \cite[Lemma 6.2]{MP21}. By substituting $\gamma = \alpha \cdot (x-1)$ and $\beta\lambda_R = 1-\alpha(1-xy)$, we obtain
\begin{align*} f(A) &= [(1+(1-xy)\alpha)(x-1),(1-xy)\beta\lambda_R] = (1+(1-xy)\alpha) \cdot [x-1,1-xy] \\
f(B) &= [-\lambda_R\alpha(x-1),(1-\lambda_R\beta)\lambda_R] = -\lambda_R\alpha \cdot [x-1,1-xy]. \end{align*}
Hence $i_R(\pi_2(X_{\mathcal{P}_R}))$ is generated by $i_R(A) = 1+(1-xy)\alpha$ and $i_R(B) = -\lambda_R \alpha$, as required.
\end{proof}

\begin{proof}[Proof of \cref{thm:Psi(P_R)}]
Let $\alpha \in \Z G$ be any element as in \cref{lemma:compute-pi_2} (i) and let $I = (1+(1-xy)\alpha, \lambda_R \alpha) \le \Z G$. By \cref{prop:Psi-explicit}, it suffices to show that (a) There is an exact sequence of $\Z G$-modules $0 \to \Z \to I \to \pi_2(X_{\mathcal{P}_R}) \to 0$, and (b) $I$ is a stably free $\Z G$-module. 

(a) We have $N = N \cdot (1+(1-xy)\alpha) \in I$ and so there is an injective map $i : \Z \to I$, $1 \mapsto N$. The quotient map $q : \Z G \twoheadrightarrow \Z G/N$ has $q(I) = (1+(1-xy)\alpha, \lambda_R \alpha) \le \Z G/N$. By \cref{lemma:compute-pi_2}, this is isomorphic to $\pi_2(X_{\mathcal{P}_R})$, and thus we obtain a surjective map $q' : I \twoheadrightarrow \pi_2(X_{\mathcal{P}_R})$. Exactness follows from the fact that $\ker(q') = \ker(q) \cap I = (\Z \cdot N) \cap I = \Z \cdot N = \IM(i)$.

(b) Let $I' = (1+(1-xy)\alpha, \lambda_R \alpha) \le \Z G/N$, which we will view as a $\Z G$-module. Since $G$ is finite and $\mathcal{P}_R$, $\mathcal{P}_n^{\std}$ have the same deficiency, it follows from \cite{Br79a} that $X_{\mathcal{P}_R} \vee S^2 \simeq X_{\mathcal{P}_n^{\std}} \vee S^2$. Since $\pi_2(X_{\mathcal{P}_R}) \cong I'$ and $\pi_2(X_{\mathcal{P}_n^{\std}}) \cong \Z G/ N$, this implies that there is an isomorphism of $\Z G$-modules $f : I' \oplus \Z G \to \Z G/N \oplus \Z G$.
To show that $I$ is stably free, our strategy will be to find $f$ explicitly and use it to define a map $\wt f : I \oplus \Z G \to \Z G^2$ which we will verify is an isomorphism.

Let $\partial_2 = \cdot \left(\begin{smallmatrix} v_1 \\ v_2 \end{smallmatrix}\right)$ and $\partial_2' = \cdot \left(\begin{smallmatrix} v_1 \\ v_3 \end{smallmatrix}\right)$ be as in \cref{eq:d2-d2'-def} from the proof of \cref{lemma:compute-pi_2} (i), so that $v_3 = \lambda_R v_2 \in \Z G^2$ and $v_2=\gamma v_1 + \beta v_3 \in \Z G^2$. Let $\partial_2'' =  \cdot \left(\begin{smallmatrix} v_1 \\ v_2 \\ v_3 \end{smallmatrix}\right)$. Then there are isomorphisms
\begin{align*} \psi : \ker(\partial_2'') & \xrightarrow[]{\cong} \ker(\partial_2) \oplus \Z G & \psi' : \ker(\partial_2'') & \xrightarrow[]{\cong}\ker(\partial_2') \oplus \Z G \\
(a,b,c) &\mapsto [(a,b+c\lambda_R), c] & (a,b,c) &\mapsto [(a+b\gamma,c+b\beta),b] \\
(a,b-c\lambda_R, c) &\mapsfrom [(a,b),c] & (a-b\gamma,b,c-b\beta) &\mapsfrom [(a,c),b]. 
\end{align*}
Recall from the proof of \cref{lemma:compute-pi_2} (iii) that $\ker(\partial_2) = \Z G \cdot (x-1,1-xy) \cong \Z G/N$ and $\ker(\partial_2') \cong I'$.
Combining this with the maps above gives mutually inverse isomorphism
\begin{align*} f : I' \oplus \Z G & \xrightarrow[]{\cong} \Z G/N \oplus \Z G & g : \Z G/N \oplus \Z G & \xrightarrow[]{\cong} I' \oplus \Z G \\
(1+(1-xy)\alpha,0) &\mapsto (1+(1-xy)\alpha, (1-xy)\beta) & (1,0) &\mapsto (1+(1-xy)\alpha,1-xy) \\
(\lambda_R\alpha,0) &\mapsto (\lambda_R\alpha, \lambda_R\beta-1) & (0,1) &\mapsto (-\lambda_R\alpha,-\lambda_R). \\
(0,1) &\mapsto (-\alpha, -\beta) & &
\end{align*}

Completely analogously, we can now define
\begin{align*} \wt f : I \oplus \Z G & \xrightarrow[]{\cong} \Z G \oplus \Z G & \wt g : \Z G/N \oplus \Z G & \xrightarrow[]{\cong} I \oplus \Z G \\
(1+(1-xy)\alpha,0) &\mapsto (1+(1-xy)\alpha, (1-xy)\beta) & (1,0) &\mapsto (1+(1-xy)\alpha,1-xy) \\
(\lambda_R\alpha,0) &\mapsto (\lambda_R\alpha, \lambda_R\beta-1) & (0,1) &\mapsto (-\lambda_R\alpha,-\lambda_R). \\
(0,1) &\mapsto (-\alpha, -\beta) & &
\end{align*}
To see that $\wt f$ is well-defined when restricted to $I$, note that the first coordinate is the just inclusion map $I \hookrightarrow \Z G$ and the second coordinate is just $F \circ q'$ where $q' : I \twoheadrightarrow I'$ and $F : I' \to \Z G$ is the second coordinate of $f|_{I'}$. Similarly, $\wt g$ is well-defined when restricted to $\Z G/N$ since $\wt g(N,0)=(0,0)$. We can now verify that $\wt f \circ \wt g = \id$ and $\wt g \circ \wt f = \id$, as required.
\end{proof}

\subsection{Determining of $\Psi(\mathcal{P})$ for specific presentations}
\label{ss:specific-examples}

As before, let $n \ge 2$, let $k \ge 0$, let $n_1,m_1,\dots,n_k,m_k \ne 0$, and let $R = a^{n_1}b^{m_1} \cdots a^{n_{k+1}}b^{m_{k+1}}$ for $n_{k+1} = 1 - \sum_{i=1}^{k} n_i$ and $m_{k+1} = 1 - \sum_{i=1}^{k} m_i$. We will write $G = Q_{4n}$. 
The main result of this section is as follows.

\begin{thm} \label{thm:MP-module}
Suppose $r \in \Z$ is such that $\mathcal{P}_n(2;r)$ is a regular presentation for $G$. Then:
\[ \wh \Psi(\mathcal{P}_n(2;r)) = [I] \in \SF(\Z G)/\sim\] 
where $I = (1+ (1-xy)\alpha_t, \,  (x^r-x-1)\alpha_t) \le \Z G$
for
\[ 
\alpha_t = \begin{cases}
    -(x-1)(\sum_{i=0}^{t-1} x^{3i})(x^r+x^{r-1}-1), & \text{if $r \equiv 3t \mod n$} \\
    -x^2(x-1)(\sum_{i=0}^{t-1} x^{3i})(x^r+x^{r-1}-1), & \text{if $r \equiv 3t+1 \mod n$}.
\end{cases}
\]
We let $I_{n,r}$ denote the ideal defined using the minimal $t \ge 0$ for which $r \equiv 3t$ or $3t+1 \mod n$.
\end{thm}

\begin{remark} \label{remark:after-MP-module}
A choice of $t$ always exists since, if $r \equiv 2 \mod 3$, then \cref{prop:MP-not-presentation} implies that $3 \nmid n$ and so $r \equiv r+n$ where now $r+n \not \equiv 2 \mod 3$.
\end{remark}

The following lemma gives a strategy for solving the equation $\alpha (xy-1) + \beta \lambda_R = 1$ given in \cref{lemma:compute-pi_2} (i).
For an integer $d \ge 1$, let $C_d = \langle x \rangle$ denote the cyclic group of order $d$, and write $\bar{\,\cdot\,}:\Z C_{d} \to \Z C_{d}$ for the map $\sum_{i=0}^{d-1} a_i x^i \mapsto \sum_{i=0}^{d-1} a_i x^{-i}$ for $a_i \in \Z$. 

\begin{lemma} \label{lemma:a_0->a}
Let $C_{2n} = \langle x \rangle \le G$ and let $q : \Z C_{2n} \to \Z C_n$ denote the quotient map.
Then:
\begin{clist}{(i)}
\item $\alpha (xy-1) + \beta \lambda_R = 1$ has a solution for $\alpha,\beta \in \Z G$ if and only if $\delta q(\lambda_R) + \gamma q(\bar{\lambda}_R)= 1$ has a solution for $\delta,\gamma \in \Z C_{n}$. 
More specifically, if $\wt \delta, \wt \gamma \in \Z C_{2n}$ are lifts of $\delta, \gamma$ under $q$, then there exists $\varepsilon \in \Z C_{2n}$ such that $\wt \delta \lambda_R + \wt \gamma \bar{\lambda}_R +\varepsilon (x^n-1) = 1$ and we can take $\alpha = \varepsilon-\wt \gamma\bar{\lambda}_R+\varepsilon x y$ and $\beta = \wt\delta+\wt \gamma xy$.
\item If $\delta \lambda_R + \gamma \bar{\lambda}_R= 1$ has a solution for $\delta,\gamma \in \Z C_{2n}$, then $\alpha (xy-1) + \beta \lambda_R = 1$ has a solution $\alpha = -\gamma\bar{\lambda}_R$ and $\beta = \delta+\gamma xy$.
\end{clist}
\end{lemma}

\begin{proof}
(i) Let $\alpha = \alpha_0 + \alpha_1y$ and $\beta = \beta_0+\beta_1y$ for some $\alpha_0,\alpha_1,\beta_0,\beta_1 \in \Z C_{2n}$. By \cref{lemma:compute-lambda}, we have that $\lambda_R \in \Z C_{2n}$ and so $y \lambda_R = \bar{\lambda}_R y$. Thus $\alpha (xy-1) + \beta \lambda_R = 1$ can be rewritten as:
\[ (-\alpha_0+\alpha_1x^{n-1}+\beta_0\lambda_R)+(\alpha_0x-\alpha_1+\beta_1\bar{\lambda}_R)y = 1. \]
This is equivalent to $\beta_0 \lambda_R + (\beta_1x^{-1})\bar{\lambda}_R+ (\alpha_1x^{-1})(x^n-1)=1$ and $\alpha_0 = \alpha_1x^{-1}-\beta_1\bar{\lambda}_R x^{-1}$.

Hence $\alpha (xy-1) + \beta \lambda_R = 1$ has a solution for $\alpha,\beta \in \Z G$ if and only if $\wt \delta \lambda_R + \wt \gamma \bar{\lambda}_R +\varepsilon (x^n-1) = 1$ has a solution for $\wt \delta,\wt \gamma,\varepsilon \in \Z C_{2n}$. In the forward direction, we take $(\wt \delta,\wt \gamma,\varepsilon) = (\beta_0,\beta_1x^{-1},\alpha_1x^{-1})$. For the converse, we take $(\alpha,\beta) =(\varepsilon-\gamma\bar{\lambda}_R+\varepsilon x y, \delta+\gamma xy)$.

Finally, note that $\delta \lambda_R + \gamma \bar{\lambda}_R +\varepsilon (x^n-1) = 1$ has a solution for $\delta,\gamma,\varepsilon \in \Z C_{2n}$ if and only if $\delta \lambda_R + \gamma \bar{\lambda}_R= 1$ has a solution for $\delta,\gamma \in \Z C_{n} \cong \Z C_{2n}/(x^n-1)$.

(ii) Just let $\varepsilon=0$ in (i).
\end{proof}

We now implement this strategy in the case $\lambda_R = \lambda_n(2;1-r) = 1+x-x^r$ for $r \ge 2$.
A much more general result is possible but the calculations are 
long and will be omitted for brevity.

\begin{lemma}\label{lem:mpgens}
Let $r \ge 2$.
In each of the following cases, we will gives elements $\alpha, \beta \in \Z G$ such that $\alpha \cdot (xy-1) + \beta \cdot (1+x-x^r) = 1 \in \Z G$.
\begin{clist}{(i)}
\item 
Suppose $r \equiv 0 \mod 3$. Let $r =3t$. Then, for $\Sigma = \sum_{i=0}^{t-2} x^{3i}$, we can take:
\[ \alpha = -(x-1)(1+x^3\Sigma)(x^r+x^{r-1}-1), \quad \beta = x+x^2(x^2-1)\Sigma+x^{r+1}(x-1)(1+x^3\Sigma)y. \]
\item 
Suppose $r \equiv 1 \mod 3$. Let $r =3t+1$. Then, for $\Sigma = \sum_{i=0}^{t-1} x^{3i}$, we can take:
\[ \alpha = -x^2(x-1)\Sigma (x^r+x^{r-1}-1), \quad \beta = 1+x(x^2-1)\Sigma + x^{r+3}(x-1)\Sigma y. \]
\end{clist}
\end{lemma}

\begin{proof}
Let $\lambda = 1+x-x^r$. By \cref{lemma:a_0->a} (ii), it suffices to find $\delta,\gamma \in \Z C_{2n}$ such that $\delta \lambda + \gamma \bar{\lambda}=1$.
We have $\bar{\lambda} = 1+x^{-1}-x^{-r} = x^{-r}(x^r+x^{r-1}-1)$.
By the argument given in \cref{prop:MP-not-presentation}, if $\delta', \gamma' \in \Z C_{n}$ are such that $\delta'(x^{r-1}+x)+\gamma'(x^2+x+1)=1$, then we can take $\delta = \delta'+\gamma'(x+1)$ and $\gamma = (\delta'+\gamma'x)x^r$.

Let $r=3t$. Then $(-1)(x^{r-1}+x) + [1+x^2(x-1)(\sum_{i=0}^{t-2} x^{3i})](x^2+x+1)=1$ and so we can take $\delta' = -1$ and $\gamma' = 1+ x^2(x-1)(\sum_{i=0}^{t-2} x^{3i})$. This gives $\delta = x+x^2(x^2-1)(\sum_{i=0}^{t-2} x^{3i})$ and $\gamma = x^r(x-1)(1+x^3(\sum_{i=0}^{t-2} x^{3i}))$. The required values of $\alpha$, $\beta$ now follow from \cref{lemma:a_0->a} (ii).

Let $r=3t+1$. Then $(-x)(x^{r-1}+x) + [1+x(x-1)(\sum_{i=0}^{t-1} x^{3i})](x^2+x+1) = 1$
and so we can take $\delta' = -x$ and $\gamma' = 1+ x(x-1)(\sum_{i=0}^{t-1} x^{3i})$.
This gives $\delta = 1+x(x^2-1)(\sum_{i=0}^{t-1} x^{3i})$ and $\gamma = x^{r+2}(x-1)(\sum_{i=0}^{t-1} x^{3i})$.
The values of $\alpha$, $\beta$ similarly now follow from \cref{lemma:a_0->a} (ii).
\end{proof}

We next make the following general observation about regular presentations of height one.

\begin{prop} \label{lem:symmetry_mods}
Let $n_1,m_1 \in \Z$ be such that $\mathcal{P}_n(n_1;m_1)$ is a regular presentation for $G$. Then $\mathcal{P}_n(m_1;n_1)$ is a regular presentation for $G$ and $\mathcal{P}_n(n_1;m_1) \simeq_{Q^*} \mathcal{P}_n(m_1;n_1)$. Furthermore, we have that $\mathcal{E}_{n,r} \simeq_{Q^*} \mathcal{E}_{n,1-r}$ for all $r \in \Z$.
\end{prop}

\begin{proof}
We will begin by showing that $\mathcal{P}_n(m_1;n_1)$ is a regular presentation for $G$. Let $R(a,b) = a^{n_1}b^{m_1}a^{1-n_1}b^{1-m_1}$ so that $\mathcal{P}_n(n_1;m_1) = \mathcal{P}_R = \langle x, y \mid x^ny^{-2}, R(x,yxy^{-1})\rangle$ is a regular presentation for $G$. Applying $x \mapsto x^{-1}$ and $y \mapsto y^{-1}$ to $\mathcal{P}_R$ gives $\mathcal{P} = \langle x, y \mid x^{-n}y^2,R(x^{-1},y^{-1}x^{-1}y) \rangle$ which is a presentation for $G$ such that $f: \mathcal{P} \to G$, $x \mapsto x^{-1}$, $y \mapsto y^{-1}$ is an isomorphism.
Note that there exists $\theta \in \Aut(G)$ such that $\theta(x) = x^{-1}$ and $\theta(y) = x^ny = y^{-1}$ (see, for example, the proof of \cref{lemma:pres-identification}), 
which implies that $\id = \theta^{-1} \circ f : \mathcal{P} \to G$ is an isomorphism.

We now claim that $\mathcal{P} \simeq_Q \mathcal{P}_n(m_1;n_1)$.
Using $Q$-moves of type (ii) and (iii), we can invert and cyclically permute the relators. In this way, $x^{-n}y^2$ can be sent to $x^ny^{-2}$ and $R(x^{-1},y^{-1}x^{-1}y) = x^{-n_1}y^{-1}x^{-m_1}yx^{n_1-1}y^{-1}x^{m_1-1}y$ can be sent to $x^{m_1}yx^{n_1}y^{-1}x^{1-m_1}yx^{1-n_1}y^{-1} = S(x,yxy^{-1})$ where $S = a^{m_1}b^{n_1}a^{1-m_1}b^{1-n_1}$, as required. Thus $\mathcal{P}_n(m_1;n_1)$ is a regular presentation for $G$.
Since $\mathcal{P}_n(n_1;m_1) \simeq_{Q^*} \mathcal{P} \simeq_Q \mathcal{P}_n(m_1;n_1)$, we have that $\mathcal{P}_n(n_1;m_1) \simeq_{Q^*} \mathcal{P}_n(m_1;n_1)$.

The last part now follows from $\mathcal{E}_{n,r} = \mathcal{P}_n(2;1-r) \simeq_{Q^*} \mathcal{P}_n(1-r;2)$ as well as the equivalence $\mathcal{P}_n(1-r;2) \simeq_Q \mathcal{P}_n(2;r) = \mathcal{E}_{n,1-r}$ obtained by using operation (III) from \cref{lemma:operations}.
\end{proof}

\begin{proof}[Proof of \cref{thm:MP-module}]
First note that, if $r \equiv r' \mod n$, then \cref{lemma:operations,lem:symmetry_mods} imply $\mathcal{P}_n(2;r) \simeq_Q \mathcal{P}_n(2;r') \simeq_Q \mathcal{P}_n(2;1-r)$ and so $\Psi(\mathcal{P}_n(2;r)) = \Psi(\mathcal{P}_n(2;r')) = \Psi(\mathcal{P}_n(2;1-r))$ since $Q^*$-equivalence implies homotopy equivalence. Hence, if $r \equiv 3t \mod n$ (resp. $r \equiv 3t+1 \mod n$), it suffices to determine $\Psi(\mathcal{P}_n(2;1-3t))$ (resp. $\Psi(\mathcal{P}_n(2;-3t))$).
The result now follows by combining \cref{thm:Psi(P_R)} and \cref{lem:mpgens}. Here we have used that $1+x^3\sum_{i=0}^{t-2} x^{3i} = \sum_{i=0}^{t-1} x^{3i}$ to simplify the formula for $\alpha$ in the case $r \equiv 3t \mod n$.
\end{proof}

We will now prove the following result from \cref{ss:specific-examples}.
The key idea we will exploit is that, if $\mathcal{P}_R$ is a regular presentation for $G$, then $\alpha (xy-1)+\beta\lambda_R=1$ has a solution for $\alpha,\beta \in \Z G$ (by \cref{lemma:compute-pi_2}). If no solution exists, then we obtain a contradiction.

\begingroup
\renewcommand\thethm{\ref{prop:MP-not-presentation}}
\begin{prop}
If $r \equiv 2 \mod 3$ and $3 \mid n$, then $\mathcal{E}_{n,r} := \mathcal{P}_n(2;1-r)$ and $\mathcal{P}_n(r;r)$ are not regular presentations for $G$. 
\end{prop}
\setcounter{thm}{\value{thm}-1}
\endgroup

\begin{proof}
It follows from \cref{cor:pn-standard-gen} that, if one of these presentation presents $G$ then it does so on the standard generating set, i.e. the presentation is regular. It therefore suffices to prove that (i) $\mathcal{E}_{n,r} := \mathcal{P}_n(2;1-r)$ and (ii) $\mathcal{P}_n(r;r)$ are not regular presentations for $G$.

(i) Suppose that $\mathcal{E}_{n,r}$ is a regular presentation for $G$. By \cref{lemma:compute-lambda}, we have $\lambda = \lambda_n(2;1-r) = 1+x-x^r$.
By \cref{lemma:compute-pi_2}, there exist $\alpha,\beta \in \Z G$ such that $\alpha(xy-1)+\beta\lambda=1$.
By \cref{lemma:a_0->a}, it then follows that there exist $\delta,\gamma \in \Z C_n$ such that $\delta \lambda + \gamma \bar{\lambda} = 1 \in \Z C_n$ where $\lambda$ and $\bar{\lambda}$ are understood to be the images of these elements under the quotient map $q: \Z C_{2n} \twoheadrightarrow \Z C_n$.

We have $\bar{\lambda} = 1+x^{-1}-x^{-r} = x^{-r}(x^r+x^{r-1}-1)$.
By applying division twice, we get 
that $x^{r-1}+x = \lambda +x^r \bar{\lambda}$ and $x^2+x+1 = \lambda + x(x^{r-1}+x) = (x+1)\lambda + x^{r+1} \bar{\lambda}$. 
In particular, if $\delta', \gamma' \in \Z C_{n}$ are such that $\delta'(x^{r-1}+x)+\gamma'(x^2+x+1)=1$, then we can take $\delta = \delta'+\gamma'(x+1)$ and $\gamma = (\delta'+\gamma'x)x^r$.
Since $3 \mid n$, there is a quotient $C_n \twoheadrightarrow C_3$. Applying the quotient $\Z C_n \twoheadrightarrow \F_2 C_3$ now gives that $\gamma' (x^2+x+1)=1 \in \F_2 C_3$ since $x^{r-1}+x =2x=0 \in \F_2 C_3$, and so $1+x+x^2 \in \F_2C_3^\times$. This is a contradiction since $(x-1)(1+x+x^2)=0 \in \F_2C_3$.

(ii) Suppose that $\mathcal{P}_n(r;r)$ is a regular presentation for $G$. By \cref{lemma:compute-lambda}, we have $\lambda = \lambda_n(r;r) = \sigma(r) - x^{1-r}\sigma(r-1)$. Similarly to above, \cref{lemma:compute-pi_2,lemma:a_0->a} imply that there exist $\delta,\gamma \in \Z C_n$ such that $\delta \lambda + \gamma \bar{\lambda} = 1 \in \Z C_n$.
We have $\lambda = 1+(1-x^{-r})x\sigma(r-1)$ and so $\lambda+\bar{\lambda}=2$ which implies that $(\delta-\gamma)\lambda =1 \in \F_2 C_n$ and so $\lambda \in \F_2C_n^\times$. Since $3 \mid n$, we have a quotient $q:\F_2 C_n \twoheadrightarrow \F_2 C_3$ where $q(\lambda) \in \F_2 C_3^\times$. Since $r \equiv 2 \mod 3$, we have $q(\lambda) = 1+(1+x)xq(\sigma(r-1))$. Since $q(\sigma(6))=0$, we have that $q(\sigma(r-1))$ depends on the value of $r \mod 6$. Since $q(\sigma(1)) = 1$ and $q(\sigma(4)) = 1+x+x^2$, we have that $q(\lambda) = 1+(1+x)x=1+x+x^2$ (if $r \equiv 2 \mod 6$) or $q(\lambda) = 1+(1+x)x(1+x+x^2) = 1+x+x^2$ (if $r \equiv 5 \mod 6$). Either way we have that $1+x+x^2 \in \F_2C_3^\times$, which is a contradiction as above.
\end{proof}

We now give more complicated examples for which $\alpha$ and $\beta$ can be computed. This is necessary for the proof of \cref{thmx:a-non-MP-presentation}.

\begin{example} \label{example:P_12}
  By Theorem~\ref{thm:Psi(P_R)} and Lemma~\ref{lemma:a_0->a} (ii), to find a representative for $\wh \Psi(\mathcal{P}_R)$ for a given regular presentation $\mathcal{P_R}$ it suffices to find $\delta, \gamma \in \Z[\langle x \rangle] \le \Z Q_{4n}$ such that $\delta \lambda + \gamma \bar \lambda = 1$, where $\lambda = \lambda_R$.
  If so, then 
  $\wh \Psi(\mathcal{P}_R) = [I] \in \SF(\Z Q_{4n})/\sim$ 
  where $I = (1+(1-xy)\alpha, \lambda \alpha) \le \Z Q_{4n}$ and $\alpha = -\gamma \bar \lambda$. 
  We will now determine $\gamma$, $\delta$ for some regular presentations over $Q_{48}$.
\begin{clist}{(i)}
\item
Consider the presentation $\mathcal{P}_{12}(3; 3)$ of $Q_{48}$. Then $\lambda = \lambda_{12}(3;3) = -x^{23} - x^{22} + x^2 + x + 1$ and $x^{23}\overline{\lambda} = x^{23} + x^{22} + x^{21} - x - 1$. Applying the Euclidean algorithm (as in the proof of Lemma~\ref{lem:symmetry_mods}) shows that $\delta \lambda + \gamma \bar \lambda = 1$, where $\delta = x^{21} + x^{19} - x^{18} + x^{14} - x^{13} + x^9 - x^8 + x^4 - x^3$ and $\gamma = x^{23}(x^{21} - x^{17} + x^{16} - x^{12} + x^{11} - x^7 + x^6 - x^2 + x - 1)$.
\item Consider the presentation $\mathcal{P}_{12}(6; 3)$ of $Q_{48}$.
Then $\lambda = \lambda_{12}(6;3) = -x^{23} - x^{22} + x^5 + x^4 + x^3$ and
$x^{23}\overline{\lambda} = x^{20} + x^{19} + x^{18} - x - 1$.
The Euclidean algorithm gives that $\delta \lambda + \gamma \bar \lambda = 1$ where
\[ \delta = 2x^{19} + 2x^{17} - x^{16} + x^{15} + x^{12} - x^{11} + x^7 - x^6 + x^2 - x - 2 \]
and $\gamma = x^{23}(2x^{22} + x^{19} - x^{18} + x^{17} - x^{13} + x^{12} - x^8 + x^7 - 2x^4 - x^3 - x^2 + x - 1)$.
\item
Consider the presentation $\mathcal{P}_{12}(2, 2; -2, -2)$ of $Q_{48}$. Then $\lambda = \lambda_{12}(2; 2) = -x^7 - x^6 + x^4 + x + 1$ and $x^7 \bar \lambda  = x^7 + x^6 + x^3 - x - 1$.
The Euclidean algorithm gives that $\delta \lambda + \gamma \bar \lambda = 1$ where
$\delta = 2x^6 + x^5 - x^4 + x^3 + 2x^2 - x - 2$ and $\gamma = x^7 (2x^6 + x^5 - x^4 - x^3 + x^2 - 3)$. 
\end{clist}
\end{example}

\section{Presentations of quaternion groups of small order}
\label{s:pres-small-order}

The aim of this section is to prove Theorems \ref{thmx:main-D2}, \ref{thmx:first-example}, \ref{thmx:a-non-MP-presentation} and \ref{thmx:main-Q4n} from the introduction.
We begin by recalling some basic facts which are used throughout (\cref{ss:prelim-Q4n}). We then consider each of the groups $Q_{4n}$ for $2 \le n \le 12$ in turn. 
The proof of \cref{thmx:main-D2} is contained in Sections \ref{ss:Q24}-\ref{ss:Q40}, the proof of \cref{thmx:first-example} is given in \cref{ss:Q40}, the proof of \cref{thmx:a-non-MP-presentation} is given in \cref{ss:Q48}, and the proof of \cref{thmx:main-Q4n} uses results from Sections \ref{ss:Q24}-\ref{ss:Q44} and is given in \cref{ss:Q4n}.

\subsection{Preliminaries and general strategy} \label{ss:prelim-Q4n}

Let $\zeta_d = e^{2 \pi i/n} \in \C$ be a primitive $d$-th root of unity.
For a field $\F$, let $\H_{\F} = \F[i,j]$ denote the quaternions over $\F$ and define $\H_{\Z}=\Z[i,j]$ and $\Lambda_d = \Z[\zeta_d, j]$ to be subrings of $\H_{\R}$.
Let $\tau_d = \zeta_d + \zeta_d^{-1}$. Then $\Q \otimes \Lambda_d$ is a quaternion algebra over $\Q(\tau_d)$ and $\Lambda_d$ a $\Z[\tau_d]$-order in $\Q \otimes \Lambda_d$.
If $n \ge 2$ and $d \mid 2n$, then there is a ring isomorphism $\Z Q_{4n} /(\Phi_d(x)) \to \Lambda_{d}$ induced by the map $x \mapsto \zeta_{d}$, $y \mapsto j$ (see \cite{Sw83}). For example, if $2 \mid n$, we have that $\Phi_4(x) = x^2+1$ and $\Z[Q_{4n}]/(x^2+1) \cong \Lambda_4 \cong \Z[i,j] = \H_{\Z}$.

Let $\SF_1(\Z G)/{\Aut(G)}$ denote the set of orbits of the standard action of $\Aut(G)$ on the rank one stably free $\Z G$-modules $\SF_1(\Z G)$, where $\theta \in \Aut(G)$ sends $P \mapsto P_\theta$.

\begin{lemma} \label{lemma:MP-homotopy}
Let $6 \le n \le 11$, let $G = Q_{4n}$, let $\l := \Z Q_{4n} / (x^n+1)$ and $I_{n,3} = (g_1,g_2) \le \Z G$ where $g_1 = 1-(1-xy)(x-1)(x^3+x^2-1)$ and $g_2 = (x^3-x-1)(x-1)(x^3+x^2-1)$. Then: 
\begin{clist}{(i)}
\item
If $n \ne 8$, then $f : \Z G \twoheadrightarrow \l$ induces a bijection $f_\# : \SF_1(\Z G) \to \SF_1(\l)$.
\item
Stably free Swan modules over $\Z G$ are free.
\item 
There is an injective map $\Psi : \HT_{\min}(G) \to \SF_1(\Z G)/{\Aut(G)}$ such that $\Psi(\mathcal{P}_n(2;-2)) = [I_{n,3}]$. 
Furthermore, $\Psi$ is bijective if and only if $G$ has the {\normalfont D2} property.
\item
If $I_{n,3}$ is non-free, then $\mathcal{P}_n(2;-2) \not \simeq \mathcal{P}_n^{\std}$.
\end{clist}
\end{lemma}

\begin{proof}
(i) This follows from the proof of \cite[Theorem 11.14]{Sw83} (for $n=6$ or $10$), the results in \cite[Section 10]{Sw83} (for $n=7$ or $11$), and \cite[Corollary 12.2]{Sw83} (for $n=9$).

(ii) First suppose that $n \ne 8$. We claim that all Swan modules over $\Z G$ are free. Since $f : \Z G \twoheadrightarrow \l$ is a surjection of orders we have that, if $I \le \Z G$ is a projective ideal, then $f_\#(I) \cong f(I) \le \l$ by \cref{lemma:ext-ideal}. If $r \in (\Z/|G|)^\times$, this implies that 
\[ f_\#((N,r)) \cong f((N,r)) = (f(N),f(r)) \le \l.\] 
Since $x^n+1 \mid N$, we have $f(N)=0$ and so $f_\#((N,r))$ is singly generated and so free. Since $f_\#$ is bijective, this implies that $(N,r)$ is free.

Next suppose that $n=8$. By \cite[Theorem 17.7]{Sw83} (combined with \cite[Proposition 3.4]{Ni20a}), we have that $(N,r)$ is free if $r \equiv \pm 1 \mod 8$ and not stably free if $r \equiv \pm 3 \mod 8$.

(iii) Since stably free Swan modules are free, the maps $\wh \Psi$ and $\Psi$ coincide and so the result follows directly from \cref{thm:Psi(P_R),thm:MP-module}.

(iv) This follows from the fact that, if $[I_{n,3}] = \Psi(\mathcal{P}_n^{\std}) = [\Z G]$, then $I_{n,3} \cong \Z G_\theta$ for some $\theta \in \Aut(G)$ which implies that $I_{n,3} \cong \Z G$.
\end{proof}

We will make frequent use of the following freeness criterion for ideals in $\Z[\zeta_{2n},j]$.

\begin{lemma}\label{lemma:nonfree-metacrit}
Let $n \geq 2$ and assume that any maximal order $\Gamma$ with $\Z[\zeta_{2n}, j] \subseteq \Gamma$ satisfies $\Gamma_0^\times = \{ \alpha \in \Gamma \mid \operatorname{nr}(\alpha) = 1\} \cong Q_{4n}$. 
Let $I \leq \Z[\zeta_{2n}, j]$ be a left ideal and
assume that there exists $u_1,\dotsc,u_l \in \Q[\zeta_{2n}, j]$ with $\operatorname{nr}(u_i) = 1$ and $I u_i \subseteq I$ for $1 \leq i \leq l$, and such that $\langle u_1,\dotsc,u_l\rangle$ is not isomorphic to a subgroup of $Q_{4n}$. Then $I$ is a non-free $\Z[\zeta_{2n}, j]$-module.
\end{lemma}

\begin{proof}
Assume that $I$ is free and let $\Gamma$ be a maximal over of $\Z[\zeta_{20}, j]$. Then also the left $\Gamma$-ideal $\Gamma I$ is free, say $\Gamma I = \Gamma x$ with $x \in \Gamma$.
Then $\mathcal{O}_r(\Gamma I) = \mathcal{O}_r(\Gamma x) = x^{-1}\Gamma x \cong \Gamma$.
By assumption $\langle u_1,\dotsc,u_l \rangle \subseteq \mathcal{O}_r(I)_0^\times \subseteq \mathcal{O}_r(\Gamma I)_0^\times \cong Q_{4n}$, a contradiction.
\end{proof}

We will use that $\Aut(Q_{4n}) =\{ \theta_{a,b} : x \mapsto x^a, y \mapsto x^{b}y \mid a \in (\Z/2n)^\times, b \in \Z/2n\}$ for $n \ge 3$.

\subsection{Presentations for $Q_{24}$} \label{ss:Q24}

In this section, we will show: 

\begin{prop} \label{prop:Q24}
$Q_{24}$ has the {\normalfont D2} property and every balanced presentation for $Q_{24}$ is homotopy equivalent to precisely one of $\mathcal{P}_6^{\std}$ and $\mathcal{P}_6(2;-2)$.
\end{prop}

We start by considering the general case $Q_{8p}$ where $p$ is an odd prime, so that the same results can be used for $Q_{40}$ in \cref{ss:Q40}.
Let $\Lambda = \Z Q_{8p}/(x^{2p}+1)$ and recall that the quotient $f : \Z Q_{8p} \twoheadrightarrow \Lambda$ induces a bijection
$f_\# : \SF_1(\Z Q_{8p}) \to \SF_1(\Lambda)$.
By \cite[Example 42.3]{CR87}, 
$x^{2p}+1 = \Phi_4(x)\Phi_{4p}(x)$ and the isomorphisms $\Z Q_{8p}/(\Phi_4(x)) \cong \H_{\Z}$, $\Z Q_{8p}/(\Phi_{4p}(x)) \cong \Z[\zeta_{4p},j]$ induce a pullback square
\[
\mathcal{R} = 
\begin{tikzcd}
	\Lambda \ar[r,"i_2"] \ar[d,"i_1"] & \Z[\zeta_{4p},j] \ar[d,"j_2"]\\
	\H_{\Z} \ar[r,"j_1"] & \H_{\F_p}
\end{tikzcd}
\quad
\begin{tikzcd}
	x,y \ar[r,mapsto] \ar[d,mapsto] & \zeta_{4p},j \ar[d,mapsto]\\
	i,j \ar[r,mapsto] & i,j
\end{tikzcd}
\]
of the form given in \cref{ss:modules-prelim}.
By \cite[Lemma 8.12 \& 8.14]{Sw83}, the induced map $(i_2)_*: C(\Lambda) \to C(\Z[\zeta_{4p},j])$ is an isomorphism if $p=3$ or $5$. By \cite[p84]{Sw83}, the ring $\H_{\Z}$ has stably free cancellation, i.e. every stably free module is free.
Consider the extension of scalars map
\[ g =(i_2)_\# \colon \SF_1(\Lambda) \to \SF_1(\Z[\zeta_{4p},j]), \]
which is surjective by \cite[Theorem A10]{Sw83}. The following is straightforward to verify (see \cite[Lemma 9.4]{Ni20a} for the case $p=3$).

\begin{lemma} \label{lemma:induced-action}
If $a \in (\Z/4p)^\times$, $b \in \Z/4p$, then $\theta_{a,b} \in \Aut(Q_{8p})$ extends to a Milnor square automorphism $\hat{\theta}_{a,b} = (\theta_{a,b}',\theta_{a,b}^1,\theta_{a,b}^2,\bar{\theta}_{a,b}) \in \Aut(\mathcal{R})$ where $\bar{\theta}_{a,b} \in \Aut(\H_{\F_p})$ sends $j \mapsto (-1)^b j$.
\end{lemma}

We now restrict to the case $p=3$. By \cite[p84]{Sw83}, $\Z[\zeta_{12},j]$ has stably free cancellation, so $\SF_1(\Lambda) = g^{-1}(\Z[\zeta_{12},j])$ and there is a bijection $\varphi : \SF_1(\Lambda) \to \H_{\Z}^\times \backslash \H_{\F_3}^\times \slash \Z[\zeta_{12},j]^\times$ by \cref{lemma:Milnor-square-ideals}.
By \cite[Lemma 9.3]{Ni20a}, we have $\H_{\Z}^\times \backslash \H_{\F_3}^\times \slash \Z[\zeta_{12},j]^\times = \{[1],[1+j],[1+k] \}$.
Hence, by \cite[Proposition 8.8]{Ni20a} and \cref{lemma:induced-action}, there is a bijection
\[ \SF_1(\Z Q_{24})/{\Aut(Q_{24})} \leftrightarrow \{[1],[1+j],[1+k] \}/{\Aut(Q_{24})} = \{ [1], [1+j]\} \]
where $\theta_{a,b} \in \Aut(Q_{24})$ acts on the double cosets via the action 
\[ \bar{\theta}_{a,b}([1+j]) = 
\begin{cases}
 	[1+(-1)^{b_0}j] = [1+j], & \text{if $b=2b_0+1$}\\
	[1+(-1)^{b_0}k]=[1+k], & \text{if $b=2b_0$}
\end{cases}
\]
and so $\bar{\theta}_{a,b}$ acts non-trivially when $b$ is even. Thus $\lvert\SF_1(\Z Q_{24})/{\Aut(Q_{24})} \rvert=2$.

We now claim that $I_{6,3} = (g_1,g_2)$ is non-free where
\[ g_1 = 1-(1-xy)(x-1)(x^3+x^2-1), \quad g_2 = (x^3-x-1)(x-1)(x^3+x^2-1).\]
Since $f_\#$ is bijective, it suffices to show that $f_\#(I_{6,3})$ is non-free.
By \cref{lemma:ext-ideal}, $f_\#(I_{6,3}) \cong J_{6,3}$ where $J_{6,3} := f(I_{6,3}) = (f(g_1),f(g_2)) \le \l$.

\begin{lemma} \label{lemma:E63-prep}
$i_1(J_{6,3}) = (2+j) \le \H_{\Z}$ and $i_2(J_{6,3}) = (1) \le \Z[\zeta_{12},j]$.   
\end{lemma}

\begin{proof}
We have $i_1(f(g_1)) = 2-i+j+2k = (1+k)(2+j)$ and $i_1(f(g_2)) = -5(1+i)$.
Since $N(2+j):=\ol{(2+j)}(2+j)=5$, we have $(1+k)(2+j), -5(1+i) \in (2+j)$ and so $ i_1(J_{6,3}) \subseteq (2+j)$. 
Conversely, $N(2-i+j+2k) = 15$ and $N(-5(1+i)) = 50$ implies $5 = \gcd(15,50) \in i_1(J_{6,3})$. Next, $(2+j)(1+k)(2+j) = 3+4j+5k \equiv -2-j \mod 5$ implies $2+j \in i_1(J_{6,3})$, and so $i_1(J_{6,3}) = (2+j)$.

We have $i_2(f(g_1)) = 1+\zeta_{12}-j$ and $i_2(f(g_2))=1+\zeta_{12}$, where we used $\zeta_{12}^4-\zeta_{12}^2+1=0$ to simplify the expressions.
This implies $j \in i_2(J_{6,3})$ and so $1 \in i_2(J_{6,3})$. Hence $i_2(J_{6,3}) = (1)$.
\end{proof}

By combining \cref{lemma:Milnor-square-ideals,lemma:E63-prep}, we get that $\varphi(J_{6,3}) = [2+j] \in \H_{\Z}^\times \backslash \H_{\F_3}^\times \slash \Z[\zeta_{12},j]^\times$.
We have $-j(2+j) \equiv 1+j \in \H_{\F_3}^\times$ and so $\varphi(J_{6,3}) = [2+j] = [1+j]$.
Thus $J_{6,3}$ is non-free, and so $I_{6,3}$ is non-free which implies that $\mathcal{P}_6^{\std} \not \simeq \mathcal{P}_6(2;-2)$ by \cref{lemma:MP-homotopy}. Hence $\lvert \HT_{\min}(Q_{24}) \rvert \ge 2$ and so, since $\lvert \SF_1(\Z Q_{24})/{\Aut(Q_{24})} \rvert =2$, $\Psi$ is bijective. Thus $Q_{24}$ has the D2 property and $\HT_{\min}(Q_{24}) = \{\mathcal{P}_6^{\std}, \mathcal{P}_6(2;-2)\}$ as required.

\subsection{Presentations for $Q_{28}$} \label{ss:Q28}

In this section, we will show the following. This was proven in \cite{MP21,Ni21b} but we give our proof since it is different and shows the utility of our method. 

\begin{prop} \label{prop:Q28}
$Q_{28}$ has the {\normalfont D2} property and every balanced presentation for $Q_{28}$ is homotopy equivalent to precisely one of $\mathcal{P}_7^{\std}$ and $\mathcal{P}_7(2;-2)$.
\end{prop}

We start by considering the general case $Q_{4p}$ where $p$ is an odd prime, so that the same results can be used for $Q_{44}$ in \cref{ss:Q44}. Let $\Lambda = \Z Q_{4p}/(x^p+1)$ and recall that $f : \Z Q_{4p} \twoheadrightarrow \Lambda$ induces a bijection
$f_\# : \SF_1(\Z Q_{4p}) \to \SF_1(\Lambda)$.
By \cite[Example 42.3]{CR87}, $x^p+1 = (x+1)\psi$ where $\psi = x^{p-1}-x^{p-2}+\cdots - x + 1$, and the isomorphisms $\Z Q_{4p}/(x+1) \cong \Z[j]$, $\Z Q_{4p}/(\psi) \cong \Z[\zeta_{2p},j]$ induce a pullback square
\[
\mathcal{R} = 
\begin{tikzcd}
	\Lambda \ar[r,"i_2"] \ar[d,"i_1"] & \Z[\zeta_{2p},j] \ar[d,"j_2"]\\
	\Z[j] \ar[r,"j_1"] & \F_p[j]
\end{tikzcd}
\quad
\begin{tikzcd}
	x,y \ar[r,mapsto] \ar[d,mapsto] & \zeta_{2p},j \ar[d,mapsto]\\
	-1,j \ar[r,mapsto] & -1,j
\end{tikzcd}
\]
of the form given in \cref{ss:modules-prelim}.
The ring $\Z[j]$ has stably free cancellation. Consider the extension of scalars map
$g =(i_2)_\# \colon \SF_1(\Lambda) \to \SF_1(\Z[\zeta_{2p},j])$.
This is surjective by \cite[Theorem A10]{Sw83}. The analogue of \cref{lemma:induced-action} holds.
By \cref{lemma:Milnor-square-ideals}, there is an injective map
\[ \varphi : g^{-1}(\Z[\zeta_{2p},j]) \hookrightarrow \Z[j]^\times \backslash \F_p[j]^\times \slash \Z[\zeta_{2p},j]^\times.\]
The image of $\varphi$ is given by
\[ \ker\left(\phi : \Z[j]^\times \backslash \F_p[j]^\times \slash \Z[\zeta_{2p},j]^\times \to  \frac{K_1(\F_p[j])}{K_1(\Z[j]) \times K_1(\Z[\zeta_{2p},j])}\right)\]
where $\phi$ is induced by the map $\F_p[j]^\times \to K_1(\F_p[j])$ (see, for example, \cite[Section 4.2.2]{Ni24}).

\begin{lemma} \label{lemma:Q4p-dcoset}
We have $\Z[j]^\times \backslash \F_p[j]^\times \slash \Z[\zeta_{2p},j]^\times = \F_p[j]^\times/(\F_p^\times \cdot \langle j \rangle)$. 
If $|C(\Z[\zeta_p])|$ is odd, then $\ker(\phi) = \ker(N : \F_p[j]^\times/(\F_p^\times \cdot \langle j \rangle) \to \F_p^\times/(\F_p^\times)^2)$ where $N : a+bj \mapsto a^2+b^2$.
\end{lemma}

\begin{proof}
Both statements follow from the proof of \cite[Lemma 4.5]{Ni24} (which is based on \cite[Lemma 7.5]{MOV83}). 
\end{proof}

We now restrict to the case $p=7$. By \cite[p85]{Sw83}, $\Z[\zeta_{14},j]$ has stably free cancellation and so $\SF_1(\Z Q_{28}) = g^{-1}(\Z[\zeta_{14},j])$.
Since $|C(\Z[\zeta_7])|=1$, \cref{lemma:Q4p-dcoset} implies there is a bijection $g^{-1}(\Z[\zeta_{14},j]) \leftrightarrow \{[1], [1+j] \}$.
Since the action of $\Aut(Q_{28})$ fixes the free module, there is a bijection $\SF_1(\Z Q_{28})/{\Aut(Q_{28})} \leftrightarrow \{[1], [1+j] \}$. Hence $\lvert \SF_1(\Z Q_{28})/{\Aut(Q_{28})} \rvert=2$.

We now claim that $I_{7,3} = (g_1,g_2)$ is non-free where
\[ g_1 = 1-(1-xy)(x-1)(x^3+x^2-1), \quad g_2 = (x^3-x-1)(x-1)(x^3+x^2-1).\]
Since $f_\#$ is bijective, it suffices to show that $f_\#(I_{7,3})$ is non-free.
By \cref{lemma:ext-ideal}, $f_\#(I_{7,3}) \cong J_{7,3}$ where $J_{7,3} := f(I_{7,3}) = (f(g_1),f(g_2)) \le \l$.
For calculations, it will be useful to note that
\[ g_1 = 1-(x-1)(x^3+x^2-1)+x(x^{-1}-1)(x^{-3}+x^{-2}-1)y = -x^{-3}(x^3-x-1)(x^4-(x-1)y).\]

\begin{lemma} \label{lemma:E73-prep}
$i_1(J_{7,3}) = (1) \le \Z[j]$ and $i_2(J_{7,3}) = (1+\zeta_{14}j) \le \Z[\zeta_{14},j]$.
\end{lemma}

\begin{proof}
We have $i_1(f(g_1)) =-1-2j$ and $i_1(f(g_2))=-2$ which easily implies that $i_1(J_{7,3}) = (1)$.

For ease of notation, we will identify $\Z[\zeta_{14},j]:= \Z Q_{28}/(\psi)$ and so write $x$ in place of $\zeta_{14}$.
We have $\ol{g}_1 := i_2(f(g_1)) = 1-(1-xj)\alpha$ and $\ol{g}_2 := i_2(f(g_2)) = (x^3-x-1)\alpha$ where $\alpha = (x-1)(x^3+x^2-1)$. If $\delta = -x^5(x^3-x^2+2x-1)$, then $\delta(x^3-x-1)=2 \in \Z[\zeta_{14},j]$. We now obtain $(1+xj) \subseteq i_2(J_{7,3})$ since $1-(xj)^2=2 \in \Z[\zeta_{14},j]$ implies that
\[ 1+xj = (1+xj - 2\alpha) + 2\alpha = (1+xj)\ol{g}_1 + \delta \ol{g}_2 \in i_2(J_{7,3}). \]

Conversely, using the form for $g_1$ established above, we have
\[ \ol{g}_1 = -x^{-4}(x^3-x-1)(x^5-(x-1)xj) \equiv -x^{-4}(x^3-x-1)(x^5+x-1) \mod (1+xj).\]
Next observe that, since $2 = (1-xj)(1+xj) \in (1+xj)$, we have
\[ (x^3-x-1)(x^5+x-1) = x(x^7+1)-\psi + 2(-x^5+x^4-x^3-x+1) =  2(-x^5+x^4-x^3-x+1) \in (1+xj).\]
Hence $\ol{g}_1 \in (1+xj)$. We now obtain $i_2(J_{7,3}) = (1+xj)$ since
\[ 
\ol{g}_2 = x^7 - 2 x^5 - 2 x^4 + 2 x^3 + 2 x^2 - 1 = 2(-x^5-x^4+x^3+x^2-1) \in (1+xj). \qedhere
\]
\end{proof}

Combining \cref{lemma:Milnor-square-ideals,lemma:E73-prep} gives $\varphi(J_{7,3}) = [(1-j)^{-1}] \in \Z[j]^\times \backslash \F_7[j]^\times \slash \Z[\zeta_{14},j]^\times$.
Since $(1+j)(1-j) = 1-j^2 = 2 \in \F_7^\times$, we have $\varphi(J_{7,3}) = [(1-j)^{-1}] = [1+j]$.
Thus $J_{7,3}$ is non-free, and so $I_{7,3}$ is non-free which implies that $\mathcal{P}_7^{\std} \not \simeq \mathcal{P}_7(2;-2)$ by \cref{lemma:MP-homotopy}. Hence $\lvert \HT_{\min}(Q_{28}) \rvert \ge 2$ and so, since $\lvert \SF_1(\Z Q_{28})/{\Aut(Q_{28})} \rvert =2$, $\Psi$ is bijective. Thus $Q_{28}$ has the D2 property and $\HT_{\min}(Q_{28}) = \{\mathcal{P}_7^{\std}, \mathcal{P}_7(2;-2)\}$ as required.

\subsection{Presentations for $Q_{32}$} \label{ss:Q32}

In this section, we will show:

\begin{prop} \label{prop:Q32}
$Q_{32}$ has the {\normalfont D2} property and every balanced presentation for $Q_{32}$ is homotopy equivalent to precisely one of $\mathcal{P}_8^{\std}$ and $\mathcal{P}_8(2;-2)$.
\end{prop}

Let $\l = \Z Q_{32}/(x^8 + 1)$ and let $f \colon \Z Q_{32} \twoheadrightarrow \Lambda$ be the quotient map. By~\cite[Section 9]{Sw83}, there is an isomorphism $\l \cong \Z[\zeta_{16}, j]$ where $x \mapsto \zeta_{16}$, $y \mapsto j$, and so $\l$ is a $\Z[\zeta_{16}+\zeta_{16}^{-1}]$-order in $\Q[\zeta_{16},j]$ with basis $\{1,x,j,xj\}$.
For ease of notation, we write $\zeta = \zeta_{16}$ and $\tau = \tau_{16} = \zeta_{16} + \zeta_{16}^{-1}$.
We will make use of the following identities in $\Lambda$: \[ \tau^2 = (x - x^7)^2 = 2 + x^2 - x^6, \quad \tau^4 - 5\tau^2 + 2 = 0, \quad x^2 = -1 + {\tau} x, \quad x^3 = -{\tau} + \left({\tau}^{2} - 1\right) x. \]

We now claim that $I_{8,3} = (g_1,g_2)$ is non-free where
\[ g_1 = 1-(1-xy)(x-1)(x^3+x^2-1), \quad g_2 = (x^3-x-1)(x-1)(x^3+x^2-1).\]
By \cref{lemma:ext-ideal}, $f_\#(I_{8,3}) \cong J_{8,3}$ where $J_{8,3} := f(I_{8,3}) = (f(g_1),f(g_2)) \le \l$.
For calculations, it will be useful to note that
\begin{align*} f(g_1) &= 
{\tau}^{2} - 2 + (-{\tau}^{3} + 3 {\tau} + 1) {x} + ({\tau}^{3} - 3 {\tau} - 1) j + (-{\tau}^{2} + 3) {x} j,
\\
f(g_2) &= 2 {\tau}^{3} + 2 {\tau}^{2} - 7 {\tau} - 5 + (-2 {\tau}^{3} + 6 {\tau} + 1) {x}.
\end{align*}

It follows from~\cite[Prop. 4.7]{Sw83} that the order $\Lambda$ is contained in exactly two maximal orders, one of which is the ring $\Gamma$ with $\Z[\tau]$-basis 
$(\gamma_1,\gamma_2, \gamma_3, \gamma_4) = (1, \zeta, (\zeta + j)/\tau, (1 + \zeta j)/\tau)$.

The following is established in~\cite[Proposition 4.3]{BHJ22} (see also~\cite[Corollary 2.4]{BJ08}):

\begin{prop}\label{prop:bhjcrit}
Let $K$ be a number field with ring of integers $\mathcal{O}_K$, let $A$ be a finite-dimensional semisimple $K$-algebra, let $\l$ be a $\mathcal{O}_K$-order in $A$, let $\Gamma$ be a maximal order in $A$ containing $\l$ and let $\mathfrak f$ be any proper two-sided $\Gamma$-ideal that is contained in $\Lambda$ with $\Lambda/\mathfrak f$ finite.

Let $I \subseteq \Lambda$ be a non-zero $\Lambda$-module with $I + \mathfrak f = \Lambda$ and $\Gamma \cdot I = (\alpha) \le \Gamma$ for some $\alpha \in \Gamma$.
Then $I$ is a free $\Lambda$-module if and only if the double coset of $[\alpha]$ in $(\Lambda/\mathfrak f)^\times \backslash (\Gamma/\mathfrak f)^\times / \Gamma^\times$ is trivial.
\end{prop}

\begin{lemma}\label{lem:q32_8_3}
$\Gamma \cdot J_{8,3} =  (-1 - \tau \zeta + (\zeta + j) - (1 + \zeta j)/\tau) \le \Gamma$ and $\Gamma \cdot J_{8,3} + (\tau) = \Lambda$.
\end{lemma}

\begin{proof}
Let $\alpha = -1 - \tau \zeta + (\zeta + j) - (1 + \zeta j)/\tau$. We have $\Gamma \cdot J_{8,3} = (f(g_1),f(g_2)) \le \Gamma$ and so $\Gamma \cdot J_{8,3} = (\alpha)$ is a consequence of the following identities:
 \begin{align*}
 f(g_1) & = ((\tau + 1) - \tau \zeta - (1 + \zeta j)/\tau) \cdot \alpha, \\
 f(g_2) &= ((-\tau^3 + 3\tau + 2) + (-\tau + 1)\zeta + (-\tau + 1) (\zeta + j)/\tau + (\tau^2 - \tau - 1) (1 + \zeta j)/\tau) \cdot \alpha, \\
 \alpha &= \alpha \cdot f(g_1) + ((-\tau^2 + 2) - \tau (\zeta + j) + (\tau^2 - 1) (1 + \zeta j))\cdot f(g_2).
 \end{align*}
 
The equality $\Gamma \cdot J_{8,3} + (\tau) = \Lambda$ follows from:
\[ 1 = (\tau^2 - 1)\zeta j \cdot \alpha_1 + \tau \cdot (\tau + (-\tau^2 + \tau + 1) \cdot (\zeta + j)/\tau -\tau^2 \cdot(1 + \zeta j)/\tau).\qedhere \]
\end{proof}

\begin{lemma}\label{lem:q32_8_3_notfree}
$J_{8,3}$ is non-free.
\end{lemma}

\begin{proof}
By Proposition~\ref{prop:bhjcrit} and Lemma~\ref{lem:q32_8_3}, it suffices to prove that the double coset of $[\alpha]$ in $(\Lambda/\tau\Gamma)^\times \backslash (\Gamma/\tau \Gamma)^\times /  \Gamma^\times$ is not trivial, where $\alpha = -1 - \tau \zeta + (\zeta + j) - (1 + \zeta j)/\tau$.
    We will make use of the isomorphism $\Gamma/\tau \Gamma \mapsto \operatorname{M}_2(\F_2)$, which maps
    \[ \zeta \mapsto \left(\begin{smallmatrix} 1 & 1 \\ 0 & 1 \end{smallmatrix}\right), \quad (\zeta + j)/\tau \mapsto \left(\begin{smallmatrix} 1 & 0 \\ 1 & 0 \end{smallmatrix}\right), \quad (1 + \zeta j)/\tau \to \left(\begin{smallmatrix} 1 & 1 \\ 1 & 1 \end{smallmatrix}\right). \]
    
    It is clear that $\Lambda/\tau \Gamma$ maps onto to the subring of upper triangular matrices and hence $(\Lambda/\tau \Gamma)^\times = \langle (\begin{smallmatrix} 1 & 1 \\ 0 & 1 \end{smallmatrix}) \rangle$.
    By \cite[Remark 3]{Sw62} and~\cite[Remark, p78]{Sw83}, the unit group of the conjugated maximal order $(1+ \zeta)^{-1} \Gamma (1 + \zeta)$ is equal to $\langle \Z[\tau]^\times, \zeta, j \rangle$ (with $\langle \zeta, j \rangle \cong Q_{32}$). Therefore $\Gamma^\times = \langle \Z[\tau]^\times, \zeta, \zeta j \rangle = \langle \Z[\tau]^\times, \zeta, j \rangle$.
    Since clearly any element of $\Z[\tau]^\times \subseteq \Gamma^\times$ is mapped to $\overline 1$ modulo $\tau$, it follows that
    $\Gamma^\times = \langle (\begin{smallmatrix} 1 & 1 \\ 0 & 1 \end{smallmatrix}) \rangle \subseteq (\Gamma/\tau\Gamma)^\times$.
    In particular the trivial double coset $(\Lambda/\tau \Gamma)^\times \cdot \overline 1 \cdot \Gamma^\times$ is equal to $\langle (\begin{smallmatrix} 1 & 1 \\ 0 & 1 \end{smallmatrix}) \rangle$.
    Finally the non-freeness of $J_{8,3}$ follows from
    $\overline \beta = (\begin{smallmatrix} 0 & 1 \\ 1 & 0 \end{smallmatrix})$.
\end{proof}

\begin{lemma}\label{lem:q32_iso}
  Let $\theta \colon x \mapsto x^{-1}, y \mapsto x^{-1}y$. Then $(J_{8,3})_\theta \not \cong \Z Q_{32}, J_{8,3}$.
\end{lemma}

\begin{proof}
Since the action of $\Aut(Q_{32})$ preserves free modules, Lemma~\ref{lem:q32_8_3_notfree} implies $(J_{8,3})_\theta \not\cong \Z Q_{32}$. By \cref{lemma:ext-ideal}, we have $\Gamma \otimes_\l J_{8,3} \cong \Gamma \cdot J_{8,3}$ which is free by \cref{lem:q32_8_3}. We will now show that $(J_{8,3})_\theta \not \cong J_{8,3}$ by proving that $\Gamma \otimes_\l (J_{8,3})_\theta \cong \Gamma \cdot (J_{8,3})_\theta$ is non-free.
We have $J_{8,3} = (f(g_1),f(g_2))$ and so \cref{cor:theta-action-ideal} implies that $\Gamma \cdot (J_{8,3})_\theta = (\alpha_1,\alpha_2)$ where $\alpha_1 = \theta^{-1}(f(g_1))$, $\alpha_2 = \theta^{-1}(f(g_2))$, i.e.
    \begin{align*}
    \alpha_1 = \tau + (\tau^3 - 3\tau - 1)\zeta + (\tau^2 - \tau - 3)j + \zeta j, \quad 
    \alpha_2 = (2\tau^3 - 6\tau - 1) + (2\tau^3 - 6\tau - 1)\zeta.
    \end{align*}

    It was shown in \cite[Lemma~1]{Sw62} that the $\Gamma$-ideal $J = ((1 + \zeta) \tau, (1 + \zeta) (\tau/\sqrt 2) (1 + j))$ is non-principal.
    We claim that $\Gamma \cdot (J_{8,3})_\theta = J \delta$ where  $\delta = \frac 1 2 ((-\tau^2 + \tau + 2) + \tau (\zeta + j + \zeta j))$, which implies that $\Gamma \cdot (J_{8,3})_\theta$ is non-principal.
    This follows from the identities:
    \begin{align*} (1 + \zeta) \tau \delta = ((-2\tau^2 + \tau + 1) + &(\tau^3 - \tau^2 - \tau + 1) \gamma_2 + (\tau^3 - 2\tau - 1) \gamma_3 + \tau^3 \gamma_4) \alpha_1 \\ &+ ((\tau^3 - 2\tau - 1) + (-2\tau^2 + 2) \gamma_2 + (\tau^3 - \tau^2 - \tau + 1) \gamma_3) \alpha_2,
    \end{align*}
    \begin{align*}
        (1 + \zeta) (\tau/\sqrt 2) (1 + j) = ((2\tau^2 - 1)  &+ (-\tau^3 + \tau^2 + 2\tau - 1) \gamma_2 + (-\tau^3 + 2\tau) \gamma_3 + (-\tau^3 + \tau) \gamma_4) \alpha_1 \\ &+ ((-\tau^3 + 3\tau)  + (\tau^2 - 1) \gamma_2 + (\tau^2 - 1) \gamma_3 -\tau^2 \gamma_4) \alpha_2,
    \end{align*}
    \begin{align*}
        \alpha_1 = ((-2\tau^2 - \tau + 3) &+ (\tau^3 + \tau^2 - 2\tau - 1) \gamma_2 -\tau^2 \gamma_3 + (\tau^3 - \tau) \gamma_4) (1 + \zeta)\tau \delta \\ &+ ((-\tau - 1) + (-\tau^3 + \tau^2 + 4\tau - 2) \gamma_2 + \tau \gamma_3) (1 + \zeta)(\tau/\sqrt 2)(1 + j) \delta,
    \end{align*}
    \begin{align*}
        \alpha_2 = ((\tau^2 - 2\tau - 4) &+ (-\tau^3 + 2\tau^2 + 4\tau - 2) \gamma_2 + (-\tau^3 - \tau^2 + \tau) \gamma_3 + (\tau^2 + \tau) \gamma_4) (1 + \zeta)  \tau \delta \\ &+ ((-\tau^3 + 2\tau - 3) + (\tau^2 + 2\tau - 1) \gamma_2 -\tau^2 \gamma_3 + \tau \gamma_4) (1 + \zeta)  (\tau/\sqrt 2)(1 + j) \delta. \qedhere
    \end{align*}
\end{proof}

It was shown in~\cite[Theorem IV]{Sw83} that $\lvert \SF_1(\Z Q_{32}) \rvert = 3$. By Lemma~\ref{lem:q32_iso}, this implies that the $\Z Q_{32}$-modules $\Z Q_{32}$, $I_{8,3}$ and $(I_{8,3})_\theta$ are representatives for the isomorphism classes in $\SF_1(\Z Q_{32})$. 
Hence $\SF_1(\Z Q_{32})/{\Aut(Q_{32})} = \{ [\Z Q_{32}], [I_{8,3}] \}$ and so $\lvert \SF_1(\Z Q_{32})/{\Aut(Q_{32})} \rvert =2$.

Since $J_{8,3}$ is non-free, $I_{8,3}$ is non-free which implies that $\mathcal{P}_8^{\std} \not \simeq \mathcal{P}_8(2;-2)$ by \cref{lemma:MP-homotopy}. Hence $\lvert \HT_{\min}(Q_{32}) \rvert \ge 2$ and so, since $\lvert \SF_1(\Z Q_{32})/{\Aut(Q_{32})}\rvert=2$, $\Psi$ is bijective. Thus $Q_{32}$ has the D2 property and $\HT_{\min}(Q_{32}) = \{\mathcal{P}_8^{\std}, \mathcal{P}_8(2;-2)\}$ as required.

\subsection{Presentations for $Q_{36}$} \label{ss:Q36}

In this section, we will show:

\begin{prop} \label{prop:Q36}
$Q_{36}$ has the {\normalfont D2} property if and only if there exists a balanced presentation $\mathcal{P}$ for $Q_{36}$ such that $\mathcal{P} \not \simeq \mathcal{P}_9^{\std}, \mathcal{P}_9(2;-2)$, and we have $\mathcal{P}_9^{\std} \not \simeq \mathcal{P}_9(2;-2)$.
\end{prop}

Let $\Lambda = \Z Q_{36}/(x^9+1)$ and recall that the map $f : \Z Q_{36} \twoheadrightarrow \Lambda$ induces a bijection
$f_\# : \SF_1(\Z Q_{36}) \to \SF_1(\Lambda)$. We have $x^9+1 = (x + 1) \psi$ where $\psi = \Phi_6(x)\Phi_{18}(x)$ and $\Phi_6(x) = x^2 - x + 1$, $\Phi_{18}(x) = x^6 - x^3 + 1$ are cyclotomic polynomials. Let $\l_{6,18}:= \Z Q_{36}/(\psi)$. By applying \cite[Example 42.3]{CR87}, and using the isomorphism $\Z Q_{36}/(x+1) \cong \Z[j]$, we obtain a pullback square
\[
\mathcal{R} = 
\begin{tikzcd}
	\l \ar[r,"i_2"] \ar[d,"i_1"] & \l_{6,18} \ar[d,"j_2"]\\
	\Z[j] \ar[r,"j_1"] & (\Z/9)[j]
\end{tikzcd}
\quad
\begin{tikzcd}
	x,y \ar[r,mapsto] \ar[d,mapsto] & x,y \ar[d,mapsto]\\
	-1,j \ar[r,mapsto] & -1,j
\end{tikzcd}
\]
of the form given in \cref{ss:modules-prelim}.
By \cite[Lemma 8.16 (p100) \& p96-99]{Sw83}, we have $\lvert \SF_1(\l) \rvert =4$ and $\lvert \SF_1(\l_{6,18}) \rvert =2$, i.e. $\SF_1(\l_{6,18}) = \{\l_{6,18},Q\}$ for some non-free stably free $\l_{6,18}$-module $Q$.
The ring $\Z[j]$ has stably free cancellation.
Consider the extension of scalars map
\[ g = (i_2)_\# : \SF_1(\l) \to \SF_1(\l_{6,18}). \]
This is surjective by \cite[Theorem A10]{Sw83} and so $\SF_1(\l) = g^{-1}(\l_{6,18}) \sqcup g^{-1}(Q)$.
By \cref{lemma:Milnor-square-ideals}, there is an injective map
$\varphi : g^{-1}(\l_{6,18}) \hookrightarrow \Z[j]^\times \backslash (\Z/9)[j]^\times/\l_{6,18}^\times$.
The image is given by
\[ \ker\left(q : \Z[j]^\times \backslash (\Z/9)[j]^\times/\l_{6,18}^\times \to \frac{K_1((\Z/9)[j])}{K_1(\Z[j]) \times K_1(\l_{6,18})}\right)\]
where $q$ is induced by the map $(\Z/9)[j]^\times \to K_1((\Z/9)[j])$.

For $a,b \in \Z$, consider the ideal $P_{(a,b)} = (a+by,1+x) \le \Z Q_{36}$.
By \cite[Proposition 2.2]{BW05}, if $(a^2+b^2,18)=1$, then $P_{(a,b)}$ is a projective $\Z Q_{36}$-module which is in $g^{-1}(\l_{6,18})$ and has corresponding element $[a+bj] \in \Z[j]^\times \backslash (\Z/9)[j]^\times/\l_{6,18}^\times$.
It was shown in \cite[Theorem 2.2]{BW08} that, if $a + by = k \cdot (c+dy)^2 \in (\Z/9)[j]$ for some $k,c,d \in \Z/9$, then $P_{(a,b)}$ is stably free.

\begin{lemma} \label{lemma:Q36-double-coset}
    The following holds:
    \begin{clist}{(i)}
    \item 
$\Z[j]^\times \backslash (\Z/9)[j]^\times/\l_{6,18}^\times = \{[1],[1+j],[1+2j],[1+3j],[1+5j],[1+6j]\}$
    \item 
$g^{-1}(\l_{6,18}) = \{\Z Q_{36}, P_{(-3,4)}, P_{(-3,2)}\}$. In particular, $|g^{-1}(\l_{6,18})|=3$ and $|g^{-1}(Q)|=1$.
\end{clist}
\end{lemma}

\begin{proof}
(i) By calculations in \cite[p98]{Sw83}, we have $\Z[j]^\times \backslash (\Z/9)[j]^\times/\l_{6,18}^\times = (\Z/9)[j]^\times/((\Z/9)^\times \cdot \langle j \rangle)$. The generators can then be computed directly.

(ii) Since $((-3)^2+4^2,18)=((-3)^2+2^2,18)=1$, $P_{(-3,4)}$ and $P_{(-3,2)}$ are both projective. By the discussion above, the modules correspond to $[-3+4j]$ and $[-3+2j]$ in $\Z[j]^\times \backslash (\Z/9)[j]^\times/\l_{6,18}^\times \cong (\Z/9)[j]^\times/((\Z/9)^\times \cdot \langle j \rangle)$ respectively We have $[-3+4j]=[1+3j]$ since $2j(-3+4j) = 1+3j$, and $[-3+2j] = [1+6j]$ since $4j(-3+2j)=1+6j$. By (i), this implies that $\Z Q_{36}$, $P_{(-3,4)}$ and $P_{(-3,2)}$ are pairwise non-isomorphic. 
On the other hand, since $-3+4j = (1+2j)^2 \in (\Z/9)[j]$ and $-3+2j = (2+5j)^2 \in (\Z/9)[j]$, $P_{(-3,4)}$ and $P_{(-3,2)}$ are both stably free.  

This implies that $|g^{-1}(\l_{6,18})| \ge 3$. Since $\SF_1(\l) = g^{-1}(\l_{6,18}) \sqcup g^{-1}(Q)$, $|g^{-1}(Q)| \ge 3$ and $\lvert\SF_1(\l)\rvert=4$, this now implies that $|g^{-1}(\l_{6,18})|=3$ and $|g^{-1}(Q)|=1$. It then follows that $g^{-1}(\l_{6,18}) = \{\Z Q_{36}, P_{(-3,4)}, P_{(-3,2)}\}$, as required.
\end{proof}

The analogue of \cref{lemma:induced-action} holds for $Q_{36}$. Hence the action of $\Aut(Q_{36})$ on $\SF_1(\l)$ induces an action on $\SF_1(\l_{6,18})$ via $g$. Since $\lvert\SF_1(\l_{6,18})\rvert=2$ and the action fixes the free module, it acts trivially on $\SF_1(\l_{6,18})$. Thus:
\[ \SF_1(\Z Q_{36})/{\Aut(Q_{36})} \leftrightarrow \SF_1(\l)/{\Aut(Q_{36})} \leftrightarrow g^{-1}(\l_{6,18})/{\Aut(Q_{36})} \sqcup g^{-1}(Q)/{\Aut(Q_{36})}. \]
Since $g^{-1}(Q)=\{P\}$ for some $P$ (by \cref{lemma:Q36-double-coset} (ii)), we have $g^{-1}(Q)/{\Aut(Q_{36})} = \{P\}$. Since \cite[Proposition 8.8]{Ni20a} applies, the action of ${\Aut(Q_{36})}$ on $g^{-1}(\l_{6,18})$ is as described by the analogue of \cref{lemma:induced-action}. By the proof of \cref{lemma:Q36-double-coset}, we have that 
\[ g^{-1}(\l_{6,18}) \leftrightarrow \{[1],[1+3j],[1+6j]\} \le \Z[j]^\times \backslash (\Z/9)[j]^\times/\l_{6,18}^\times.\]
Since $[1-3j] = [1+6j] \in (\Z/9)[j]^\times$, we have that
\[ g^{-1}(\l_{6,18})/ {\Aut(Q_{36})} \leftrightarrow \{[1],[1+3j]\}.\] 
This corresponds to the fact that $(P_{(-3,4)})_\theta \cong P_{(-3,2)}$ where $\theta = \theta_{1,1} : x \mapsto x$, $y 
\mapsto xy$. Hence
\[ \SF_1(\Z Q_{36})/{\Aut(Q_{36})} \leftrightarrow \{\Z Q_{36}, P_{(-3,4)}, P\} \]
and so $\lvert\SF_1(\Z Q_{36})/{\Aut(Q_{36})}\rvert=3$.

We now claim that $I_{9,3} = (g_1,g_2)$ is non-free where
\[ g_1 = 1-(1-xy)(x-1)(x^3+x^2-1), \quad g_2 = (x^3-x-1)(x-1)(x^3+x^2-1).\]
Since $f_\#$ is bijective, it suffices to show that $g(f_\#(I_{9,3}))$ is non-free. By \cref{lemma:ext-ideal}, $g(f_\#(I_{9,3})) \cong J_{9,3}$ where $J_{9,3} = i_2(f(I_{6,3})) = (i_2(f(g_1)), i_2(f(g_2))) \le \l_{6,18}$.
If $\Lambda_{18} = \Z Q_{36}/(\Phi_{18}(x))$ and $\Lambda_{6} = \Z Q_{36}/(\Phi_6(x))$, then we obtain the following pullback square
\[
\mathcal{R} = 
\begin{tikzcd}
	\l_{6,18} \ar[r,"i_2'"] \ar[d,"i_1'"] & \l_{18} \ar[d,"j_2'"]\\
	\Lambda_6 \ar[r,"j_1'"] & \l_6/3\l_6.
\end{tikzcd}
\]
of the form given in \cref{ss:modules-prelim} (see \cite[p96]{Sw83}).
Denote by $\psi \colon \SF_1(\Lambda_{6,18}) \to \SF_1(\Lambda_6) \times \SF_1(\l_{18})$
the map induced by extension of scalars.
By \cref{lemma:Milnor-square-ideals}, there is an injective map $\varphi' : \psi^{-1}(\Lambda_6, \l_{18}) \hookrightarrow \Lambda_6^\times \backslash (\l_6/3\l_6)^\times / \l_{18}^\times$.

\begin{lemma}\label{lemma:mp93-gens}
$i_1'(J_{9,3}) = (2x+(1+x)y) \le \l_6$ and $i_2'(J_{9,3}) = (1) \le \l_{18}$.
\end{lemma}

\begin{proof}
Let $\ol{g}_1 := i_2(f(g_1))$ and $\ol{g}_2 := i_2(f(g_2))$, so that $J_{9,3} = (\ol{g}_1,\ol{g}_2)$. The result for $i_1'(J_{9,3})$ follows from the following equalities in $\Lambda_{6}$:
$2 x + (1 + x) y = (-1 - y)i_1'(\ol{g}_1) + (1 - x)i_1'(\ol{g}_2)$
  and
  \[ i_1'(\ol{g}_1) = (x - x y)(2 x + (1 + x) y), \quad i_1'(\ol{g}_2) = (2 x + (2 - x) y)(2 x + (1 + x) y).\]

The result for $i_2'(J_{9,3})$ follows from the following equality in $\Lambda_{18}$:
\[ 1 = (-x y)i_2'(\ol{g}_1) + (1 + x^{2} + (x + x^{2} - x^{5}) y)i_2'(\ol{g}_2). \qedhere \]
\end{proof}

Let $u = j_1'(2x+(1+x)y)j_2'(1)^{-1} = 2x+(1+x)y \in \l_6/3\l_6$.
It was shown in~\cite[p. 96]{Sw83} that
  $\lvert\Lambda_6^\times \backslash (\l_6/3\l_6)^\times / \l_{18}^\times\rvert = 3$ and that $a = 1 + y$ and $c = 1 - y(x + 1)$ are representatives for the non-trivial double cosets.
Now $x - 1 = x^2 \in \Lambda_6^\times$ and
  \[ u \cdot (x - 1) - c = -3 + (3 - 3 x) y \in 3 \Lambda_6, \]
which implies that $u \in (\l_6/3\l_6)^\times$ and $[u] = [c] \ne 0 \in \Lambda_6^\times \backslash (\l_6/3\l_6)^\times / \l_{18}^\times$. By  \cref{lemma:Milnor-square-ideals,lemma:mp93-gens}, we obtain  
$\varphi'(J_{9,3}) = [u] = [c] \ne 0 \in \Lambda_6^\times \backslash (\l_6/3\l_6)^\times / \l_{18}^\times$ and so $J_{9,3}$ is non-free.

Since $J_{9,3}$ is non-free, $I_{9,3}$ is non-free which implies that $\mathcal{P}_9^{\std} \not \simeq \mathcal{P}_9(2;-2)$ by \cref{lemma:MP-homotopy}.  
Since $\lvert\SF_1(\Z Q_{36})/{\Aut(Q_{36})}\rvert=3$, $\Psi$ is bijective if and only if $\lvert\HT_{\min}(Q_{36})\rvert = 3$.
We have $\lvert\HT_{\min}(Q_{36})\rvert \ge 2$ and so $Q_{36}$ has the D2 property if and only if there exists a balanced presentation $\mathcal{P}$ for $Q_{36}$ such that $\mathcal{P} \not \simeq \mathcal{P}_9^{\std}, \mathcal{P}_9(2;-2)$. This completes the proof of \cref{prop:Q36}.

\subsection{Presentations for $Q_{40}$} \label{ss:Q40}

In this section, we will prove the following two results. By combining the former with Propositions \ref{prop:Q24}, \ref{prop:Q28}, \ref{prop:Q32} and \ref{prop:Q36}, this completes the proof of \cref{thmx:main-D2}. 

\begin{prop} \label{prop:Q40}
$Q_{40}$ has the {\normalfont D2} property if and only if there exist homotopically distinct balanced presentations $\mathcal{P}_1$, $\mathcal{P}_2$ for $Q_{40}$ such that $\mathcal{P}_i \not \simeq \mathcal{P}_{10}^{\std}, \mathcal{P}_{10}(2;-2), \mathcal{P}_{10}(3;5)$ for $i=1,2$, and we have that $\mathcal{P}_{10}^{\std}$, $\mathcal{P}_{10}(2;-2)$ and $\mathcal{P}_{10}(3;5)$ are homotopically distinct presentations for $Q_{40}$.
\end{prop}

\begingroup
\renewcommand\thethm{\ref{thmx:first-example}}
\begin{thm}
$\mathcal{P}_{10}(3;5) = \langle x, y \mid x^{10}y^{-2}, x^3yx^5y^{-1}x^{-2}yx^{-4}y^{-1} \rangle$ is not homotopy equivalent to any presentation for $Q_{40}$ of the form $\mathcal{P}_{10}(2;r)$ for $r \in \Z$.
\end{thm}
\setcounter{thm}{\value{thm}-1}
\endgroup

Recall from \cref{ss:Q24} (in the case $p=5$) that, if $\l = \Z Q_{40}/(x^{10}+1)$, then there is a bijection $\SF_1(\Z Q_{40}) \cong \SF_1(\l)$ and a surjection $g =(i_2)_\# \colon \SF_1(\Lambda) \to \SF_1(\Z[\zeta_{20},j])$.
By the proof of \cite[Lemma 8.12]{Sw83}, we have $\SF_1(\Z[\zeta_{20},j]) = \{\Z[\zeta_{20},j], P\}$ for some $P$ non-free and so $\SF_1(\Lambda) = g^{-1}(\Z[\zeta_{20},j]) \sqcup g^{-1}(P)$.
By the fundamental theorem of Milnor squares \cite{Mi71} (see also \cref{lemma:Milnor-square-ideals}), there are bijections
\[ \varphi : g^{-1}(\Z[\zeta_{20},j]) \to \H_{\Z}^\times \backslash \H_{\F_5}^\times \slash \Z[\zeta_{20},j]^\times, \quad  \varphi' : g^{-1}(P) \to \H_{\Z}^\times \backslash \H_{\F_5}^\times \slash {\Aut(P)} \]
where $\Aut(P)$ is the group of $\Z[\zeta_{20},j]$-module automorphisms of $P$.

\begin{lemma} \label{lemma:double-coset-calc-Q40}
We have the following:
\begin{clist}{(i)}
\item
$\H_{\Z}^\times \backslash \H_{\F_5}^\times \slash 
\Z[\zeta_{20},j]^\times = \{[1],[1+j],[1+k],[1+i+j],[1+2i+j],[1+2i+2j]\}$.
\item 
$\H_{\Z}^\times \backslash \H_{\F_5}^\times \slash {\Aut(P)} = \{[1],[1+i]\}$.
\end{clist}
\end{lemma}

\begin{proof}
(i) We will begin by investigating the image of the map $j_2 : \Z[\zeta_{20},j]^\times \to \H_{\F_5}^\times$. By \cite[Lemma 7.5 (b)]{MOV83}, we have $\Z[\zeta_{20},j]^\times = \Z[\zeta_{20}]^\times \cdot \langle j \rangle$. Clearly $j_2(\langle j \rangle) = \{\pm 1, \pm j\}$ and, since $j_2(\zeta_{20}) = i$, we have $j_2(\Z[\zeta_{20}]^\times) \le \{a+bi : a^2+b^2 \ne 0\} = \F_5^\times \cdot \langle 1+i, i \rangle$.
We claim that $j_2(\Z[\zeta_{20}]^\times) = \F_5^\times \cdot \langle 1+i, i \rangle$. Firstly, $\zeta_{20} \in \Z[\zeta_{20}]^\times$ and $j_2(\zeta_{20}) = i$. Secondly, since $20$ is not a prime power, $1-\zeta_{20}^r \in \Z[\zeta_{20}]^\times$ for all $r \in (\Z/20)^\times$ and we have $j_2(1-\zeta_{20}^3)=1+i$. Finally, since $\zeta_{10} = \zeta_{20}^2$, we have that $\Z[\zeta_{10}]^\times \le \Z[\zeta_{20},j]^\times$ and $j_2(\zeta_{10}) = i^2=-1$. If $k \ge 1$ and $k \in \F_5^\times$, then 
\[ u:=1+(-\zeta_{10})+(-\zeta_{10})^2 + \cdots + (-\zeta_{10})^{k-1} = \frac{1-(-\zeta_{10})^k}{1-(-\zeta_{10})} \in \Z[\zeta_{10}]^\times \]
with $j_2(u) = k$. Thus $\F_5^\times \cup \{ 1+i, i\} \subseteq j_2(\Z[\zeta_{20}]^\times)$ and the claim follows.
Generators for the double cosets $\H_{\Z}^\times \backslash \H_{\F_5}^\times \slash 
\Z[\zeta_{20},j]^\times$ can now be obtained by direct calculation.

(ii) By \cite[p95]{Sw83} (case $\Gamma_0^* = \wt I$), the unique surjection $\H_{\F_5}^\times \twoheadrightarrow \F_5^\times/(\F_5^\times)^2$ induces an isomorphism $\H_{\Z}^\times \backslash \H_{\F_5}^\times \slash {\Aut(P)} \cong \Z/2$. The surjection is induced by the norm map $N : \H_{\F_5}^\times \to \F_5^\times$. The result now follows from the fact that $\F_5^\times/(\F_5^\times)^2 = \{[1],[2]\}$ and $N(1+i) = 2$.
\end{proof}

By the analogue of \cref{lemma:induced-action}, the action of $\Aut(Q_{40})$ on $\SF_1(\l)$ induces an action on
$\SF_1(\Z[\zeta_{20},j])$. Since $\lvert\SF_1(\Z[\zeta_{20},j])\rvert=2$ and the action fixes the free module, it acts trivially on $\SF_1(\Z[\zeta_{20},j])$. Thus there is a one-to-one correspondence:
\[ \SF_1(\Z Q_{40})/{\Aut(Q_{40})} \leftrightarrow g^{-1}(\Z[\zeta_{20},j])/{\Aut(Q_{40})} \sqcup g^{-1}(P)/{\Aut(Q_{40})}. \]

By combining \cref{lemma:double-coset-calc-Q40} with \cite[Proposition 8.8]{Ni20a}, this gives a bijection
\begin{align*} g^{-1}(\Z[\zeta_{20},j]) /{\Aut(Q_{40})} & \leftrightarrow \{[1],[1+j],[1+k],[1+i+j],[1+2i+j],[1+2i+2j]\}/{\Aut(Q_{40})} \\
& \, = \{ [1], [1+j], [1+i+j], [1+2i+j]\}
\end{align*}
where $\theta_{a,b} \in \Aut(Q_{40})$ acts on the double cosets via the following action 
\begin{align*} \bar{\theta}_{a,b}([1+j]) &= 
\begin{cases}
 	[1+(-1)^{b_0}j] = [1+j], & \text{if $b=2b_0+1$}\\
	[1+(-1)^{b_0}k]=[1+k], & \text{if $b=2b_0$}
\end{cases} \\
\bar{\theta}_{a,b}([1+i+j]) &= 
\begin{cases}
 	[1+(-1)^{a_0}i + (-1)^{b_0}j] = [1+i+j], & \text{if $b=2b_0+1$}\\
	[1+(-1)^{a_0}i + (-1)^{b_0}k]=[1+i+j], & \text{if $b=2b_0$}
\end{cases} \\
\bar{\theta}_{a,b}([1+2i+j]) &= 
\begin{cases}
 	[1+2(-1)^{a_0}i +(-1)^{b_0}j] = [1+2i+j], & \text{if $b=2b_0+1$}\\
	[1+2(-1)^{a_0}i +(-1)^{b_0}k]=[1+2i+2j], & \text{if $b=2b_0$.}
\end{cases} 
\end{align*}
Here the equalities follow from the identities establish in the proof of \cref{lemma:double-coset-calc-Q40}.

Determining the action of $\Aut(Q_{40})$ on $g^{-1}(P)$ is more difficult. It will follow from \cref{prop:Q40-calculations} (ii) and (iv) that this action in non-trivial and so $\lvert g^{-1}(P)/{\Aut(Q_{40})}\rvert = 2$, which implies that $\lvert\SF_1(\Z Q_{40})/{\Aut(Q_{40})}\rvert=5$.

We will now turn to the proofs of \cref{prop:Q40} and \cref{thmx:first-example}.
First, using Magma, we can show that $\mathcal{P}_{10}(3;5)$ is a regular presentation for $Q_{40}$. Note that this does not follow from \cref{thm:even-coeffs} since the coefficients are not of the required form.
By \cref{lem:symmetry_mods}, every presentations for $Q_{40}$ of the form $\mathcal{P}_{10}(2;r)$ for $r \in \Z$ is homotopy equivalent to $\mathcal{P}_{10}(2;s)$ for some $1 \le s \le 5$.
We have that $\mathcal{P}_{10}(2;1) = \mathcal{P}_{10}^{\std}$, $\mathcal{P}_{10}(2;r) \simeq \mathcal{E}_{10,r}$ and it can be checked using Magma that $\mathcal{P}_{10}(2;r)$ is a regular presentation for $Q_{40}$ for $r \in \{2,3,4,5\}$ (see \cref{remark:after-MP-module}).

\begin{prop} \label{prop:Q40-calculations}
For each $(n_1,m_1) \in \{(2,2), (2,3), (2,4), (2,5), (3,5)\}$, 
there exists an ideal $I_{10}(n_1;m_1) \le \Z Q_{40}$ such that $\Psi(\mathcal{P}_{10}(n_1;m_1)) = [I_{10}(n_1;m_1)]$ and the following holds:
\begin{clist}{(i)}
    \item 
    $I_{10}(2; 2) \cong \Z G$.
    \item
    $(i_2 \circ f)_\#(I_{10}(2;r))$ is non-free for $r \in \{3,4,5\}$.
    \item 
    $I_{10}(3;5)$ is non-free, but $(i_2 \circ f)_\#(I_{10}(3;5))$ is free.
    \item 
    $I_{10}(2;3) \not \cong (I_{10}(2;3))_\theta$ where $\theta \colon x \mapsto x^{-1}, y \mapsto x^{-1}y$.
    \end{clist}
\end{prop}

\begin{proof}
(i) We have $\mathcal{P}_{10}(2;2) \simeq \mathcal{E}_{10,2} \simeq \mathcal{E}_{10,12}$ and so, by \cref{thm:MP-module}, we can take $I_{10}(2;2) := I_{10,12}$. By \cref{lem:mpgens}, we have that $f_\#(I_{10}(2;2)) \cong f(I_{10}(2;2))$ is generated by the two elements
    \[ \alpha = -x^{2} + \left(x - x^{9}\right) y, \quad \beta = -1 - x - 2 x^{2} - x^{3} - x^{4}.\]
    
    The claim follows from
    \[ 
    1 = \alpha + (-x^{3} + x^{4} + x^{5} - 2 x^{6} + 2 x^{8} - x^{9} + \left(-x^{2} + 2 x^{3} - 2 x^{5} + x^{6} + x^{7} - x^{8}\right) y) + \beta
    \]
    since this implies that $f(I_{10}(2;2)) = \Lambda$ and so $I_{10}(2;2)$ is free since $f_\#$ is a bijection. Note that we could have used $I_{10,2}$, but the formulae for $\alpha$, $\beta$ turned out to be simpler in the case $I_{10,12}$.

(ii) For brevity, we will consider only the case $r=5$. The cases $r=3,4$ can be dealt with similarly. We have $\mathcal{P}_{10}(2;5) \simeq \mathcal{E}_{2,-4} \simeq \mathcal{E}_{2,6}$.
By \cref{thm:MP-module} and \cref{lem:mpgens}, we can take $I_{10}(2;6) := I_{10,6}$ where $f_\#(I_{10}(2;6)) \cong f(I_{10}(2;6))$ is generated by the two elements:
    \begin{align*} 
        \alpha_6 &= 1 + x - x^{3} + x^{4} + x^{5} - x^{7} + x^{8} + \left(-1 + x^{3} - x^{4} + x^{6} + x^{7} - x^{8}\right) y, \\
\beta_6 &= 1 + 2 x + x^{2} - 2 x^{3} + x^{4} + 2 x^{5} + x^{6} - 2 x^{7} + 2 x^{9}.
    \end{align*}
We have $(i_2 \circ f)_\#(I_{10}(2;6)) \cong M$ where $M := i_2(f(I_{10}(2;6))) \le \Z[\zeta_{20},j]$, and we claim that $M$ is non-free.
Let $\Gamma_{20}$ be a maximal order of $\Z[\zeta_{20}, j]$.
Consider 
\[ u = (-\tau^2 + 3) + (4/5\tau^3 - 3\tau) \cdot \zeta_{20} + (-2/5\tau^3 + \tau)\cdot \zeta_{20} j. \]
This element has norm $1$ and satisfies $u^2 \neq 1$, $u^3 = 1$. Moreover $M \cdot u \subseteq M$ and $\langle u \rangle \cong C_3$ is not isomorphic to a subgroup of $Q_{40}$.
Since by \cite[Corollary 6.4]{Sw83} the assumption of Lemma~\ref{lemma:nonfree-metacrit} is satisfied for $n = 10$, the claim follows.

(iii) Similarly to Example~\ref{example:P_12}, $f_\#(I_{10}(3; 5)) \cong f(I_{10}(3; 5))$ is generated by the two elements
\begin{align*}
\alpha &= 1 + (1 - x y)(-2 - x + x^{2} + 2 x^{3} + x^{4} - x^{5} - 2 x^{6} - 2 x^{7} - 2 x^{8}), \\
\beta &= (1 + x + x^{2} + x^{6} + x^{7})(-2 - x^{1} + x^{2} + 2 x^{3} + x^{4} - x^{5} - 2 x^{6} - 2 x^{7} - 2 x^{8}). 
\end{align*}

Now $i_1(f(I_{10}(3;5))) = (1) \leq \H_\Z$, since in $\H_\Z$ we have
\[ 1  = i_1(\alpha) + (1 - x y) i_1(\beta). \]
On the other hand, we have $i_2(f(I_{10}(3;5))) = (\lambda_2) \leq \Z[\zeta_{20}, j]$, where $\lambda_2 = x^{3} - x^{4} + \left(-x^{3} - x^{6}\right) y) \in \Z[\zeta_{20}, j]$. This follows from
\[ \lambda_2 = (x^{3} - x^{4} - (x^{3} + x^{6}) y) \cdot i_2(\alpha) +  (1 - x^{2} + 2 x^{4} - x^{5} - x^{6} + x^{7} + (-1 + x - x^{3} + x^{4}) y) \cdot i_2(\beta) \]
and a similar computation, which proves the other inclusion.
Hence by Lemma~\ref{lemma:Milnor-square-ideals}, $f(I_{10}(3;5))$ corresponds to the double coset
\[ [j_2(\lambda_2)^{-1} ] = [1 - i + j + k] \in \H_\Z^\times \backslash \H_{\F_5}^\times / \Z[\zeta_{20}, j]^\times. \]
Since $-2(1 + i) \in j_2(\Z[\zeta_{20}, j])$ and $(1 - i + j + k) \cdot (-2)(1 + i) = 1 + j$, we therefore have that
$[1 - i + j + k ] = [1 + j]$
which implies that $f(I_{10}(3;5))$, and thus $I_{10}(3;5)$, is non-free by the result on $g^{-1}(\Z[\zeta_{20}, j])/{\Aut(Q_{40})}$ following Lemma~\ref{lemma:double-coset-calc-Q40}.

(iv) Let $J_{10, 3} = f(I_{10, 3}) \cong f_\#(I_{10, 3})$ and
let $J_{10, 3}' = J_{10, 3} \cdot \frac 1 5(2x + xy)$, which is isomorphic to $J_{10, 3}$ and satisfies $i_1(J_{10, 3}') = \H_\Z$.
The elements
\begin{align*}
\lambda_1 = -x + 2y, \quad \lambda_2 &= 17 - 3 x - 17 x^{2} + 4 x^{3} + 16 x^{4} - 4 x^{5} - 17 x^{6} + 4 x^{7} + 17 x^{8} - 4 x^{9}\\
&\quad + \left(1 + x - x^{2} - x^{3} + x^{4} + x^{5} + x^{6} - x^{7} + x^{8} - x^{9}\right) y
\end{align*}
satisfy $i_k(J_{10, 3}') \cdot i_k(\lambda_k) = i_k((J_{10, 3})_{\theta})$ and $(J_{10, 3})_\theta \subseteq J_{10, 3}' \lambda_2$.
Thus we have isomorphisms
\[ f_k \colon (i_k)_\#(J_{10,3}') \to (i_k)_\#((J_{10,3})_\theta), x \mapsto x i_k(\lambda_k). \]
Hence the induced automorphism $\bar f_2^{-1} \bar f_1$ of $(j_1 \circ i_1)(J_{10,3}')$ is given by right multiplication  with $(j_1(i_1(\lambda_1))(i_2 (j_2(\lambda_2))^{-1} = [i - j - k] \in \H_{\F_5}$.
Thus, by the theory of Milnor squares (see Remark~\ref{lemma:Milnor-square-modules}), $(J_{10, 3})_\theta$ corresponds to the double coset $[i - j - k] \in \H_\Z^\times \backslash  \H_{\F_5}^\times /{\Aut(i_2(J_{10, 3}'))}$.
As in Lemma~\ref{lemma:double-coset-calc-Q40}, we have $\H_{\Z}^\times \backslash \H_{\F_5}^\times /{\Aut((i_2)_\#(J_{10, 3}'))}\cong \F_5^\times/(\F_5^{\times})^2 \cong \Z/2\Z$, where the isomorphism is induced by the norm map. The claim now follows from $N(i - j - k) \equiv 3 \bmod 5$.
\end{proof}

\begin{proof}[Proof of \cref{prop:Q40}]
Let $G = Q_{40}$. By \cref{lemma:MP-homotopy}, there is an injective map $\Psi : \HT_{\min}(G) \to \SF_1(\Z G)/{\Aut(G)}$ which is bijective if and only if $Q_{40}$ has the D2 property, and where $\Aut(G)$ acts on $\SF_1(\Z G)$ in the standard way. 
By \cref{prop:Q40-calculations}, $I_{10}(2;3)$ and $I_{10}(3;5)$ are non-free. Since $i_2 \circ f$ respects that $\Aut(G)$ action and $(i_2 \circ f)_\#(I_{10}(3;5))$ is free whilst $(i_2 \circ f)_\#(I_{10}(2;3))$ is non-free, they are also not $\Aut(G)$-isomorphic. Hence $\mathcal{P}_{10}^{\std}$, $\mathcal{P}_{10}(2;-2)$ and $\mathcal{P}_{10}(3;5)$ are all distinct under $\Psi$ and so are homotopically distinct.
Since $\lvert \SF_1(\Z G)/{\Aut(G)}\rvert=5$, which is proven using \cref{prop:Q40-calculations} (ii) and (iv), $\Psi$ is bijective if and only if $\lvert \HT_{\min}(G)\rvert = 5$. We have $\lvert \HT_{\min}(G) \rvert \ge 3$ and so $G$ has the D2 property if and only if there exist balanced presentations $\mathcal{P}_1$, $\mathcal{P}_2$ for $Q_{40}$ such that $\mathcal{P}_i \not \simeq \mathcal{P}_{10}^{\std}, \mathcal{P}_{10}(2;-2), \mathcal{P}_{10}(3;5)$ for $i=1,2$.
\end{proof}

\begin{proof}[Proof of \cref{thmx:first-example}]
As above, \cref{lem:symmetry_mods} implies that every presentations for $Q_{40}$ of the form $\mathcal{P}_{10}(2;r)$ for $r \in \Z$ is homotopy equivalent to $\mathcal{P}_{10}(2;s)$ for some $1 \le s \le 5$. 
It follows from \cref{prop:Q40-calculations} that $I_{10}(3;5)$ is not $\Aut(G)$-isomorphic to $I_{10}(2;r)$ for any $1 \le r \le 5$. Hence $\mathcal{P}_{10}(3;5)$ is not homotopy equivalent to $\mathcal{P}_{10}(2;r)$ for any $r$ since they are distinct under $\Psi$.
\end{proof}

\subsection{Presentations for $Q_{44}$} \label{ss:Q44}

In this section, we will show:

\begin{prop} \label{prop:Q44}
$\mathcal{P}_{11}^{\std} \not \simeq \mathcal{P}_{11}(2;-2)$.
\end{prop}

Recall from \cref{ss:Q28} (in the case $p=11$) that, if $\l = \Z Q_{44}/(x^{11}+1)$, then there is a bijection $\SF_1(\Z Q_{44}) \cong \SF_1(\l)$ and a surjection $g =(i_2)_\# \colon \SF_1(\Lambda) \to \SF_1(\Z[\zeta_{22},j])$.
Note that $\Z[\zeta_{22},j]$ does not have stably free cancellation (see \cite[Lemma 6.4]{Ni20b}).
Since $|C(\Z[\zeta_{11}])|=1$, \cref{lemma:Q4p-dcoset} implies there is a bijection $g^{-1}(\Z[\zeta_{22},j]) \leftrightarrow \{[1], [1+2j], [1+6j] \}$.
By \cref{lemma:induced-action}, the action of $\Aut(Q_{44})$ induces an action on $\SF_1(\Z[\zeta_{22},j])$. Since the action fixes the free module, it acts on $g^{-1}(\Z[\zeta_{22},j])$. Hence, by \cite[Proposition 8.8]{Ni20a}, there is a bijection
\[ g^{-1}(\Z[\zeta_{22},j])/{\Aut(Q_{44})} \leftrightarrow \{[1], [1+2j], [1+6j] \}/{\Aut(Q_{44})} = \{ [1], [1+2j]\} \]
where $\theta_{a,b} \in \Aut(Q_{44})$ acts on the double cosets via the action $\bar{\theta}_{a,b}([1+xj]) = [1+(-1)^bxj]$.

We now claim that $I_{11,3} = (g_1,g_2)$ is non-free where
\[ g_1 = 1-(1-xy)(x-1)(x^3+x^2-1), \quad g_2 = (x^3-x-1)(x-1)(x^3+x^2-1).\]
Since $f_\#$ is bijective, it suffices to show that $f_\#(I_{11,3})$ is non-free.

By \cref{lemma:ext-ideal}, $f_\#(I_{11,3}) \cong J_{11,3}$ where $J_{11,3} := f(I_{11,3}) = (f(g_1),f(g_2)) \le \l$.
We also have $(i_2)_\#(J_{11,3}) \cong i_2(J_{11,3}) \le \Z[\zeta_{22},j]$.

\begin{lemma}
The module $i_2(J_{11, 3}) \leq \Z[\zeta_{22}, j]$ is not free.
\end{lemma}

\begin{proof}
    Let $I = i_2(J_{11, 3})$, $\overline g_j = i_2(f(g_j))$ for $j = 1, 2$, and let
    \begin{align*} u =\ (1/23) \cdot (&18 + 13 x - 11 x^{2} + 4 x^{3} - 10 x^{4} + 7 x^{5} - 6 x^{6} + 3 x^{7} - 9 x^{8} + 2 x^{9} \\
    &+ (7 - 8 x + 7 x^{2} + 5 x^{4} + 4 x^{5} + 6 x^{6} - 6 x^{7} - 4 x^{8} - 5 x^{9}) y) 
    \end{align*}
    We have $\operatorname{nr}(u) = 1$, and $I u \subseteq I$ follows by observing that
    \begin{align*}
       m_1 &= -x^{3} - x^{5} + x^{8} + (2 - x - x^{3} + x^{4} + x^{6} - x^{7} + x^{8} - x^{9}) y, \\
       m_2 &= x + x^{3} + x^{5} + x^{9} + (-1 - x^{5} + x^{9}) y,\\
       n_1 &= 1 - x + x^{2} - 2 x^{3} + x^{4} - x^{5} + x^{6} - x^{7} + 2 x^{8} - x^{9} - x^{8} y,\\
       n_2 &= -1 + x - x^{2} + 2 x^{3} - 2 x^{4} + 2 x^{5} - x^{6} + x^{7} - x^{8} + x^{9} + (x + x^{7} + x^{9}) y
    \end{align*}
    satisfy $\overline g_1 u = m_1 \overline g_1 + m_2 \overline g_2$ and $\overline g_2 u = n_1 \overline g_1 + n_2 \overline g_2$. As $u^3 = 1$ and $u^2 = 1$ we have $\langle u \rangle \cong C_3$.
    Since by \cite[Corollary 6.4]{Sw83} the assumption of Lemma~\ref{lemma:nonfree-metacrit} is satisfied for $n = 11$, we obtain a contradiction.
\end{proof}

Since $I_{11,3}$ is non-free, we now have that $\mathcal{P}_{11}^{\std} \not \simeq \mathcal{P}_{11}(2;-2)$ by \cref{lemma:MP-homotopy}.

\subsection{Presentations for $Q_{48}$} \label{ss:Q48}

In this section, we will prove the following two results:

\begin{prop} \label{prop:Q48-main} 
$\mathcal{P}_{12}^{\std}$, $\mathcal{P}_{12}(2;-2)$, $\mathcal{P}_{12}(4;-4)$ and $\mathcal{P}_{12}(2,2;-2,-2)$ are homotopically distinct presentations for $Q_{48}$.  
\end{prop}

\begingroup
\renewcommand\thethm{\ref{thmx:a-non-MP-presentation}}
\begin{thm}
$\mathcal{P}_{12}(2,2;-2,-2) = \langle x, y \mid x^{12}y^{-2}, (x^{2}yx^{-2}y^{-1})^2x^{-3}yx^{5}y^{-1} \rangle$ is not homotopy equivalent to any presentation for $Q_{48}$ of the form $\mathcal{P}_{12}(n_1;m_1)$ for $n_1,m_1 \in \Z$.\end{thm}
\setcounter{thm}{\value{thm}-1}
\endgroup

Let $\Lambda = \Z Q_{48}/(x^{12} + 1)$ and let $f \colon \Z Q_{48}  \to \Lambda$ denote the quotient map.
To distinguish isomorphism classes of $\Z Q_{48}$-modules, we will make use of the induced map $f_\# \colon \SF_1(\Z Q_{48}) \to \SF_1(\Lambda)$ which is surjective by \cite[Theorem A10]{Sw83}. 

Let $\Lambda_8 = \Z[\zeta_8, j]$ and $\Lambda_{24} = \Z[\zeta_{24}, j]$.
Then $x^{12} + 1 = \Phi_{8}(x)\Phi_{24}(x) = (x^4 + 1)(x^8 - x^4 + 1)$
and the isomorphisms $\Z Q_{48}/(\Phi_8(x)) \cong \Z[\zeta_{8}, j]$, $\Z Q_{48}/(\Phi_{24}(x)) \cong \Z[\zeta_{24},j]$ induce a pullback square
\[
\mathcal{R} = 
\begin{tikzcd}
	\Lambda \ar[r,"i_2"] \ar[d,"i_1"] & \Lambda_{24} \ar[d,"j_2"]\\
	\Lambda_8 \ar[r,"j_1"] & \Lambda_8/3\Lambda_8
\end{tikzcd}
\quad
\begin{tikzcd}
	x,y \ar[r,mapsto] \ar[d,mapsto] & \zeta_{24},j \ar[d,mapsto]\\
	\zeta_{8},j \ar[r,mapsto] & \overline{\zeta_8},\overline{j}
\end{tikzcd}
\]
which is of the form given in \cref{ss:modules-prelim}. By~\cite[Section 7]{Sw83} we know that $\Lambda_8$ has stably free cancellation and that $\SF_1(\Lambda_{24}) = \{\Lambda_{24}, P\}$ for some non-free stably free $\Lambda_{24}$-modulo $P$.
Consider the extension of scalars map
$g  = (i_2)_\# \colon \SF_1(\Lambda) \to \SF_1(\Lambda_{24})$. This is surjective by \cite[Theorem A10]{Sw83}, and so 
$\SF_1(\Lambda) = g^{-1}(\Lambda_{24}) \sqcup g^{-1}(P)$.
Note that the analogue of \cref{lemma:induced-action} holds for $Q_{48}$. In particular, we have
\[ \SF_1(\Z Q_{48})/{\Aut(Q_{48})} \dhxrightarrow[]{g} \SF_1(\l)/{\Aut(Q_{48})} \xrightarrow[]{\cong} g^{-1}(\Lambda_{24})/{\Aut(Q_{48})} \sqcup g^{-1}(P)/{\Aut(Q_{48})}. \]

In order to prove \cref{thmx:a-non-MP-presentation}, we will need the following.

\begin{lemma} \label{lemma:height1}
If $\mathcal{P}$ is a height one regular presentation for $Q_{48}$, then $\mathcal{P} \simeq_{Q^*} \mathcal{P}_{12}(n_1;m_1)$ where
\[ (n_1, m_1) \in \{(1, 1), (2,3), (2,4), (2,6), (3, 3),  (3,4), (3,5), (3,6), (4, 4), (4,5), (4,6), (5,6), (6, 6)\}. \]
\end{lemma}

\begin{remark}
Conversely, Magma computations imply that $\mathcal{P}_{10}(n_1;m_1)$ is a regular presentation for $Q_{48}$ for all pairs $(n_1,m_1)$ listed above. However, we will not need this to prove \cref{thmx:a-non-MP-presentation}.
\end{remark}

\begin{proof}
By \cref{corollary:heightone} (i) and \cref{lem:symmetry_mods}, $\mathcal{P} \simeq_{Q^*} \mathcal{P}_{12}(n_1;m_1)$ where $1 \le n_1 \le m_1 \le 6$. If $n_1 = 1$, then $\mathcal{P}_{12}(n_1;m_1) \simeq_Q \mathcal{P}_{12}(1;1)$ by \cref{corollary:heightone} (ii).
It remains to show that $\mathcal{P}_{12}(n_1;m_1)$ does not present $Q_{48}$ for the pairs $(n_1;m_1) \in \{ (2,2), (2,5), (5,5)\}$. By \cref{lem:symmetry_mods}, 
we have $\mathcal{P}_{12}(2;2) \simeq_{Q^*} \mathcal{E}_{12,2}$ and $\mathcal{P}_{12}(2;5) \simeq_{Q^*} \mathcal{E}_{12,5}$. 
Since $n=12$ has $3 \mid n$ and $n \not \equiv 2 \mod 4$, \cref{prop:MP-not-presentation} implies that $\mathcal{E}_{12,2}$, $\mathcal{E}_{12,5}$ and $\mathcal{P}_{12}(5;5)$ do not present $Q_{48}$.
\end{proof}

It now suffices to prove that $\Psi(\mathcal{P}_{12}(2,2;-2,-2)) \ne \Psi(\mathcal{P}) \in \SF_1(\Z Q_{48})/{\Aut(Q_{48})}$ for all $\mathcal{P}$ given in \cref{lemma:height1}. We can use \cref{thm:Psi(P_R)} to determine $\wh \Psi(\mathcal{P}) \in \SF_1(\Z Q_{48}) /{\sim}$ where $\sim$ denotes the equivalence relation defined in \cref{lemma:SF-eq-rel}.
However, without knowing that stably free Swan modules are free over $\Z Q_{48}$, we do not know if $\sim$ coincides with equivalence modulo the action of $\Aut(Q_{48})$. We will now give a method for circumventing this issue.

First note that the action of $\Aut(Q_{48})$ on $\Z Q_{48}$ induces an action on $\l$ via $f$, which we denote by $\bar{\cdot}$. By \cite[Corollary 8.4]{Ni20a}, this implies that $f_\#$ induces a map
\[ F = f_\# : \SF_1(\Z Q_{48})/{\Aut(Q_{48})} \, \to \, \SF_1(\l)/{\Aut(Q_{48})} \]
where $(\,\cdot\,) /{\Aut(Q_{48})}$ denotes the set of orbits under the standard actions of ${\Aut(Q_{48})}$.

By \cref{lemma:ext-ideal}, $f_\#((N,r)) \cong f((N,r)) = (f(N),f(r)) = (f(r)) \le \l$ is principal since $x^{12}+1 \mid N$.
In particular, if $P \cong (N,r) \otimes  Q_\theta$ for $P, Q \in \SF_1(\Z Q_{48})$ and $\theta \in \Aut(Q_{48})$, then $f_\#(P) \cong (f_\#(Q))_{\bar{\theta}}$ by \cite[Lemma 8.3]{Ni20a}. In particular, $f_\#$ induces another map
\[ F' = f_\# : \SF_1(\Z Q_{48})/{\sim} \,\, \to \, \SF_1(\l)/{\Aut(Q_{48})}. \]
There is a natural quotient map $\SF_1(\Z Q_{48})/{\Aut(Q_{48})} \twoheadrightarrow \SF_1(\Z Q_{48})/{\sim}$ and, via this map, the discussion above implies that $F$ factors through $F'$. Thus, if $\mathcal{P}$ is a balanced presentation for $Q_{48}$, then $F(\Psi(\mathcal{P})) = F'(\wh \Psi(\mathcal{P}))$, i.e. $f_\#(\Psi(\mathcal{P})) = f_\#(\wh \Psi(\mathcal{P})) \in \SF_1(\l)/{\Aut(Q_{48})}$.

We now aim to establish the following result which will imply \cref{prop:Q48-main} when combined with \cref{lemma:height1}. In each case, the module is determined using the form given for $\wh \Psi(\mathcal{P})$ in \cref{thm:Psi(P_R)} and the equality $f_\#(\Psi(\mathcal{P})) = f_\#(\wh \Psi(\mathcal{P}))$ established above. We say that an ideal $I$ is an \textit{ideal representative} for a balanced presentation $\mathcal{P}$ if $\Psi(\mathcal{P}) = [I] \in \SF_1(\Z Q_{48})/{\Aut(Q_{48})}$.

\begin{prop} \label{prop:Q48-main-detailed}
Let $\mathcal{P}$ be a presentation for $Q_{48}$ of the form $\mathcal{P}_{12}(n_1;m_1)$ for 
\[ (n_1, m_1) \in \{(1, 1), (2,3), (2,4), (2,6), (3, 3),  (3,4), (3,5), (3,6), (4, 4), (4,5), (4,6), (5,6), (6, 6)\}. \]
or $\mathcal{P}_{12}(2,2;-2,-2)$, and let $I_\mathcal{P}$ be an ideal representative for $\mathcal{P}$. Then:
\begin{clist}{(i)}
\item 
If $\mathcal{P} \in \{\mathcal{P}_{12}(1; 1), \mathcal{P}_{12}(2;4), \mathcal{P}_{12}(3; 3), \mathcal{P}_{12}(4; 4), \mathcal{P}_{12}(4;5), \mathcal{P}_{12}(6; 6), \mathcal{P}_{12}(2,2;-2,-2)\}$, then $(i_2 \circ f)_\#(I_\mathcal{P}) \in \SF_1(\l_{24})$ is free. It is non-free for the remaining seven presentations.
\item 
There is a function $q : g^{-1}(\l_{24})/{\Aut(Q_{48})} \to \mathcal{D}$ where $\mathcal{D}$ is a set containing distinct elements $a,b,c \in \mathcal{D}$ such that
\[ q(f_\#(I_\mathcal{P})) = 
\begin{cases}
a, & \text{if $\mathcal{P} \in \{\mathcal{P}_{12}(1; 1), \mathcal{P}_{12}(3; 3), \mathcal{P}_{12}(4; 4), \mathcal{P}_{12}(6;6)\}$} \\
b, & \text{if $\mathcal{P} \in \{\mathcal{P}_{12}(2; 4), \mathcal{P}_{12}(4; 5) \}$}\\
c, & \text{if $\mathcal{P} = \mathcal{P}_{12}(2,2;-2,-2)$.} \\
\end{cases}
\]
\end{clist}    
\end{prop}

The proof of \cref{prop:Q48-main-detailed} (i) will be split between the following two lemmas. We omit the case $\mathcal{P}_{12}(1;1) = \mathcal{P}_{12}^{\std}$ since $I_{\mathcal{P}_{12}^{\std}} \cong \Z Q_{48}$ is already known to be free.

\begin{lemma}\label{lemma:q48-nonfree}
Let $\mathcal{P} = \mathcal{P}_{12}(n_1;m_1)$ for
$(n_1,m_1) \in \{ (2,3),(2,6),(3,4),(3,5),(3,6),(4,6),(5,6) \}$
and let $I_\mathcal{P}$ be an ideal representative for $\mathcal{P}$. Then $(i_2 \circ f)_\#(I_\mathcal{P})$ is non-free.
\end{lemma}

\begin{proof}
We first consider the case $\mathcal{P} = \mathcal{P}_{12}(2;3)$. By \cref{lemma:operations,lem:symmetry_mods}, we have $\mathcal{P}_{12}(2;3) \simeq_Q \mathcal{E}_{12,10}$. Corollary~\ref{thm:MP-module} implies $\wh \Psi(\mathcal{P}) = \wh \Psi(\mathcal{E}_{12,10}) = [I_{12,10}]$ where $I_{12,10} = (g_1,g_2) \le \Z Q_{48}$ for generators $g_1,g_2$ described in \cref{thm:MP-module}.
Since $f_\#(\Psi(\mathcal{P})) \cong f_\#(\wh \Psi(\mathcal{P}))$, we have $(i_2 \circ f)_\#(I_\mathcal{P}) \cong (i_2 \circ f)_\#(I_{12,10})$ where $I_\mathcal{P}$ is an ideal representative for $\mathcal{P}$. 
By \cref{lemma:ext-ideal}, $(i_2 \circ f)_\#(I_{12,10}) \cong i_2(f(I_{12,10})) = (\ol{g}_1,\ol{g}_2) \le \l_{24}$ where $\ol{g}_j := i_2(f(g_j))$ for $j=1,2$ and
\begin{align*}
  \ol{g}_1 &=  -3 + x - x^{2} - 3 x^{3} + 3 x^{4} + x^{5} - x^{6} + 3 x^{7}, \\
  \ol{g}_2 &=  2 - 2 x^{2} - x^{5} + x^{6} + 2 x^{7} + (1 - x + 2 x^{3} - x^{4} + 2 x^{6} - x^{7}) y.
\end{align*}

Consider the elements $u, v \in \Q[\zeta_{24}, j]$ defined by
\[ u = -1 + x^{4}, \quad  v = 2 - x^{2} - x^{3} + x^{5} + x^{6} + \left(1 - x - x^{3} - 2 x^{4} + 2 x^{5} + x^{6}\right) y. \]
We have $u^3 = v^3 = 1$, $u^2 = -x^4 \neq v$ and $\operatorname{nr}(u) = \operatorname{nr}(v) = 1$. The inclusions $Iu \subseteq I$ and $Iv \subseteq I$ follow by observing that
\[
m_{1} = -1 + x^{4}, \quad m_{2} = \left(1 - x^{2} - x^{4} + x^{5} + x^{6}\right) y, \quad n_{1} = 0, \quad n_{2} = -1 + x^{4}
\]
satisfy $\ol{g}_1 u = m_{1}\ol{g}_1 + m_2 \ol{g}_2$ and $\ol{g}_2 u = n_{1} \ol{g}_1 + n_2 \ol{g}_2$.
A quick calculation shows that similar elements exist for $v$.
As the order of $\langle u, v \rangle$ is divisible by $9$, it cannot be isomorphic to a subgroup of $Q_{48}$.
The claim now follows from Lemma~\ref{lemma:nonfree-metacrit}, since the assumption for $n = 12$ is satisfied by~\cite[Corollary 6.4]{Sw83}.

We next consider the case $\mathcal{P} = \mathcal{P}_{12}(3;6)$. By \cref{lem:symmetry_mods}, $\mathcal{P} \simeq_Q \mathcal{P}_{12}(6;3)$ and so \cref{thm:Psi(P_R)} and Example~\ref{example:P_12} imply $\wh \Psi(\mathcal{P}) = \wh \Psi(\mathcal{P}_{12}(6;3)) = [I]$ for some ideal $I = (g_1,g_2) \le \Z Q_{48}$.
Similarly to the case above, we get that $(i_2 \circ f)_\#(I) \cong i_2(f(I)) = (\ol{g}_1,\ol{g}_2) \le \l_{24}$ where
\begin{align*}
  \ol{g}_1 &= 3 + x + 7 x^{2} + 5 x^{3} - 2 x^{4} - 2 x^{5} - 6 x^{6} - 5 x^{7} + \left(1 - 5 x^{2} - 7 x^{3} - 2 x^{4} - 2 x^{5} + x^{7}\right) y, \\
  \ol{g}_2 &= -4 - 2 x + 8 x^{2} + 14 x^{3} + 12 x^{4} + 11 x^{5} + 7 x^{6} + 2 x^{7}.
\end{align*}
We again use Lemma~\ref{lemma:nonfree-metacrit} to show that $i_2(f(I))$ is non-free.
Let $u, v \in \Q[\zeta_{24}, j]$ be defined by
\begin{align*} u &= \frac{1}{5}\left( -3 + x^{4} - 2 x^{5} - x^{6} - 2 x^{7} + (-1 + 2 x - x^{2} - 2 x^{3} + x^{7}) y\right),\\
  v &= \frac{1}{5} \left(-2 - 2 x + 2 x^{3} - x^{4} - x^{6} - 2 x^{7} + (-1 + 2 x + 2 x^{3} + x^{4} - x^{6} - x^{7}) y\right).
\end{align*}
As before, one verifies $u^3 = v^3 = 1$, $u^2 = -x^4 \neq v$ and $\operatorname{nr}(u) = \operatorname{nr}(v) = 1$ and it remains to show that $Iu \subseteq I$ and $Iv \subseteq I$.
The former follows by observing that
\begin{align*}
  m_{1} &= 2 + x - 2 x^{2} - 2 x^{4} + x^{5} + 4 x^{6} - 2 x^{7} + (3 + x + 2 x^{2} + 5 x^{3} - 2 x^{4} - 2 x^{5} - x^{6} - 2 x^{7}) y, \\
  m_{2} &= -1 - x + 2 x^{3} + 4 x^{4} - 2 x^{5} + (3 x - 2 x^{3} - 3 x^{4} - x^{6} + 2 x^{7}) y, \\
  n_{1} &= x + 3 x^{2} + x^{3} + x^{4} - x^{5} - 4 x^{6} + x^{7} + (2 - x - 3 x^{2} - x^{3} - x^{4} + 3 x^{6} - x^{7}) y,\\
  n_{2} &= 1 + 2 x - 2 x^{2} + x^{3} - 4 x^{4} + x^{5} + x^{6} + x^{7} + (-3 - 2 x + 2 x^{2} + x^{3} + 3 x^{4} - 2 x^{7}) y,
\end{align*}
satisfy $\ol{g}_1 u = m_{1}\ol{g}_1 + m_2 \ol{g}_2$ and $\ol{g}_2 u = n_{1} \ol{g}_1 + n_2 \ol{g}_2$. Similarly for $v$, so $(i_2 \circ f)_\#(I_\mathcal{P})$ is non-free.

The other presentations are dealt with analogously but the details are omitted for brevity. Note that $\mathcal{P}_{12}(2;6) \simeq_Q \mathcal{E}_{12,6}$ by \cref{corollary:heightone,lem:symmetry_mods}, and so this can be dealt with similarly to the case $\mathcal{P}_{12}(2;3)$. 
The other presentations are not $Q$-equivalent to a Mannan--Popiel presentation (in any obvious way) and so are dealt with similarly to $\mathcal{P}_{12}(3;6)$.
\end{proof}

\begin{lemma}\label{lemma:l8l24gens}
Let $\mathcal{P} \in \{\mathcal{P}_{12}(2;4), \mathcal{P}_{12}(3; 3), \mathcal{P}_{12}(4; 4), \mathcal{P}_{12}(4;5), \mathcal{P}_{12}(6; 6), \mathcal{P}_{12}(2,2;-2,-2)\}$. There exists an ideal representative $I_\mathcal{P}$ such that $i_1(f(I_\mathcal{P})) = (\lambda_1) \le \l_8$, $i_2(f(I_\mathcal{P})) = (\lambda_2) \le \l_{24}$ where
  \[ 
(\lambda_1,\lambda_2) = 
  \begin{cases}
  (1 + x^{2} + (1 + x - x^{3}) y,1), &\text{if }\mathcal{P} = \mathcal{P}_{12}(2;4)  \\
  (-x + x^{2} - x^{3} + (1 + x^{2}) y, 1 + x + x^{3} - x^{7} + (1 - x^{4} - x^{6}) y), &\text{if }\mathcal{P} = \mathcal{P}_{12}(3;3)  \\
  (2 + 2 x + 2 x^{2} + y, x - x^{3} + x^{7} + (-1 - x + x^{4} - x^{6}) y), &\text{if }\mathcal{P} = \mathcal{P}_{12}(4;4)  \\
  (x + (2 + 2 x - 2 x^{3}) y, 1), &\text{if }\mathcal{P} = \mathcal{P}_{12}(4;5)  \\
  (-x + x^{2} - x^{3} + (1 + x^{2}) y, -x + x^{3} - x^{7} + (2 x + x^{2} + x^{4} - x^{6}) y), &\text{if }\mathcal{P} = \mathcal{P}_{12}(6;6) \\
  (1, -x + x^{3} + 2 x^{6} - x^{7} + (1 - x + x^{2})y) ,& \text{otherwise.}
  \end{cases}
  \]
Furthermore, for any choice of ideal representative $I_\mathcal{P}'$, we have that $(i_2 \circ f)_\#(I_\mathcal{P}')$ is free.
\end{lemma}

\begin{proof}
We first consider the case $\mathcal{P} = \mathcal{P}_{12}(3;3)$. Similarly to the proof of \cref{lemma:q48-nonfree}, combining \cref{example:P_12} with \cref{lemma:ext-ideal}, \cref{thm:Psi(P_R)} and the fact that $f_\#(\Psi(\mathcal{P})) \cong f_\#(\wh \Psi(\mathcal{P}))$ implies that there exists an ideal representative $I_\mathcal{P}$ such that $f_\#(I_\mathcal{P}) \cong f(I_\mathcal{P}) = (g_1,g_2) \le \l$ where
  \[
    g_1 = 3 x - 2 x^{3} - 2 x^{4} + 4 x^{6} + x^{7} - x^{8} - 2 x^{9} - x^{10} + x^{11} + (-3 + x + x^{2} - x^{3} - 2 x^{4} - x^{5} + x^{6} + 4 x^{7} - 2 x^{9} - 2 x^{10}) y
  \]
  \[ g_2 = -4 + 3 x + 6 x^{2} + 3 x^{3} - 8 x^{4} - 9 x^{5} + 2 x^{6} + 8 x^{7} + 7 x^{8} - 2 x^{9} - 6 x^{10}. \hspace{41mm} \]
The claim for $i_1(f(I_{\mathcal{P}}))$ now follows from:
\begin{align*}
    \lambda_1 &= (-x + x^{2} - x^{3} + (1 + x^{2}) y)g_1 + (x - x^{2} + x^{3} + (1 - x^{2} + x^{3}) y)g_2 \\
  g_1 &= (-1 - x - x^{2} - 2 x^{3} + (-3 - x + 2 x^{3}) y) \lambda_1 \\
  g_2 &= (-1 + 2 x + 2 x^{2} + x^{3} + (-5 - 7 x - 5 x^{2} - x^{3}) y) \lambda_1.
\end{align*}
For $i_2(f(I_{\mathcal{P}}))$, the claim follows from:
  \begin{align*}
    \lambda_2 &= (1 + x + x^{3} - x^{7} + (1 - x^{4} - x^{6}) y)g_1 + (x^{5} + x^{7} + (-1 + x^{4} + x^{6}) y)g_2 \\
    g_1 &= (3 x + 2 x^{2} - 3 x^{3} - 2 x^{4} - x^{5} + 3 x^{7} + (-2 + x + 2 x^{2} - 2 x^{5} - x^{6} + 2 x^{7}) y)\lambda_2 \\
    g_2 &= (-6 + 2 x + 8 x^{2} + x^{3} - x^{4} - 6 x^{5} - 3 x^{6} + 5 x^{7} + (2 - x - 3 x^{2} + x^{4} + 2 x^{5} - x^{7}) y )\lambda_2.
  \end{align*}
  
We next consider the case $\mathcal{P} = \mathcal{P}_{12}(2, 2; -2, -2)$. Similarly to the case above, \cref{example:P_12} implies there exists an ideal representative $I_\mathcal{P}$ such that $f_\#(I_\mathcal{P}) \cong f(I_\mathcal{P}) = (g_1,g_2) \le \l$ where
\begin{align*}
   g_1 = 1 &- x + x^{2} + 3 x^{3} - 2 x^{4} - x^{5} + 7 x^{6} + 6 x^{7} - 2 x^{8} - 2 x^{9} + 2 x^{10} + \\
   & \hspace{20mm} (1 + 2 x^{3} - 2 x^{4} - 2 x^{5} + 6 x^{6} + 7 x^{7} - x^{8} - 2 x^{9} + 3 x^{10} + x^{11}) y
  \end{align*}
  \[ g_2 = 8 + 14 x + 2 x^{2} + x^{4} - 2 x^{5} + 7 x^{6} + 17 x^{7} + 2 x^{8} - 9 x^{9} + 6 x^{10} + 11 x^{11}. \] 
  The claim for $i_1(f(I_{\mathcal{P}}))$ now follows from:
  \[ 1 = g_1 + (x - x^{2} + x^{3} + (-1 + x^{2} - x^{3}) y)g_2 \]
 where $\lambda_1=1$.
  For $i_2(f(I_{\mathcal{P}}))$, the claim follows from:
  \begin{align*} 
    \lambda_2 &= (-x + x^{3} + 2 x^{6} - x^{7} + (1 - x + x^{2}) y)g_1 +\\
    &\phantom{==}(1 - x^{2} + x^{3} - 2 x^{4} + x^{5} - x^{6} + x^{7} + (1 - x + x^{2} - x^{3} - x^{4} + 2 x^{5} - 2 x^{6} + 2 x^{7}) y)g_2 
\end{align*}
 \[ g_1 = (-6 - 2 x + 5 x^{2} + 8 x^{3} + 8 x^{4} - 4 x^{5} - 8 x^{6} + (-10 - 3 x + 8 x^{2} + 5 x^{3} + 2 x^{4} - 8 x^{5} - 8 x^{6} + 3 x^{7}) y)\lambda_2 \] 
 \[ g_2 = (5 + 6 x + 2 x^{2} - x^{3} + 2 x^{5} - 3 x^{6} - 5 x^{7} + (-23 + 7 x + 27 x^{2} + 6 x^{3} - x^{4} - 27 x^{5} - 17 x^{6} + 16 x^{7}) y) \lambda_2.\]
The other presentations are dealt with analogously but the details are omitted for brevity. 
\end{proof}

Recall that $g  = (i_2)_\# \colon \SF_1(\Lambda) \twoheadrightarrow \SF_1(\Lambda_{24})$. By \cref{lemma:Milnor-square-ideals}, there is an injective map
\[ \varphi : g^{-1}(\l_{24}) \to \Lambda_8^\times \backslash (\Lambda_8/3\Lambda_8)^\times \slash \Lambda_{24}^\times \]
such that, if $i_1(f(I_\mathcal{P})) = (\lambda_1) \le \l_8$, $i_2(f(I_\mathcal{P})) = (\lambda_2) \le \l_{24}$ and $j_1(\lambda_1) \in (\Lambda_8/3\Lambda_8)^\times$, then $\varphi(I_\mathcal{P}) = [j_1(\lambda_1)j_2(\lambda_2)^{-1}]$.

Let $z = x + x^{-1} \in \Lambda_8$ and note that $z^2 = 2$. By abuse of notation we denote the image of $z$ in $\Lambda_8/3\Lambda_8$ also by $z$. In particular $\F_3[z] = \F_9$.
Finally denote by $S$ the subgroup
\[ \left\langle \left(\begin{smallmatrix} 0 & 1 \\ 1 & 0 \end{smallmatrix}\right), \left(\begin{smallmatrix} z + 1 & 0 \\ 0 & 1 \end{smallmatrix}\right), \left(\begin{smallmatrix} 1 & 0 \\ 0 & z + 1 \end{smallmatrix}\right) \right\rangle = \left\{ \left(\begin{smallmatrix} a & 0 \\ 0 & b \end{smallmatrix}\right), \left(\begin{smallmatrix} 0 & a \\ b & 0 \end{smallmatrix}\right) \mid a, b \in \F_9^\times \right\} , \]
of $\GL_2(\F_9)$, which has order $2^7$. 
We now aim to prove \cref{prop:Q48-main-detailed} (ii).

\begin{lemma}\label{lem:48tomat}
  There is a well-defined function
  \[ \ol{\psi} \colon g^{-1}(\Lambda_{24}) \hookrightarrow
  \Lambda_8^\times \backslash (\Lambda_8/3\Lambda_8)^\times \slash \Lambda_{24}^\times \twoheadrightarrow \langle\F_3[\zeta_8]^\times, j \rangle \backslash (\Lambda_8/3\Lambda_8)^\times /\langle\F_3[\zeta_8]^\times, j \rangle 
  \xrightarrow[]{\cong} S \backslash \GL_2(\F_9) / S \]
 where the final function is induced by
 $\psi \colon \Lambda_8/3 \Lambda_8 \xrightarrow[]{\cong} \operatorname{M}_2(\F_9)$, $x \mapsto \left(\begin{smallmatrix}
2z + 1 & 0 \\
0 & 2z + 2 \end{smallmatrix}\right), \
y \mapsto \left(\begin{smallmatrix}
0 & 1 \\ 2 & 0 \end{smallmatrix}\right)$.
\end{lemma}

\begin{proof}
Note that $\Lambda_8$ is a free $\Z[z]$-algebra with basis $\{1, x, y, xy\}$.
As $3$ is inert in $\Z[z]$ we have that $\Lambda_8/3 \Lambda_8$ is a free $\F_9$-algebra.
Since it is a central 4-dimensional algebra over $\F_9$, it must be isomorphic to $\operatorname{M}_2(\F_9)$
and it is straightforward to check that $\psi$ is an isomorphism of $\F_9$-algebras.

By~\cite[Lemma~7.5(b)]{MOV83} we have that $\Lambda_n^\times = \Z[\zeta_n]^\times \cdot \langle j \rangle$. Hence the images of $\Lambda_8^\times$ and $\Lambda_{24}^\times$ in $(\Lambda_8/3 \Lambda_8)^\times$ are contained in $\F_3[\zeta_8]^\times = \F_3[x]^\times$. Since $z = x + x^{-1} \in \F_3[x]$, we can write $\F_3[x]^\times = \F_9[x]^\times$.
As $\F_9[x]^\times \cong \F_9^\times \times \F_9^\times$ we have $\lvert \F_3[\zeta_8]^\times \cdot \langle j \rangle \rvert = 2^7 = \lvert S \rvert$.
Since on the other hand, $\psi(\F_9[x]^\times \cdot \langle j \rangle) \subseteq S$, we must have equality and the claim follows.
\end{proof}

\begin{lemma}
  Let $\theta \in \Aut(Q_{4n})$ and let $\overline{\theta} \in \Aut((\Lambda_8/3\Lambda_8)^\times)$ be the induced automorphism.
  Then $\overline\theta(\langle\F_3[\zeta_8]^\times, j \rangle) = \langle\F_3[\zeta_8]^\times, j \rangle$.
  In particular $\overline{\theta}$ acts on 
  $\langle\F_3[\zeta_8]^\times, j \rangle \backslash (\Lambda_8/3\Lambda_8)^\times /\langle\F_3[\zeta_8]^\times, j \rangle$
  via $\overline \theta([u]) = [\overline{\theta}(u)]$, and the natural map
  \[ g^{-1}(\Lambda_{24}) \hookrightarrow
  \Lambda_8^\times \backslash (\Lambda_8/3\Lambda_8)^\times \slash \Lambda_{24}^\times \twoheadrightarrow \langle\F_3[\zeta_8]^\times, j \rangle \backslash (\Lambda_8/3\Lambda_8)^\times /\langle\F_3[\zeta_8]^\times, j \rangle\]
  is $\Aut(Q_{4n})$-equivariant.
\end{lemma}

\begin{proof}
For $\theta = \theta_{a, b}$, we have $\overline \theta(\F_3[\zeta_8]^\times) \subseteq(\F_3[\zeta_8^a]^\times) \subseteq \F_3[\zeta_8]^\times$ and
$\overline \theta(j) = \zeta_8^b j \in \langle \F_3[\zeta_8]^\times, j \rangle$.
\end{proof}

Via the isomorphism $\psi \colon \Lambda_8/3\Lambda_8 \to \operatorname{M}_2(\F_9)$ of Lemma~\ref{lem:48tomat} we transport the $\Aut(Q_{48})$-action to $S \backslash \GL_2(\F_9) / S$, so that $\overline{\psi}$ becomes $\Aut(Q_{48})$-equivariant.
Let $\mathcal{D} = (S \backslash {\GL_2(\F_9)} / S)/{\Aut(Q_{48})}$ and define
\[ q \colon g^{-1}(\Lambda_{24})/\Aut(Q_{48}) \to \mathcal{D} \] 
induced by $\overline \psi$.
We next determine the double coset space $S \backslash{\GL_{2}(\F_9)} / S$ explicitly.
For $a \in \F_9$ with $a \neq 1$, let
\[ M_a = \left(\begin{smallmatrix} 1 & a \\ 1 & 1 \end{smallmatrix}\right) \in \operatorname{M}_{2}(\F_9) \quad \text{and} \quad S_a = SM_aS \in S \backslash {\GL_2(\F_9)} / S. \]

\begin{lemma}\label{lem:q48double}
  The following holds:
  \begin{clist}{(i)}
    \item
      $\lvert S \backslash{\GL_2(\F_9)} / S \rvert = 6$ and $S \backslash{\GL_2(\F_9)} / S = \{S, S_0, S_{-1}, S_{z}, S_{1 + z}, S_{1 - z} \}$.
    \item
      Let $M = \left(\begin{smallmatrix} a & b \\ c & d \end{smallmatrix}\right) \in \GL_2(\F_9)$ with $a,d \neq 0$.
      Then either $SMS = S$ or $SMS = S_{(bc)/(ad)}$.
    \end{clist}
\end{lemma}

\begin{proof}
  Multiplication with elements of $S$ from the left (resp. right) corresponds to scaling and swapping of columns (resp. rows). Hence any double coset $SMS \neq S$ has a representative of the form $\left( \begin{smallmatrix} 1 & a \\ 1 & 1 \end{smallmatrix}\right)$ with $a \in \F_9$, $a \neq 1$. If $a$ is invertible, it is unique up to inverting.
  This implies (i) and (ii), using that $(\F_9 \setminus \{0,1\}) / (a \sim a^{-1}) = (\F_3[z] \setminus \{0,1\})/ (a \sim a^{-1}) = \{ [-1], [z], [1+z],[1-z]\}$.
\end{proof}

\begin{lemma}\label{lemma:image-psi-48}
Let $\mathcal{P} \in \{\mathcal{P}_{12}(2;4), \mathcal{P}_{12}(3; 3), \mathcal{P}_{12}(4; 4), \mathcal{P}_{12}(4;5), \mathcal{P}_{12}(6; 6), \mathcal{P}_{12}(2,2;-2,-2)\}$ and let $I_\mathcal{P}$ be an ideal representative for $\mathcal{P}$. Then:
\[ \ol{\psi}(f_\#(I_\mathcal{P})) = 
\begin{cases}
S, & \text{if $\mathcal{P} \in \{\mathcal{P}_{12}(3; 3), \mathcal{P}_{12}(4; 4), \mathcal{P}_{12}(6;6)\}$} \\
S_z, & \text{if $\mathcal{P} \in \{\mathcal{P}_{12}(2; 4), \mathcal{P}_{12}(4; 5) \}$}\\
S_{-1}, & \text{if $\mathcal{P} = \mathcal{P}_{12}(2,2;-2,-2)$.} \\
\end{cases}
\]
\end{lemma}

\begin{proof}
  If $I \subseteq \Lambda$ is an ideal such that $\Lambda_{24} I = \Lambda_{24}u$ and $\Lambda_{8}I = \Lambda_{8}v$, then the theory of Milnor squares implies that $I \in g^{-1}(\Lambda_{24})$ corresponds to the double coset $\Lambda_8^\times (vu^{-1}) \Lambda_{24}^\times$.
  As respective generators are determined \cref{lemma:l8l24gens}, it is straightforward to determine the double coset using Lemma~\ref{lem:q48double}. Consider for example $\mathcal{P} = \mathcal{P}_{12}(3; 3)$.
  Then $u = 1 + x + x^{3} - x^{7} + \left(1 - x^{4} - x^{6}\right) y$ and $v = -x + x^{2} - x^{3} + \left(1 + x^{2}\right) y$.
  Under $\psi$ the element $v u^{-1}$ is mapped to $\left(\begin{smallmatrix} 0 & \tau \\ 1 & 0\end{smallmatrix}\right)$ and hence the double coset is trivial.
  In case $\mathcal{P} = \mathcal{P}_{12}(2, 2; -2, -2)$ we have $u = 1$ and $v = -x + x^{3} + 2 x^{6} - x^{7} + (1 - x + x^{2})y$. As $\psi(v^{-1}) = (\begin{smallmatrix}
2 z + 1 & 2 z + 1 \\
z + 1 & 2 z + 2
\end{smallmatrix})$ and $(z + 1)/(2z + 2) = -1$, the claim follows from~\cref{lem:q48double}.
\end{proof}

\begin{lemma}\label{lemma:action-psi-48}
For all $\theta \in \Aut(Q_{48})$ one has $\theta(S_{-1}) = S_{-1}$ and $\theta(S_z) = S_z$.
\end{lemma}

\begin{proof}
We first consider $S_{-1}$.
Let $u = x + x y \in \Lambda_8/3 \Lambda_8$.
  Then $\psi(u) = \left(\begin{smallmatrix} 2z+1 & 2z+1 \\ z+1 & 2z+2 \end{smallmatrix}\right)$ and so $\ol{\psi}(u) = S_{-1}$.
  For $\theta = \theta_{a, b} \in \Aut(Q_{48})$ we have
  $\theta(S_{-1}) = \overline{\psi}(\theta(u)) = \overline{\psi}(x^a + x^{a + b}y)$, which is now easy to evaluate.
  For example
  $\psi(\overline{\theta}_{1, 1}(u)) = (\begin{smallmatrix} 2 z + 1 & z \\ z & 2 z + 2 \end{smallmatrix})$.
  Hence the double coset is $S_{z^2/(2z + 1)(2z + 2)} = S_{-1}$, implying that $\theta(S_{-1}) =  S_{-1}$ as claimed.
  Analogously this is done for $\{ \theta_{a, b} \mid a \in \{1, 5, 7, 11\}, b \in \{0, 1\}\}$, which are representatives for the outer automorphism group $\Out(Q_{48})$, and the first claim follows.
  The claim for $S_{z}$ follows similarly, using that for $ v = 1 + (x + x^{-1} - 1) xy$ one has $\psi(v) = (\begin{smallmatrix} 1 & -z \\ 1 & 1 \end{smallmatrix})$, in particular, $\overline{\psi}(v) = S_{-z} = S_{z}$.
\end{proof}

\begin{proof}[Proof of \cref{prop:Q48-main-detailed}]
(i) The non-freeness statement is \cref{lemma:q48-nonfree} and the freeness statement follows from \cref{lemma:l8l24gens} and the fact that $(i_2 \circ f)_\#(I_{\mathcal{P}}) \cong i_2(f(I_{\mathcal{P}}))$ by \cref{lemma:ext-ideal}.

(ii) It follows from the discussion above that there is a map $q : g^{-1}(\l_{24}) \to \mathcal{D}$ where $\mathcal{D} := (S \backslash {\GL_2(\F_9)} / S)/{\Aut(Q_{48})}$. By \cref{lemma:action-psi-48}, the elements $a := S$, $b:= S_z$ and $c:=S_{-1}$ are distinct in $\mathcal{D}$ and so the result follows from \cref{lemma:image-psi-48}.
\end{proof}

We now use \cref{prop:Q48-main-detailed} to establish the two main results of this section.

\begin{proof}[Proof of \cref{prop:Q48-main}]
By \cref{lemma:operations} and \cref{lem:symmetry_mods}, we have $\mathcal{P}_{12}(2;-2) \simeq \mathcal{P}_{12}(2;3)$ and $\mathcal{P}_{12}(4;-4) \simeq \mathcal{P}_{12}(4;6)$. It now follows from \cref{prop:Q48-main-detailed} that the presentations $\mathcal{P}_{12}^{\std}$, $\mathcal{P}_{12}(2;-2)$, $\mathcal{P}_{12}(4;-4)$ and $\mathcal{P}_{12}(2,2;-2,-2)$ are represented by ideals which are not $\Aut(Q_{48})$-isomorphic. Hence they are distinct under $\Psi$ and so are homotopically distinct.
\end{proof}

\begin{proof}[Proof of \cref{thmx:a-non-MP-presentation}]
If $\mathcal{P}_{12}(n_1;m_1)$ is a regular presentation for $Q_{48}$ then \cref{lemma:height1} implies that it is homotopy equivalent to one of 13 such presentations.
It follows from \cref{prop:Q48-main-detailed} that $\mathcal{P}_{12}(2,2;-2,-2)$ is represented by an ideal which is not $\Aut(Q_{48})$-isomorphic to an ideal represented by one of these 13 presentations.
Hence $\mathcal{P}_{12}(2,2;-2,-2)$ is distinct from each of these presentations under $\Psi$ and so is not homotopy equivalent to one of these presentations.        
\end{proof}

\subsection{Presentations for $Q_{4n}$ for arbitrary $n$} \label{ss:Q4n}

In this section, we will establish the following:

\begingroup
\renewcommand\thethm{\ref{thmx:main-Q4n}}
\begin{thm}
If $n = m k$ where $6 \le m \le 12$ and $k \ge 1$ is odd, then $Q_{4n}$ has an exotic presentation. More specifically, $\mathcal{P}_{n}(2;-2) \not \simeq \mathcal{P}_n^{\std}$.
\end{thm}
\setcounter{thm}{\value{thm}-1}
\endgroup

We will begin by using \cref{thm:MP-module} to prove the following.

\begin{prop} \label{prop:of-MP-module}
Let $n, m \ge 2$ be such that $n \mid m$ with $m/n$ odd and let $r \in \Z$ be such that $\mathcal{P}_{n}(2;r)$, $\mathcal{P}_{m}(2;r)$ are regular presentations for $Q_{4n}$, $Q_{4m}$ respectively. If $\mathcal{P}_{n}(2;r) \not \simeq \mathcal{P}_n^{\std}$, then $\mathcal{P}_{m}(2;r) \not \simeq \mathcal{P}_m^{\std}$.
\end{prop}

\begin{proof}
By \cref{thm:MP-module}, we have that $\Psi(\mathcal{P}_{m}(2;r)) = [I_{m,r}] \in \SF(\Z Q_{4m})/{\Aut(Q_{4m})}$ where $I_{m,r} = (1+(1-xy)\alpha_t,(x^r-x-1)\alpha_t) \le \Z Q_{4m}$ and $\alpha_t$ is defined where $t \ge 0$ is minimal such that $r \equiv 3t$ or $3t+1 \mod m$ (which always exists by \cref{remark:after-MP-module}). 
Next, for the same choice of $t \in \Z$, we have that $r \equiv 3t$ or $3t+1 \mod n$ and so $\Psi(\mathcal{P}_{n}(2;r)) = [I_{n,r}'] \in \SF(\Z Q_{4n})/{\Aut(Q_{4n})}$ where $I_{n,r}' = (1+(1-xy)\alpha_t',(x^r-x-1)\alpha_t') \le \Z Q_{4n}$ for $\alpha_t'$ as defined in \cref{thm:MP-module}.
Since $m/n$ is odd, there is a quotient map $Q_{4m} \twoheadrightarrow Q_{4n}$. If $f : \Z Q_{4m} \to \Z Q_{4n}$ denotes the induced ring homomorphism, then $f_\#(I_{m,r}) \cong (1+(1-xy)f(\alpha_t),(x^r-x-1)f(\alpha_t))$ by \cref{lemma:ext-ideal}.
Next note that, since the formulae for $\alpha_t$, $\alpha_t'$ given in \cref{thm:MP-module} involve only $t$ and $r$ (but not $n$ and $m$), we have that $f(\alpha_t)=\alpha_t'$. In particular, $f_\#(I_{m,r}) \cong I_{n,r}'$. 

Suppose that $\mathcal{P}_{m}(2;r) \simeq \mathcal{P}_m^{\std}$. Since $\Psi$ is injective, this implies that $[I_{m,r}] = [\Z Q_{4m}] = \Psi(\mathcal{P}_m^{\std})$. Since $(\Z Q_{4m})_\theta$ is free for all $\theta \in \Aut(Q_{4m})$, this implies that $I_{m,r}$ is a free $\Z Q_{4m}$-module. Since $f_\#(I_{m,r}) \cong I_{n,r}'$, this implies that $I_{n,r}'$ is a free $\Z Q_{4n}$-module and so $\Psi(\mathcal{P}_{n}(2;r)) = [\Z Q_{4n}] = \Psi(\mathcal{P}_n^{\std})$ and so $\mathcal{P}_{n}(2;r) \simeq \mathcal{P}_n^{\std}$ since $\Psi$ is injective.
\end{proof}

\begin{proof}[Proof of \cref{thmx:main-Q4n}]
Let $6 \le m \le 12$. By Propositions \ref{prop:Q24}, \ref{prop:Q28}, \ref{prop:Q32}, \ref{prop:Q36}, \ref{prop:Q40}, \ref{prop:Q44} and \ref{prop:Q48-main}, we have that $\mathcal{P}_m(2;-2) \not \simeq \mathcal{P}_m^{\std}$. If $n=mk$ for some $k \ge 1$ odd, \cref{prop:of-MP-module} now implies that $\mathcal{P}_n(2;r) \not \simeq \mathcal{P}_n^{\std}$ as required.   
\end{proof}

\bibliography{biblio.bib}
\bibliographystyle{amsalpha}

\end{document}